\documentclass[twoside,11pt]{article}

\usepackage{amsthm}
\usepackage[abbrvbib,preprint]{jmlr2e}
\usepackage[T1]{fontenc}
\usepackage{mathrsfs}
\usepackage{amsmath}
\usepackage{amssymb}
\usepackage{enumerate}
\usepackage{tikz}
\usetikzlibrary{patterns, arrows.meta, decorations.pathreplacing, calligraphy}
\usepackage{subcaption}
\usepackage{color,soul}

\usepackage{aliascnt}
\let\oldtheorem\newtheorem
\RenewDocumentCommand{\newtheorem}{s m o m O{}}{%
\IfBooleanTF{#1}%
{\oldtheorem{#2}{#4}}%
{\IfNoValueTF{#3}{\oldtheorem{#2}{#4}[#5]}%
{\newaliascnt{#2}{#3}%
\oldtheorem{#2}[#2]{#4}%
\aliascntresetthe{#2}}}}

\newtheorem{theorem}{Theorem}
\newtheorem{lemma}[theorem]{Lemma} 
\newtheorem{proposition}[theorem]{Proposition} 
\newtheorem{remark}[theorem]{Remark}
\newtheorem{corollary}[theorem]{Corollary}
\newtheorem{definition}[theorem]{Definition}

\usepackage[
    backend=biber,
    citestyle=numeric-comp,
    bibstyle=ieee,
    sorting=nyt,
    isbn=false,
    url=false,
    citecounter=true,
    ]{biblatex}
\AtEveryBibitem{\clearlist{language}}

\DeclareSourcemap{
  \maps[datatype=bibtex]{
    \map{
      \step[fieldsource=title,
        match=\regexp{\\\$},
        replace=\regexp{\$}]
      \step[fieldsource=title,
        match=\regexp{\{\\textasciicircum\}},
        replace=\regexp{\^}]
    \step[fieldsource=title,
        match={{\\textasciicircum}\\{p,q\\}},
        replace={\^{p,q}}]
      \step[fieldsource=title,
        match=\regexp{\\\_},
        replace=\regexp{\_}]
      \step[fieldsource=title,
        match=\regexp{\{\\textbackslash\}},
        replace=\regexp{\\}]
      \step[fieldsource=title,
        match=\regexp{\\\{\{R\}\\\}},
        replace=\regexp{\{R\}}]
    }
  }
}
\usepackage{todonotes}
\usepackage{marginnote}

\AtEveryBibitem{\clearlist{language}}

\DeclareSourcemap{
  \maps[datatype=bibtex]{
    \map{
      \step[fieldsource=title,
        match=\regexp{\\\$},
        replace=\regexp{\$}]
      \step[fieldsource=title,
        match=\regexp{\{\\textasciicircum\}},
        replace=\regexp{\^}]
      \step[fieldsource=title,
        match=\regexp{\\\_},
        replace=\regexp{\_}]
      \step[fieldsource=title,
        match=\regexp{\{\\textbackslash\}},
        replace=\regexp{\\}]
      \step[fieldsource=title,
        match=\regexp{\\\{\{R\}\\\}},
        replace=\regexp{\{R\}}]
    }
  }
}

\usepackage[capitalize]{cleveref}
\usepackage{float}

\newcommand{\rrd}{\mathbf{R}^{d}}
\newcommand{\nn}[1]{\mathbf{N}^{#1}}
\newcommand{\rr}[1]{\mathbf{R}^{#1}}

\newcommand{\cc}[1]{\mathbf{C}^{#1}}

\newcommand{\zz}[1]{\mathbf{Z}^{#1}}
\newcommand{\zzp}[1]{\mathbf{Z}_+^{#1}}

\newcommand{\rd}{\rr d}
\newcommand{\rdd}{\rr {2d}}
\newcommand{\scal}[2]{\langle #1,#2\rangle}
\newcommand{\abs}[1]{| #1|}     %%%%%   for |x|
\newcommand{\cdo}{\, \cdot \, }
\newcommand{\aff}{\operatorname{a}}
\newcommand{\bb}[1]{\mathbb{#1}}
\newcommand{\cF}{\mathscr{F}}
\newcommand{\conv}[0]{*}
\newcommand{\norm}[2]{\|#1\|_{#2}}

\newcommand{\F}[0]{\mathscr{F}}
\newcommand{\FL}[1]{\F L^{#1}}
\newcommand{\chF}[1]{\chi_{#1}}

\newcommand{\chEF}[2]{\chF{#1}^{#2}}

\numberwithin{equation}{section}

\begin{document}

\title{Time-Frequency Analysis for Neural Networks}
\author{\name Ahmed Abdeljawad\email ahmed.abdeljawad@oeaw.ac.at \\
	\addr Johann Radon Institute of Computational and Applied Mathematics (RICAM)\\
	\addr Austrian Academy of Sciences\\
	\addr Altenberger Straße 69, A-4040 Linz, Austria\\
	\AND
	\name Elena Cordero  \email elena.cordero@unito.it \\
	\addr Dipartimento di Matematica\\
	\addr Università degli Studi di Torino\\
	\addr via Carlo Alberto 10, 10123 Torino, Italy
}
\editor{}

\maketitle
\begin{abstract}
We develop a quantitative approximation theory for shallow neural networks using tools from time-frequency analysis. Working in weighted modulation spaces $M^{p,q}_m(\rr{d})$, we prove dimension-independent approximation rates in Sobolev norms $W^{n,r}(\Omega)$ for networks whose units combine standard activations with localized time-frequency windows. Our main result shows that for $f \in M^{p,q}_m(\rr{d})$ one can achieve
\[
\|f - f_N\|_{W^{n,r}(\Omega)} \lesssim N^{-1/2}\,\|f\|_{M^{p,q}_m(\rr{d})},
\]
on bounded domains, with explicit control of all constants. We further obtain global approximation theorems on $\rr{d}$ using weighted modulation dictionaries, and derive consequences for Feichtinger’s algebra,  Fourier-Lebesgue spaces, and Barron spaces. Numerical experiments in one and two dimensions confirm that modulation-based networks achieve substantially better Sobolev approximation than standard ReLU networks, consistent with the theoretical estimates.
\end{abstract}

% \vspace{.5cm}

%%%%%%%%%%%%%%%%%%%%%%%%%%%%%%%%%%%%%%%%%%%%%%%%%%%%%%%%%%%%%%%%%%
\section{Introduction}
%%%%%%%%%%%%%%%%%%%%%%%%%%%%%%%%%%%%%%%%%%%%%%%%%%%%%%%%%%%%%%%%%%

Neural networks have established themselves as a central tool in modern machine learning, driving breakthroughs in fields ranging from computer vision and natural language processing to scientific computing and control. Their empirical success is often attributed to a combination of high expressive power, scalability in high dimensions, and the availability of efficient training algorithms. At the same time, it has prompted a growing effort to understand these models from a mathematical point of view. Classical universal approximation theorems guarantee that neural networks with a single hidden layer (also known as shallow neural networks) can approximate to arbitrary accuracy a wide class of continuous functions on compact domains  \cite{Cybenko89ApproximationSuperpositionsSigmoidal}, as well as other function spaces \cite{Abdeljawad24WeightedApproximationBarron, Abdeljawad23SpaceTimeApproximationShallow, Siegel20ApproximationRatesNeural, Siegel23CharacterizationVariationSpaces, Klusowski18ApproximationCombinationsReLU}. In other words, the class of functions generated by such networks is dense in many natural function spaces.

\emph{Qualitative} expressivity results provide valuable insights into the ability of neural networks to approximate highly complex functions~\cite{Abdeljawad2025UniformApproximationQuadratic, Abdeljawad22ApproximationsDeepNeural, Caragea23NeuralNetworkApproximation, Marwah23Neuralnetworkapproximations}, including those arising as solutions to partial differential equations (PDEs)~\cite{Chen23RegularityTheoryStatic, Chaudhari2018DeepRelaxationPartial, Grohs23ProofTheorytoPracticeGap, Kutyniok2022TheoreticalAnalysisDeep, Lagaris1998ArtificialNeuralNetworks, Raissi2019PhysicsinformedNeuralNetworks, Chen2018NeuralOrdinaryDifferential, E2018DeepRitzMethod}. %These results clarify why neural networks are, in principle, expressive enough to approximate PDE solutions.

Beyond these qualitative insights, a substantial body of theoretical work has contributed to quantifying how the network complexity scales with the target accuracy, the input dimension, and the regularity of the target function~\cite{Yarotsky2017ErrorBoundsApproximations, Petersen2018OptimalApproximationPiecewise, Elbrachter2021DeepNeuralNetwork, Suzuki2019AdaptivityDeepReLU}. Nevertheless, many of the existing results are derived for specific classes of functions, architectures, or norms, and do not fully account for the structural and analytical properties typical of PDE problems. This leaves several important questions open regarding the efficiency and scalability of neural network-based solvers, particularly in relation to solution regularity, dimensionality, and architectural design.

Much of this quantitative theory, however, has been developed for standard regression or data-fitting problems, where the primary performance metrics are based on $L^p$ norms and pointwise prediction error.
Such an $L^p$-centric viewpoint is not fully aligned with the requirements of the burgeoning field of scientific computing, particularly for the numerical solution of PDEs. In this context, the approximant must faithfully capture \emph{both} the target function $f$ and its derivatives $\partial^\alpha f$ up to a given order $n \in \zzp{}$. The latter requirement naturally shifts the focus from Lebesgue-type error measurements to error measures in Sobolev norms
\[
  \|f - f_N\|_{W^{n,r}(\Omega)} 
  = \Bigg( \sum_{|\alpha|\le n} \|\partial^\alpha f - \partial^\alpha f_N\|_{L^r(\Omega)}^r \Bigg)^{1/r},
\]
for $r \ge 2$ and bounded domains $\Omega \subset \rr{d}$, which are closely aligned with the analytical structure of variational formulations.

From a theoretical perspective, one of the main obstacles in developing such quantitative approximation results is the well-known \emph{curse of dimensionality}: for generic function classes on $\rr{d}$, the number of parameters required to obtain a prescribed accuracy $\varepsilon > 0$ often scales like $\varepsilon^{-\mathcal{O}(d)}$ as $d$ grows. A productive way to circumvent this has been to restrict attention to more structured function classes. A prominent example is the \emph{Barron space} introduced in the seminal work of Barron \cite{Barron93UniversalApproximationBounds}, which characterizes functions by the finiteness of a certain spectral moment of their Fourier transform. In this setting, shallow neural networks can achieve \emph{dimension-independent} approximation rates of order $\mathcal{O}(N^{-1/2})$ in $L^2$, as refined in \cite{Yang2025OptimalRatesApproximation, DeVore2025WeightedVariationSpaces,E2022BarronSpaceFlowInduced,Siegel22SharpBoundsApproximation, Siegel24SharpBoundsApproximation, Siegel23CharacterizationVariationSpaces, Voigtlaender22SamplingNumbersFourierAnalytic}.
This explicitly links neural network training to dictionary learning and greedy approximation theory, drawing on classical results from DeVore~\cite{DeVore98NonlinearApproximation} and Cohen et al.~\cite{Cohen22OptimalStableNonlinear} regarding nonlinear approximation with redundant dictionaries. 

Our aim in this work is to extend this quantitative perspective to a phase-space framework based on \emph{modulation spaces} and to error measures in high-order Sobolev norms.% that are natural for PDE applications.

However, despite the success of Barron-type spaces, several important gaps remain:
\begin{enumerate}
  \item Most existing results are formulated in $L^2$ (or $L^p$) norms and do not directly address Sobolev norms $W^{n,r}(\Omega)$ that are more natural for PDE applications.
  \item The Fourier-only viewpoint underlying spectral Barron spaces is not well-suited to capturing functions with nontrivial \emph{time-frequency localization}, i.e., functions whose behavior is constrained in both space and frequency.
  \item Approximation results on unbounded domains $\rr{d}$ are comparatively scarce, especially in settings where both the function and its derivatives are controlled.
\end{enumerate}

These issues motivate the search for a more flexible analytical framework that can simultaneously:
(i) encode phase-space information (space and frequency),
(ii) capture decay and regularity in a unified way, and
(iii) support dimension-independent approximation estimates in high-order Sobolev norms, with explicit control of the dependence of the constants on the problem parameters.

To address these challenges, we work in the setting of \emph{modulation spaces} $M^{p,q}_m(\rr{d})$, introduced by Feichtinger \cite{Feichtinger03ModulationSpacesLocally} and treated in depth in \cite{Grochenig01FoundationsTimeFrequencyAnalysis}. Roughly speaking, modulation spaces measure the size and distribution of the \emph{short-time Fourier transform} (STFT)
\[
  V_\varphi f(x,\xi)
  = \int_{\rr{d}} f(t)\,\overline{\varphi(t-x)}\,e^{-2\pi i\, t\cdot\xi}\,\mathrm{d}t,
\]
where $\varphi$ is a fixed nonzero window function in the Schwartz class $\mathscr{S}(\rr{d})$. 
For a weight $m:\rr{d}\times\rr{d}\to(0,\infty)$ and exponents $0 < p,q \le \infty$, the modulation norm is given by
\[
  \|f\|_{M^{p,q}_m(\rr{d})}
  = \big\|  m \,V_\varphi f\big\|_{L^{p,q}(\rr{d}\times\rr{d})}.
\]
This norm imposes a specific geometric structure on the phase space. As visualized in Figure~\ref{fig:tiling}, this norm induces a uniform phase-space tiling, contrasting with the dyadic decompositions of Besov spaces. While dyadic grids widen at high scales to localize singularities, the STFT maintains constant frequency bandwidth, making it superior for capturing high-frequency oscillations.
% --------------------------
\begin{figure}[htbp]
    \centering
    \definecolor{modBlue}{RGB}{70, 130, 180}   
    \definecolor{besovOr}{RGB}{210, 105, 30}   

    % ==================================================================
    % LEFT: MODULATION SPACE (Uniform Tiling)
    % ==================================================================
    \begin{tikzpicture}[scale=0.7, >=Latex]
        % Axis limits
        \draw[->, thick] (-0.2, 0) -- (4.5, 0) node[right] {$x$};
        \draw[->, thick] (0, -0.2) -- (0, 7.5) node[above] {$\xi$};

        % Uniform Grid: 1x1 boxes everywhere
        % We draw up to y=7 to match the height of the neighbor plot
        \foreach \x in {0, 1, ..., 3} {
            \foreach \y in {0, 1, ..., 6} {
                \draw[fill=modBlue!10, draw=modBlue!40, thin] (\x, \y) rectangle ++(1, 1);
            }
        }

        % Highlight an atom
        \draw[fill=modBlue!40, draw=modBlue, thick] (1, 3) rectangle (2, 4);
        % Title
        \node[align=center] at (2,-1.25) {Modulation Space ($M^{p,q}$) \\ Uniform Decomposition \\ $\Delta \xi = \text{const}$};
    \end{tikzpicture}
    \hspace{1cm}
    %
    % ==================================================================
    % RIGHT: BESOV/SOBOLEV SPACE (Dyadic Tiling)
    % ==================================================================
    \begin{tikzpicture}[scale=0.7, >=Latex]
        % Axis limits
        \draw[->, thick] (-0.2, 0) -- (4.5, 0) node[right] {$x$};
        \draw[->, thick] (0, -0.2) -- (0, 7.5) node[above] {$\xi$};

        % --- ROW 1: [0, 1] ---
        % Height = 1, Width = 1
        \foreach \x in {0, 1, ..., 3} {
            \draw[fill=besovOr!10, draw=besovOr!40, thin] (\x, 0) rectangle ++(1, 1);
        }

        % --- ROW 2: [1, 3] ---
        % Height = 2 (Doubled), Width = 0.5 (Halved)
        % Note: y goes from 1 to 3
        \foreach \x in {0, 0.5, ..., 3.5} {
            \draw[fill=besovOr!10, draw=besovOr!40, thin] (\x, 1) rectangle ++(0.5, 2);
        }

        % --- ROW 3: [3, 7] ---
        % Height = 4 (Doubled Again), Width = 0.25 (Quartered)
        % Note: y goes from 3 to 7
        \foreach \x in {0, 0.25, ..., 3.75} {
            \draw[fill=besovOr!10, draw=besovOr!40, thin] (\x, 3) rectangle ++(0.25, 4);
        }

        % Annotations for Doubling
        \draw[decorate, decoration={calligraphic brace, amplitude=4pt}, thick, gray] (-0.2, 0) -- (-0.2, 1);
        \node[anchor=east, font=\tiny, text=gray] at (-0.3, 0.5) {$\Delta \xi=1$};

        \draw[decorate, decoration={calligraphic brace, amplitude=4pt}, thick, gray] (-0.2, 1) -- (-0.2, 3);
        \node[anchor=east, font=\tiny, text=gray] at (-0.3, 2) {$\Delta \xi=2$};

        \draw[decorate, decoration={calligraphic brace, amplitude=4pt}, thick, gray] (-0.2, 3) -- (-0.2, 7);
        \node[anchor=east, font=\tiny, text=gray] at (-0.3, 5) {$\Delta \xi=4$};

        % Highlight atoms
        \draw[fill=besovOr!40, draw=besovOr, thick] (2, 0) rectangle (3, 1); % Low freq
        \draw[fill=besovOr!40, draw=besovOr, thick] (2.5, 3) rectangle (2.75, 7); % High freq

        % Title
        \node[align=center] at (2,-1.25) {Besov Space ($B^s_{p,q}$) \\ Dyadic Decomposition \\ $\Delta \xi \sim 2^j, \Delta x \sim 2^{-j}$};
    \end{tikzpicture}
    \caption{Visualizing the tiling of the time-frequency plane. Left: Modulation spaces use a uniform grid. Right: Besov spaces use a dyadic grid where the frequency bandwidth doubles at each scale ($1 \to 2 \to 4$).}
    \label{fig:tiling}
\end{figure}
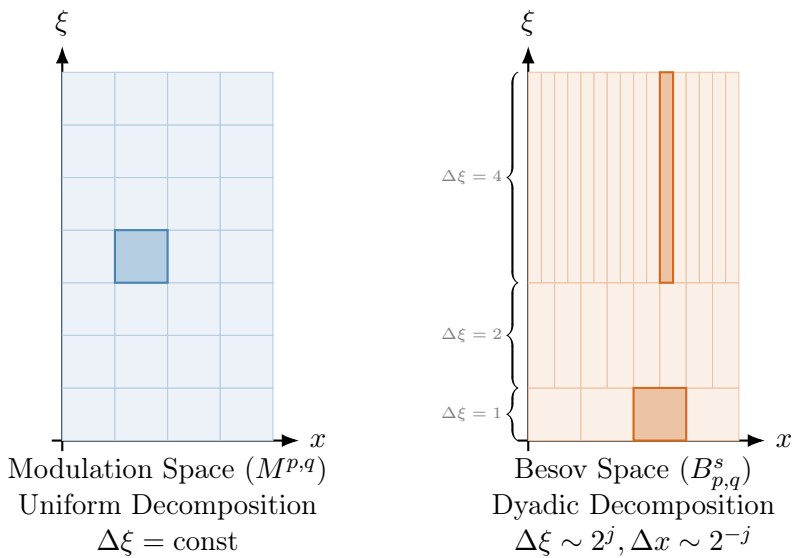
% --------------------------
Within this framework, different choices of $m$, $p$, and $q$ give rise to a rich scale of function spaces. In particular:
\begin{itemize}
  \item The \emph{Feichtinger algebra} $M^1$ is obtained for $p=q=1$ and a suitable polynomial weight, and it is closely related to spectral Barron spaces \cite{Liao23SpectralBarronSpace}.
  \item Weighted Fourier-Lebesgue spaces $\mathscr{F}L^q_{v_s}$ arise as modulation spaces with weights depending only on the frequency variable $\xi$.
  \item Various classical function spaces, including Shubin-Sobolev, Bessel potential, Besov, and Sobolev spaces, can be embedded into weighted modulation spaces via appropriate choices of weight and integrability parameters; see \cite{Benyi20ModulationSpacesApplications, Kobayashi11InclusionRelationSobolev, Guo17InclusionRelationsModulation}.
\end{itemize}

This phase-space perspective offers a unified formalism that simultaneously characterizes spatial decay, frequency decay, and regularity. From the perspective of neural network approximation, modulation spaces are particularly attractive because they admit natural atomic decompositions into localized building blocks such as Gabor atoms~\cite{FEICHTINGER1989ATOMICCHARACTERIZATIONSMODULATION}. In this work, we exploit this structure by introducing a dictionary $\bb{D}$ of \emph{windowed activation functions}, consisting of terms of the form
\begin{equation}\label{eq:act_dic_intro}
      x \mapsto \sigma\Big(\tfrac{\eta\cdot x}{\tau}+b\Big)
  \,\varphi\Big(\tfrac{\eta\cdot x}{\tau}+b-t\Big)\,
  \phi(x-y),
\end{equation}
where $\sigma$ is a standard activation function e.g., $ReLU$, $\varphi\in\mathscr{S}(\rr{})$ and $\phi\in\mathscr{S}(\rr{d})$ are window functions, and $(y,\eta,b)$ parameterize the spatial, frequency, and bias components. This construction retains the flexibility of neural activations while introducing explicit phase-space localization.

%%%%%%%%%%%%%%%%%%%%%%%%%%%%%%%%%%%%%%%%%%%%%%%%%%%%%%%%%%%%%%%%%%
\subsection{Main Contributions}
%%%%%%%%%%%%%%%%%%%%%%%%%%%%%%%%%%%%%%%%%%%%%%%%%%%%%%%%%%%%%%%%%%

We develop a unified approximation theory for shallow neural networks acting on weighted modulation spaces and measured in high-order Sobolev norms. Throughout, $d\in\nn{}$ denotes the ambient dimension, $\Omega\subset\rr{d}$ is a bounded domain. The error is measured in a Sobolev norm $W^{n,r}(\Omega)$ with exponent $r\ge 2$ and regularity of order $n\in\zzp{}$.

%%%%%%%%%%%%%%%%%%%%%%%%%%%%%%%%%%%%%%%%%%%%%%%%%%%%%%%%%%%%%%%%%%
\paragraph{1. Local Sobolev Approximation in Modulation Spaces.}
Our first main result (Theorem~\ref{thm:approximation_sobolev_space_local}) shows that for any
\[
  f \in M^{p,q}_m(\rr{d}), 
  \qquad 0 < p < \infty, \quad 0 < q \le 2 \le r,
\]
with a weight $m(x, \xi) =  (1 + |x|^2)^{s_1/2}  (1 + |\xi|^2)^{s_2/2}$ satisfying suitable conditions on $s_1$ and $s_2$, % for all $x, \xi \in \rr{d}$,
there exists a constant $C>0$ such that
\[
  \inf
  \| f - f_N \|_{W^{n,r}(\Omega)}
  \le C\, N^{-1/2}\, |\Omega|^{1/r}\, \|f\|_{M^{p,q}_m(\rr{d})},
\]
for all $N\in\nn{}$, where the infimum is taken over all shallow networks \( f_N \) with \( N \) neurons whose activation functions are of the form given in \eqref{eq:act_dic_intro}; see \Cref{sec:prelim:VariationSpace} for details on the structure of such networks. The resulting approximation rate is dimension-independent, and the proof yields explicit control of the constant $C$.

%%%%%%%%%%%%%%%%%%%%%%%%%%%%%%%%%%%%%%%%%%%%%%%%%%%%%%%%%%%%%%%%%%
\paragraph{2. Unified Consequences for Feichtinger, Shubin, and Fourier-Lebesgue Spaces.}
Specializing the weight and exponents yields a series of concrete corollaries. For $p=q=1$, Theorem~\ref{thm:approximation_sobolev_space_local} recovers a local Sobolev approximation result for the weighted Feichtinger algebra $M^1_m$ (Corollary~\ref{cor:Feichtinger_approximation_sobolev_space_local}). Furthermore, we obtain local Sobolev approximation bounds in Shubin-Sobolev spaces $Q^{s}$ and in classical weighted spaces $L^2_{v_{s}}$ and $\mathscr{F}L^2_{v_{s}}$ for suitable choices of $s$ (Corollary~\ref{cor:Hilbert_approximation_sobolev_space_local}), which can be viewed as a quantitative formulation of the uncertainty principle.
Using the local equivalence between modulation and weighted Fourier-Lebesgue spaces, we further obtain a local approximation result in $\mathscr{F}L^q_{v_{s}}$ (Proposition~\ref{prop:supported_approximation_sobolev_space_local}).%, with explicit dependence on $|\Omega|$ (or $|\Omega+\Omega|$).

%%%%%%%%%%%%%%%%%%%%%%%%%%%%%%%%%%%%%%%%%%%%%%%%%%%%%%%%%%%%%%%%%%
\paragraph{3. Sobolev Approximation in Barron Spaces.}
A particularly important case for the machine-learning community is that of Barron spaces. For $p=1$ and an appropriate frequency weight, Corollary~\ref{cor:Barron} yields a Barron-space approximation result of the form
\[
  \inf
  \| f - f_N \|_{W^{n,r}(\Omega)}
  \le C\, N^{-1/2}\, |\Omega+\Omega|\, \|f\|_{B_{v_{n+1}}},
\]
 with a simplified bound when $\Omega$ is convex, where the infimum is taken over all shallow networks \( f_N \) with \( N \) neurons activated by functions of the form given in \eqref{eq:act_dic_intro}. This extends the $H^n(\Omega)$-based results of Siegel and Xu \cite{Siegel20ApproximationRatesNeural} to general Sobolev norms $W^{n,r}(\Omega)$ and arbitrary dimension, establishing a natural connection in the phase-space framework.

%%%%%%%%%%%%%%%%%%%%%%%%%%%%%%%%%%%%%%%%%%%%%%%%%%%%%%%%%%%%%%%%%%
\paragraph{4. Global Approximation on $\rr{d}$.}
Local results do not immediately extend to unbounded domains. Our second main theorem (Theorem~\ref{thm:approximation_sobolev_space_G}) addresses this by considering a modified dictionary $\bb{D}_\Omega$ where the spatial shifts $y$ are restricted to a fixed bounded set $\Omega\subset\rr{d}$. We show that for all $f \in M^{p,q}_m(\rr{d})$ with $0<p,q<\infty$ and suitable $m$, one still has the global bound
\[
  \inf
  \| f - f_N \|_{W^{n,r}(\rr{d})}
  \le C\, N^{-1/2}\, \|f\|_{M^{p,q}_m(\rr{d})},
\]
for all $N\in\nn{}$, where the infimum is taken over all shallow networks \( f_N \) with \( N \) neurons activated by functions of the form given in \eqref{eq:act_dic_intro} such that the  spatial shifts $y$ are restricted to a fixed bounded set $\Omega\subset\rr{d}$. As a corollary, we obtain global Sobolev approximation results for the weighted Feichtinger algebra and, via embeddings, for Bessel potential spaces $W^{r,t}(\rr{d})$. We emphasize that our results significantly generalize the findings in \cite{Parhi2023ModulationSpacesCurse}.

%%%%%%%%%%%%%%%%%%%%%%%%%%%%%%%%%%%%%%%%%%%%%%%%%%%%%%%%%%%%%%%%%%
\paragraph{5. Numerical Validation via Modulation Neural Networks.}
Finally, we complement our theoretical analysis with numerical experiments based on a \emph{Modulation Neural Network} architecture that is directly inspired by the dictionary $\bb{D}$ in Theorem~\ref{thm:approximation_sobolev_space_local}. In this architecture, the network units implement windowed activation functions of the form used in our approximation results. Through extensive experiments in one and two spatial dimensions, we observe that:
\begin{enumerate}[(i)]
    \item modulation networks consistently outperform standard shallow ReLU networks of comparable (or even larger) parameter counts when the error is measured in Sobolev norm;
    \item the windowed structure yields markedly better localization, leading to significantly improved approximation of derivatives compared to vanilla architectures;
    \item the proposed architecture exhibits faster convergence during training (for both Adam and AdamW optimizers) and higher expressivity per parameter, providing empirical support for the efficiency suggested by our theoretical bounds.
\end{enumerate}
In two-dimensional test problems, the loss-vs-epochs plots in \cref{fig:2d_prediction_expressivity} indicate that the modulation network achieves an empirical decay rate in the $H^1$ error that is steeper than a Monte Carlo-type $N^{-1/2}$ baseline. This suggests that the classical Monte Carlo rate may not be sharp for this architecture and function class, and it naturally raises the open question of what the optimal approximation rate should be in this phase-space-informed setting. Taken together, these experiments show that our phase-space-guided architectural design is not merely of theoretical interest: it leads to tangible improvements in accuracy and convergence in learning tasks arising from PDE settings.

%%%%%%%%%%%%%%%%%%%%%%%%%%%%%%%%%%%%%%%%%%%%%%%%%%%%%%%%%%%%%%%%%%
\subsection{Organization of Paper}
%%%%%%%%%%%%%%%%%%%%%%%%%%%%%%%%%%%%%%%%%%%%%%%%%%%%%%%%%%%%%%%%%%

The remainder of this article is organized as follows. In Section~\ref{sec:preliminaries}, we introduce the necessary functional analytic background, including the definition and properties of the STFT and the weighted modulation spaces $M^{p,q}_m$. Section~\ref{subsec:ConvergenceInBochnerSobolevNorms} establishes key embedding results between modulation and Sobolev spaces. The main theoretical contributions are presented in Section~\ref{sec:M_space_approximation}, where we derive approximation rates for shallow neural networks first on bounded domains \cref{thm:approximation_sobolev_space_local} and subsequently on unbounded domains \cref{thm:approximation_sobolev_space_G}. We discuss specific implications for the Feichtinger algebra, Shubin--Sobolev spaces, Barron spaces, and Bessel Potential spaces within this section. Finally, 
Section~\ref{sec:experiments} presents numerical experiments that illustrate the computational efficacy of our approach, demonstrating the superior performance of the proposed windowed architecture compared to standard neural networks in various approximation tasks.

%%%%%%%%%%%%%%%%%%%%%%%%%%%%%%%%%%%%%%%%%%%%%%%%%%%%%%%%%%%%%%%%%%
\section{Preliminaries}\label{sec:preliminaries}
%%%%%%%%%%%%%%%%%%%%%%%%%%%%%%%%%%%%%%%%%%%%%%%%%%%%%%%%%%%%%%%%%%

In what follows we recall the basic definitions and properties we shall use in the current paper. Main subject is the introduction of the \emph{short-time Fourier transform} (STFT) and its use to define the related modulation spaces.\par 
\textbf{Notations}.
We denote by $d\in\nn{}$ the dimension of the space. The space   $\mathscr{S}(\rd)$ is the Schwartz class of smooth rapidly
decreasing functions and $\mathscr{S}'(\rd)$ its dual (the space of tempered distributions). The class $\mathcal{C}_c^\infty(\rd)$ is  the space of compactly supported and smooth functions.

The brackets  $(f,g)$ means the extension to $\mathscr{S}' (\rd)\times\mathscr{S}(\rd)$ of the inner product $( f,g)=\int f(t){\overline {g(t)}}dt$ on $L^2(\rd)$ (conjugate-linear in the second component).

We denote the
Fourier transform and its inverse by
$$
\mathscr Ff(\xi )= \widehat f(\xi ) = \int _{\rr
{d}} f(x)e^{-2\pi i\scal  x\xi } dx,\quad \mathscr F^{-1}f(\xi )= \check f(\xi ) =\int _{\rr
{d}} f(x)e^{2\pi i\scal  x\xi } dx,
$$
where  $f\in\mathscr{S}(\rd)$ and $\scal \cdo \cdo$ denotes the standard inner product
on $\rr d$. 
The map $\mathscr F$ extends
uniquely to a homeomorphism on 
${\mathscr S} '(\rr d)$,
to a unitary operator on $L^2(\rr d)$ and restricts
to  a homeomorphism on the Schwartz space $\mathscr S(\rr d)$.
With this normalization, the Fourier
transform satisfies the classical
convolution relations:
\begin{equation*}%\label{eq:FourTransfConv}
{\mathscr F} (f\cdot g)
=
\widehat f *\widehat g
\quad \text{and}\quad
{\mathscr F} (f* g)
=
\widehat f \cdot \widehat g
\end{equation*}
for all $f,g\in {\mathscr S} (\rr d)$.

%%%%%%%%%%%%%%%%%%%%%%%%%%%%%%%%%%%%%%%%%%%%%%%%%%%%%%%%%%%%%%%%%%
\subsection{The Short-Time Fourier Transform}\label{subsec1.1}
%%%%%%%%%%%%%%%%%%%%%%%%%%%%%%%%%%%%%%%%%%%%%%%%%%%%%%%%%%%%%%%%%%

In signal analysis and time-frequency methods, it is often insufficient to analyze a signal solely in either the time or frequency domain. To capture how frequency content evolves over time, one employs the STFT. Unlike the classical Fourier transform, which offers a global frequency representation, the STFT introduces a windowing function to localize the signal temporally before applying the Fourier transform. This results in a two-variable function capturing both time and frequency behavior simultaneously.  If we introduce the translation $T_x$ and modulation $M_\omega$ operators, namely
$$
    T_x f(t)=f(x-t),\quad M_\omega f(t)=e^{2\pi i\omega\cdot t}f(t),
$$
the STFT of a signal $f\in L^2(\rr d)$ with respect to a non-zero window $g\in L^2(\rr d)$  is given by

\begin{equation}\label{eq:STFTdef}
    (V_g f)(x,\omega )=
(f, M_\omega T_x g )_{L^2}= {\mathscr F} (f\cdot T_x\overline {g })(\omega )=\int_{\rr d} f(y) \, \overline{g(y - x)} \, e^{-2\pi i  y\cdot\omega} \, dy
\end{equation}
The definition is extended to $(f,\phi)\in\mathscr{S}'(\rrd)\times\mathscr{S}(\rrd)$ ,  see \cite[Chapter 2]{Cordero20TimeFrequencyAnalysisOperators}
for the properties of the STFT. 

%%%%%%%%%%%%%%%%%%%%%%%%%%%%%%%%%%%%%%%%%%%%%%%%%%%%%%%%%%%%%%%%%%
\subsection{Function Spaces}\label{sec:modulation_spaces}
%%%%%%%%%%%%%%%%%%%%%%%%%%%%%%%%%%%%%%%%%%%%%%%%%%%%%%%%%%%%%%%%%%

In this section we collect the definitions and basic properties of the function spaces used throughout our analysis. We recall weighted Fourier–Lebesgue spaces, Barron spaces, modulation spaces and their embeddings, and classical Sobolev spaces. These spaces provide the analytic framework for our approximation results.
Note that, many of the function spaces considered below are defined with respect to
weight functions.
To streamline the presentation, we first introduce the class of
weights that will be used throughout this section.

\textbf{Weight Functions}.
Let $v$ be a continuous, positive, and submultiplicative weight function on $\rr d$, that is,
\[
v(z_1 + z_2) \leq v(z_1) v(z_2), \quad \text{for all } z_1, z_2 \in \rr d.
\]
A function $m$ belongs to the class $\mathcal{M}_v(\rr d)$ if it is positive, continuous, and satisfies the $v$-moderateness condition:
\[
m(z_1 + z_2) \leq C v(z_1) m(z_2), \quad \forall z_1, z_2 \in \rr d,
\]
for some constant $C > 0$.

We will focus on polynomial-type weights on $\rr {n}$, $n=d$ or $n=2d$, given by
\begin{equation}\label{weightvs}
v_s(z) = \langle z \rangle^s, \quad z \in \rr {n},
\end{equation}
where
\begin{equation*}%\label{weightv}
\langle z \rangle = (1 + |z|^2)^{1/2},
\end{equation*}
and their tensor products on $\rr {2d}$:
\begin{equation*}%\label{weight1tensorvs}
(v_s \otimes 1)(x, \xi) = (1 + |x|^2)^{s/2}, \quad (1 \otimes v_s)(x, \xi) = (1 + |\xi|^2)^{s/2}, \quad x, \xi \in \rr {d}.
\end{equation*}
Note that for $s < 0$, the function $v_s$ is $v_{|s|}$-moderate.

Given two weights $m_1$ and $m_2$ on $\rr d$, their tensor product is defined as
\[
(m_1 \otimes m_2)(x, \xi) = m_1(x) m_2(\xi), \quad x, \xi \in \rr d,
\]
and similarly when $m_1, m_2$ are defined on $\rr {2d}$.

%%%%%%%%%%%%%%%%%%%%%%%%%%%%%%%%%%%%%%%%%%%%%%%%%%%%%%%%%%%%%%%%%%
\subsubsection{Weighted Lebesgue and Fourier-Lebesgue spaces}
%%%%%%%%%%%%%%%%%%%%%%%%%%%%%%%%%%%%%%%%%%%%%%%%%%%%%%%%%%%%%%%%%%

Let $0<p\leq\infty$ and let $m:\rr{d}\to(0,\infty)$ be a weight function.
The weighted Lebesgue space $L^p_m(\rr{d})$ consists of all measurable
functions $f:\rr{d}\to\cc{}$ such that the following (quasi-)norm
\[
\|f\|_{L^p_m(\rr{d})}
:=\begin{cases}
\displaystyle
\left(\int_{\rr{d}} |f(x)|^p m(x)^p\,dx\right)^{1/p}, & 0<p<\infty,\\[1em]
\displaystyle
\operatorname*{ess\,sup}_{x\in\rr{d}} |f(x)|\,m(x), & p=\infty,
\end{cases}
\]
is finite.

Similarly, for  $0<p,q\leq\infty$,  and $F:\rdd\to\cc{}$ measurable, we set 
\[
\|f\|_{L^{p,q}_m\rr{2d}}:=\left(\int_{\rd}\left(\int_{\rd}|F(x,y)|^pm(x,y)^p dx\right)^{\frac{q}{p}}dy\right)^{\frac{1}{q}},
\]
where $m$ is a weight function on $\rdd$.

The weighted Fourier-Lebesgue spaces $\mathcal{F}L^p_s(\rr{d})$ are defined in terms of the weighted integrability of the Fourier transform (see \cite{Pilipovic10MicrolocalanalysisFourier, Katznelson2004IntroductionHarmonicAnalysis}).

\begin{definition}[Weighted Fourier-Lebesgue Spaces]
Let $0<p \le \infty$ and $s \in \rr{}$.  
The \emph{weighted Fourier-Lebesgue space} $\mathcal{F}L^p_s(\rr{d})$ is defined by
\begin{equation}\label{eq:FLs}
\mathcal{F}L^p_s(\rr{d})
= \left\{ f \in \mathscr{S}'(\rr{d}) : 
\| f \|_{\mathcal{F}L^p_s} :=\|v_s\hat f\|_{L^p(\rr{d})} < \infty \right\},
\end{equation}
where $v_s$ is defined in \eqref{weightvs}.
\end{definition}

%%%%%%%%%%%%%%%%%%%%%%%%%%%%%%%%%%%%%%%%%%%%%%%%%%%%%%%%%%%%%%%%%%
\subsubsection{Barron spaces}
%%%%%%%%%%%%%%%%%%%%%%%%%%%%%%%%%%%%%%%%%%%%%%%%%%%%%%%%%%%%%%%%%%

Barron spaces, introduced in the seminal works of Barron \cite{Barron93UniversalApproximationBounds}, and further developed e.g., in \cite{E2022BarronSpaceFlowInduced, Voigtlaender22SamplingNumbersFourierAnalytic, Abdeljawad23SpaceTimeApproximationShallow}, provide a Fourier-analytic framework for functions efficiently approximated by shallow neural networks.

\begin{definition}[Barron Norm and Barron Space]

For $s \in \rr{}$, we define the \emph{Barron space} as
\begin{equation*}
B_s(\rr{d}) = \left\{ f \in \mathscr{S}'(\rr{d}) : \| f \|_{B_s} < \infty \right\},
\end{equation*}
where the \emph{Barron norm} of $f$ is defined as
\begin{equation*}
%\label{eq:barron-norm}
\| f \|_{B_s} = \int_{\rr{d}} (1 + |\xi|)^s \, |\widehat{f}(\xi)| \, d\xi.
\end{equation*}
\end{definition}

Putting $s=1$ in \eqref{eq:FLs}, we obtain $$\cF L^1_{v_s}(\rr{})=\{f\in  \mathscr{S}': \|f\|_{\cF L^1_{v_s}}:=\|\hat{f}v_s\|_{L^1}<\infty\}.$$
Since $(1+|\xi|)^s\asymp v_s(\xi)$, $s\in\rr{},$
see, e.g., \cite{Cordero20TimeFrequencyAnalysisOperators,Grochenig01FoundationsTimeFrequencyAnalysis},
we infer that 
\begin{equation}\label{eq:Barron}
   \|f\|_{\cF L^1_{v_s}}\asymp \| f \|_{B_s},
\end{equation}
so that we have the equality of the normed spaces:
\begin{equation}\label{eq:Barron-spaces}
 B_s(\rr{d})=\cF L^1_{v_s}(\rr{d}),\quad\forall s\in\rr{}.
\end{equation}

%%%%%%%%%%%%%%%%%%%%%%%%%%%%%%%%%%%%%%%%%%%%%%%%%%%%%%%%%%%%%%%%%%
\subsubsection{Modulation spaces}
%%%%%%%%%%%%%%%%%%%%%%%%%%%%%%%%%%%%%%%%%%%%%%%%%%%%%%%%%%%%%%%%%%

Modulation spaces, originally introduced by Feichtinger in \cite{Feichtinger03ModulationSpacesLocally}, and further developed in works such as \cite{Galperin04TimefrequencyAnalysisModulation}, are now a standard topic in time-frequency analysis, with detailed treatments found in \cite{Benyi20ModulationSpacesApplications,Cordero20TimeFrequencyAnalysisOperators,Grochenig01FoundationsTimeFrequencyAnalysis,Grochenig00NonlinearApproximationLocal, Feichtinger92WilsonBasesModulation, Abdeljawad2020LiftingsUltramodulationSpaces}.

Let $g \in \mathscr{S}(\rr d)$ be a nonzero window function, $m \in \mathcal{M}_v$, and $0< p, q \leq \infty$. The modulation space $M^{p,q}_m(\rr d)$ consists of all tempered distributions $f \in \mathscr{S}'(\rr d)$ such that
\begin{equation*}%\label{norm-mod}
\|f\|_{M^{p,q}_m} = \|V_g f\|_{L^{p,q}_m} = \left( \int_{\rr d} \left( \int_{\rr d} |V_g f(x, \omega)|^p m(x, \omega)^p \, dx \right)^{q/p} d\omega \right)^{1/q} < \infty,
\end{equation*}
with the usual conventions when $p = \infty$ or $q = \infty$. The STFT $V_g f$ is defined as in \eqref{eq:STFTdef}. We also use the simplified notation $M^p_m(\rr d)$ for $M^{p,p}_m(\rr d)$ and $M^{p,q}(\rr d)$ when $m \equiv 1$.

The space $M^{p,q}(\rr d)$ is a Banach space whenever $p,q\geq1$ and a quasi-Banach one in the other cases. Its (quasi-)norm does not depend (up to equivalence) on the specific choice of the window function $g$, provided $g \neq 0$. The class of admissible windows can be enlarged to include all functions of $M^1_v(\rr d)$, also known as the Feichtinger algebra. In particular, $M^{\infty,1}(\rr d)$ is referred to as Sjöstrand's class \cite{Sjostrand94AlgebraPseudodifferentialOperators}.\par
\emph{Duality}. If $p, q < \infty$, then
\[
\big( M^{p,q}_m(\rr{d}) \big)' \cong M^{p',q'}_{1/m}(\rr{d}),
\]
where
\begin{equation*}%\label{eq:index}
p' :=
\begin{cases}
\infty, & 0 < p \leq 1, \\[6pt]
\frac{p}{p-1}, & 1 < p < \infty,
\end{cases}
\qquad
q' :=
\begin{cases}
\infty, & 0 < q \leq 1, \\[6pt]
\frac{q}{q-1}, & 1 < q < \infty.
\end{cases}
\end{equation*}

Modulation spaces satisfy the following inclusion chain:
if $0<p_1\leq p_2\leq\infty$, $0<q_1\leq q_2\leq\infty$ and $m_1,m_2$ weights in $\rdd$ which satisfy $m_2\lesssim m_1$, then
\begin{equation}\label{eq:inclmod}
    \mathscr{S}(\rd)\hookrightarrow M^{p_1,q_1}_{m_1}(\rd)\hookrightarrow M^{p_2,q_2}_{m_2}(\rd)\hookrightarrow\mathscr{S}'(\rd).
\end{equation}
The closure of $\mathscr{S}(\rr d)$ in the $M^{p,q}_m$ norm is denoted by $\mathcal{M}_m^{p,q}(\rr d)$ and satisfies
\begin{equation}\label{eq:density_mod}
    \mathcal{M}_m^{p,q}(\rr d) \subseteq M^{p,q}_m(\rr d), \quad \text{and} \quad \mathcal{M}_m^{p,q}(\rr d) = M^{p,q}_m(\rr d)
\end{equation}
whenever $p < \infty$ and $q < \infty$. Inclusion relations for modulation spaces were refined in the following recent contribution (see also \cite[Theorem 2.22]{Bastianoni21SubexponentialDecayRegularity}), which is convenient for our purposes,
and which will be used in Section \ref{sec:M_space_approximation}.
\begin{theorem}[{\cite[Theorem 4.11]{Guo18SharpWeightedConvolution}}] \label{thm:embeddings}
Let $0 < p_j, q_j \leq \infty$, $s_j, t_j \in \rr {}$, for $j=1,2$, and consider the 
polynomial weights $v_{t_j}, v_{s_j}$ defined as in \eqref{weightvs}. Then
\[
M^{p_1,q_1}_{v_{t_1}\otimes v_{s_1}}(\rr d) \hookrightarrow 
M^{p_2,q_2}_{v_{t_2}\otimes v_{s_2}}(\rr d)
\]
if the following two conditions hold:

\begin{enumerate}[(i)]
    \item $(p_1,p_2,t_1,t_2)$ satisfies one of the following:
    \begin{description}
        \item[(C1)] $\dfrac{1}{p_2} \leq \dfrac{1}{p_1}$, \quad $t_2 \leq t_1$,
        \item[(C2)] $\dfrac{1}{p_2} > \dfrac{1}{p_1}$, \quad 
        $\dfrac{1}{p_2} + \dfrac{t_2}{d} < \dfrac{1}{p_1} + \dfrac{t_1}{d}$;
    \end{description}
    
    \item $(q_1,q_2,s_1,s_2)$ satisfies either \textbf{\rm (C1)} or \textbf{\rm (C2)}  with 
    $p_j$ replaced by $q_j$ and $t_j$ replaced by $s_j$, respectively.
\end{enumerate}
\end{theorem}

%%%%%%%%%%%%%%%%%%%%%%%%%%%%%%%%%%%%%%%%%%%%%%%%%%%%%%%%%%%%%%%%%%
\paragraph{Embedding Between Barron and Modulation Spaces.}
In the sequel we shall use the inclusion of the weighted Feichtinger algebra in the Barron space as follows:
\begin{lemma}\label{lem:Barron-Feichtinger-alg}
For any $s\in\rr{}$, we have 
\begin{equation*}%\label{eq:emb-barron}
M^1_{1\otimes v_s}(\rr{d})\hookrightarrow B_s(\rr{d}),
\end{equation*}
with continuous inclusion.
\end{lemma}
\begin{proof}
 We use    the equality \eqref{eq:Barron} and the properties of the weighted Feichtinger algebra \cite{Feichtinger03ModulationSpacesLocally}:
 \begin{equation*}%\label{eq:inclFB}
 M^1_{1\otimes v_s}(\rr{d})\hookrightarrow (L^1\cap\cF L^1_{v_s})(\rr{d})\hookrightarrow B_s(\rr{d}).
 \end{equation*}
 This concludes the proof.
\end{proof}

%%%%%%%%%%%%%%%%%%%%%%%%%%%%%%%%%%%%%%%%%%%%%%%%%%%%%%%%%%%%%%%%%%
\subsubsection{Potential Sobolev Spaces \texorpdfstring{\( W^{s,r}(\rr d) \)}{W r s(Rd)}.}\label{subsec:sobolev}
%%%%%%%%%%%%%%%%%%%%%%%%%%%%%%%%%%%%%%%%%%%%%%%%%%%%%%%%%%%%%%%%%%

Let \( s \in \rr{} \) and \( 1 \leq p \leq \infty \). The Sobolev space \( W^{s,r}(\rr d) \) is defined as the set of all tempered distributions \( f \in \mathscr{S}'(\rr d) \) such that
\[
\|f\|_{W^{s,r}} := \left\| \mathscr{F}^{-1} \left(  v_s \hat{f}\right) \right\|_{L^r(\rr d)} < \infty,
\]
where $v_s$ is defined in  \eqref{weightvs}.
Equivalently, we can write
\[
W^{s,r}(\rr d) = \left\{ f \in \mathscr{S}'(\rr d) : \langle D \rangle^s f \in L^r(\rr d) \right\},
\]
where the Bessel potential operator \( \langle D \rangle^s \) is defined by
\[
\langle D \rangle^s f := \mathscr{F}^{-1} \left(v_s \hat{f} \right).
\]

If $s=n\in\zzp{}$, then spaces above coincide with those defined by derivatives.
Note that the Fourier-Lebesgue space is the Fourier image of the Bessel
potential space (see \cite{Pilipovic10MicrolocalanalysisFourier}).
Furthermore, the inclusion relations between Sobolev and modulation spaces were proved by Toft, see  Proposition 2.9. in  \cite{Toft04ContinuityPropertiesModulation}.
\begin{proposition}\label{prop:embedding_in_fractional_Sobolev_spaces}
    Assume that \( s\in \rr{} \), $1\leq p,q,r\leq\infty$.  If \( q\leq p\le r \le q'  \),  then 
    \begin{equation*}%\label{Toft-inc}
         M^{p,q}_{1\otimes v_s}(\rr d)  \hookrightarrow W^{s,r}(\rr d),
    \end{equation*}
    with continuous inclusion.
\end{proposition}

%%%%%%%%%%%%%%%%%%%%%%%%%%%%%%%%%%%%%%%%%%%%%%%%%%%%%%%%%%%%%%%%%%
\subsubsection{Shubin–Sobolev spaces } 
%%%%%%%%%%%%%%%%%%%%%%%%%%%%%%%%%%%%%%%%%%%%%%%%%%%%%%%%%%%%%%%%%%

The Shubin–Sobolev spaces admit a characterization in terms of localization operators with Gaussian windows (cf. \cite{Shubin01PseudodifferentialOperatorsSpectral}), commonly known as anti-Wick operators. Namely, define \(\varphi(t) = 2^{d/4} e^{-\pi t^2},\,t\in\rd\). Given a function or distribution \(a\) on \(\rr {2d}\), we define the \emph{anti-Wick operator} \(A_a^{\varphi, \varphi}\) by the (formal) integral
\[
A_a^{\varphi, \varphi} f := \int_{\rr {2d}} a(x, \omega) \, V_{\varphi}f(x, \omega) \, M_\omega T_x \varphi \, dx \, d\omega,
\]
where \(V_{\varphi}f\) denotes the STFT of \(f\) with respect to \(\varphi\), \(T_x\) is the translation operator, and \(M_\omega\) is the modulation operator.
Set \(a(z) = \langle z \rangle^s\) for \(s \in \rr{}\), and define \(A_s := A^{\varphi,\varphi}_a\). Then \emph{Shubin–Sobolev space} \(Q^s\) for \(s \in \rr{}\) is defined by
\[
Q^s(\rr d) := \left\{ f \in \mathscr{S}'(\rr d) : A_s f \in L^2(\rr d) \right\} = A_s^{-1} L^2(\rr d),
\]
with norm
\[
\|u\|_{Q^s} := \|A_s u\|_{L^2}.
\]
It was proved in \cite[Lemma 2.3 ]{Boggiatto04GeneralizedAntiWickOperators} (see also \cite{Cordero03TimeFrequencyanalysis}) the following characterization via modulation spaces:
\begin{lemma}[Characterization of Shubin–Sobolev Spaces]\label{lem:shubin}
For all \(s \in \rr{}\), we have
\[
M^2_{v_s}(\rr d) = L^2_s(\rr d)\cap \cF L^2_s(\rr d)= Q^s(\rr d)
\]
with equivalent norms.
\end{lemma}

%%%%%%%%%%%%%%%%%%%%%%%%%%%%%%%%%%%%%%%%%%%%%%%%%%%%%%%%%%%%%%%%%%
 \paragraph{The Adjoint of the Short-Time Fourier Transform.} Fix $\gamma\in L^2(\rrd)$ , the STFT $V_\gamma: L^2(\rrd)\to L^2(\rr{2d})$ has adjoint  $V_\gamma^*$  given by
 $$ V_\gamma^   \ast F=\int_{\rr{2d}} F(x,\omega) M_\omega T_x \gamma\,dxd\omega.
 $$
The operator $V_\gamma^   \ast$ is a bounded operator from $L^2(\rr{2d})$ onto $L^2(\rrd)$. For $F=V_g f$, with $g,\gamma\in L^2(\rrd)$, $(g,\gamma)\not=0$,  the  inversion formula is given by
 \begin{equation*}%\label{invadj}
 f=\frac1{( \gamma,g)}V_\gamma^* V_gf,\quad f\in L^2(\rrd).
 \end{equation*}
\begin{theorem}\label{C2invfmod}
    Consider $m\in\mathcal{M}_{v}$ and $g,\gamma\in M^1_v\left(\rr {d}\right)$. Then for $1\leq p,q\leq \infty$,
    \begin{enumerate}[(i)]
        \item  $V_{\gamma}^{*}\,:\,L_{m}^{p,q}\left(\rr {2d}\right)\rightarrow M_{m}^{p,q}\left(\rr {d}\right)$
			and the following estimate holds
			\begin{equation*}%\label{eq:stima inversione}
			\left\Vert V_{\gamma}^{*}F\right\Vert _{M_{m}^{p,q}}=\left\Vert V_{g}\left(V_{\gamma}^{*}F\right)\right\Vert _{L_{m}^{p,q}}\lesssim\left\Vert V_{g}\gamma\right\Vert _{L_{v}^{1}}\left\Vert F\right\Vert _{L_{m}^{p,q}}.
			\end{equation*}
		\item If $F=V_{g}f$ and $(\gamma,g)\not=0$, we have the inversion formula in $M_{m}^{p,q}(\rrd)$
			\begin{equation}\label{invfmpq}
			f=\frac{1}{( \gamma,g)}\int_{\rr {2d}}V_{g}f\left(x,\xi\right)M_{\xi}T_{x}\gamma\,{d}x{d}\xi.
			\end{equation}
			In short, 
            \[
			\mathrm{Id}_{M_{m}^{p,q}}=( \gamma,g)^{-1}V_{\gamma}^{*}V_{g}.
			\]		
    \end{enumerate}
\end{theorem}

%%%%%%%%%%%%%%%%%%%%%%%%%%%%%%%%%%%%%%%%%%%%%%%%%%%%%%%%%%%%%%%%%%
\subsection{Variation Space and Maurey’s Sampling Result}
\label{sec:prelim:VariationSpace}
%%%%%%%%%%%%%%%%%%%%%%%%%%%%%%%%%%%%%%%%%%%%%%%%%%%%%%%%%%%%%%%%%%

Let $\mathcal{B}$ be a Banach space, and let $\bb{D}\subset \mathcal{B}$ be a collection of non-zero elements, which we call a \emph{dictionary},
(i.e.,\ a collection of \emph{atoms}).
For geneal nonlinear approximation, the ordering of $\bb{D}$ is irrelevant, only the choice of atoms and their coefficients matters.
For $N\in \nn{}$, and  $M>0$, we define the nonlinear manifold of $N$-term, $\ell_1$-regularized approximants with respect to the dictionary $\bb{D}$ as follows:
\begin{equation*}%\label{eq:Sigma_N_M}
    \Sigma_{N,M}(\bb{D})
        :=\Bigl\{\;
            \textstyle\sum_{j=1}^{N} a_{j} h_{j} \; : \;
            h_{j}\in\bb{D},\;
            \sum_{j=1}^{N} |a_{j}| \le M
          \Bigr\}.
\end{equation*}
Removing the $\ell_1$ regularization constraint yields the $N$-term nonlinear manifold
\begin{equation*}%\label{eq:Sigma_N}
    \Sigma_{N}(\bb{D})
        := \bigcup_{M>0} \Sigma_{N,M}(\bb{D}).
\end{equation*}
Thus $\Sigma_{N,M}(\bb{D})$ consists of all linear combination of at most $N$ atoms drawn from $\bb{D}$, obtained through an $\ell^1$-regularization argument,
whereas $\Sigma_{N}(\bb{D})$ does not involve any regularization. Given a target function in the Banach space $\mathcal{B}$, the nonlinear manifold is used to generate the best possible combination of atoms that approximate the target as accurately as possible.
To rigorize this, we observe that any bounded linear combination of atoms can be normalized into a convex combination. This allows us to measure the complexity of a function by the smallest scaling factor required to fit it within the convex hull of the dictionary.
\begin{definition}
    Let $\mathcal{B}$ be a Banach space and $\bb{D}\subseteq\mathcal{B}$ be a dictionary.
    Then for $f\in\mathcal{B}$, the variation norm of $\bb{D}$ is defined as
    \begin{align*}
        \norm{f}{\mathcal{K}(\bb{D})}&:=\inf\{c>0:f/c\in\overline{\operatorname{conv}(\pm\bb{D})}\}
    \intertext{were, $\overline{\operatorname{conv}(\pm\bb{D})}$
        is the closure of the convex hull of $\bb{D}\cup(-\bb{D})$.
        The corresponding variation space is then the set of functions with finite variation norm}
        \mathcal{K}(\bb{D})&:=\{f\in\mathcal{B}:\norm{f}{\mathcal{K}(\bb{D})}<\infty\}.
    \end{align*}
\end{definition}
The significance of this norm lies in its ability to control the convergence rate of sparse approximations. Specifically, an adaptation of Maurey's approximation result for functions belonging to the variation space of a dictionary is presented in \cite{Siegel22SharpBoundsApproximation} as follows:
\begin{proposition}[Approximation Rate in Type-2 Banach Spaces]
\label{prop:approximation_type2}
    Let $\mathcal{B}$ be a type-2 Banach space
    and $\bb{D} \subset \mathcal{B}$ be a dictionary
    with $K_{\bb{D}}:=\sup _{d \in \bb{D}}\|d\|_{\mathcal{B}}<\infty$.
    Then for $f \in \mathcal{K}(\bb{D})$, we have
    \begin{align*}
        \inf_{f_N \in \Sigma_{N, M_f}(\bb{D})}\left\|f-f_N\right\|_{\mathcal{B}}
        \leq 4 C_{2,\mathcal{B}} K_{\bb{D}}\norm{f}{\mathcal{K}(\bb{D})} N^{-\frac{1}{2}}
    \end{align*}
    with $M_f=\norm{f}{\mathcal{K}(\bb{D})}$.
\end{proposition}
Potential extensions of Maurey's approximation result are discussed in \cite[Section 8]{DeVore98NonlinearApproximation}.
For later use, we state the following characterization of elements in 
$\mathcal{K}(\bb{D})$ in terms of representing measures on the dictionary $\bb{D}$.

\begin{proposition}[{\cite[Lemma 3]{Siegel23CharacterizationVariationSpaces}}]
\label{prop:variation_norm}
    Let $\mathcal{B}$ be a Banach space and suppose that $\bb{D}\subset\mathcal{B}$ is bounded.
    Then $f\in\mathcal{K}(\bb{D})$ if there
    exists a Borel measure $\mu$ on $\bb{D}$ such that
    \begin{align*}
        f=\int_{\bb{D}}i_{\bb{D}\to\mathcal{B}}\,d \mu.
    \end{align*}
    Moreover,
    \begin{align*}
        \norm{f}{\mathcal{K}(\bb{D})}
        =\inf\left\{\norm{\mu}{}:
        f=\int_\bb{D}i_{\bb{D}\to\mathcal{B}}\,d \mu\right\},
    \end{align*}
    where the infimum is taken over all Borel measures $\mu$
    defined on $\bb{D}$, and $\| \mu\|$ is the total variation of $\mu$. 
\end{proposition}

%%%%%%%%%%%%%%%%%%%%%%%%%%%%%%%%%%%%%%%%%%%%%%%%%%%%%%%%%%%%%%%%%%
\section{Embedding Results for Modulation Spaces}
\label{subsec:ConvergenceInBochnerSobolevNorms}
%%%%%%%%%%%%%%%%%%%%%%%%%%%%%%%%%%%%%%%%%%%%%%%%%%%%%%%%%%%%%%%%%%
In this section, we present the embedding results between Sobolev spaces and modulation spaces that will be required for the neural approximation analysis developed in the subsequent section.
\begin{remark}
    In a finite-dimensional space all the norms are equivalent. In particular, for every $q\in[1,\,\infty]$,  we have
    \begin{align*}
         \|f\|_{W^{n,q}(\Omega)} 
         &= \left( \sum_{|\alpha|\leq n} \|\partial^\alpha f\|_{L^q(\Omega)}^q \right)^{\frac{1}{q}} = \left\| \left( \|\partial^\alpha f\|_{L^q(\Omega)} \right)_{|\alpha|\leq n} \right\|_{\ell^q} \\
         &\asymp \left\| \left( \|\partial^\alpha f\|_{L^q(\Omega)} \right)_{|\alpha|\leq n} \right\|_{\ell^1} = \sum_{|\alpha|\leq n} \|\partial^\alpha f\|_{L^q(\Omega)}.
    \end{align*}
In view of the norm equivalence, we will use whichever definition is appropriate in context. 
\end{remark}

\begin{proposition}\label{prop:SobolevModulation_embedding}
Consider $d\in \nn{}$, $n\in \zzp{}$, $2\leq r<\infty$,   and $1\leq s,q\leq2$ such that 
\begin{equation}\label{indici}
		\frac1{r'}+1=\frac1s+\frac1q.
\end{equation} 
\begin{center}
\begin{tikzpicture}[scale=5, >=latex]
    % --- Axes and Grid ---
    \draw[->, thick] (0, 0) -- (1.2, 0) node[right] {$1/s$};
    \draw[->, thick] (0, 0) -- (0, 1.2) node[above] {$1/q$};
    \draw[dashed, gray!50] (0,0) grid (1.1,1.1);
    
    % Ticks
    \foreach \x in {0.5, 1} \draw (\x, 0.02) -- (\x, -0.02) node[below] {\x};
    \foreach \y in {0.5, 1} \draw (0.02, \y) -- (-0.02, \y) node[left] {\y};
    
    % --- Admissible Region (Triangle) ---
    % Vertices: (0.5, 1), (1, 0.5), (1, 1)
    \fill[blue!20] (0.5, 1) -- (1, 0.5) -- (1, 1) -- cycle;
    \draw[blue, thick] (0.5, 1) -- (1, 0.5) node[midway, below, align=center, rotate=-45, black] {\footnotesize $\frac{1}{s} + \frac{1}{q} = 1.5$};
    \draw[blue, thick] (1, 0.5) -- (1, 1);
    \draw[blue, thick] (1, 1) -- (0.5, 1);
    
    % --- Constraint Box (1 <= s,q <= 2 => 0.5 <= 1/s, 1/q <= 1) ---
    \draw[dashed, black] (0.5, 0.5) rectangle (1, 1);
    
    % --- Labels for Vertices ---
    \fill[red] (0.5, 1) circle (0.3pt) node[above left, black] {$(r=2)$};
    \fill[red] (1, 0.5) circle (0.3pt) node[below right, black] {$(r=2)$};
    \fill[red] (1, 1) circle (0.3pt) node[above right, black] {$(r=\infty)$};
    
    % --- Main Label ---
    \node[align=center] at (0.5, -0.25) {\textbf{Admissible Region}\\ $2 \le r < \infty$};
\end{tikzpicture}
\end{center}
    Let $U\subset \rr{d}$  be a bounded and measurable set with non-empty interior. If $f\in M^{p,q}_{1\otimes v_n}(\rr{d})$, then we have  
    $$
        \norm{f}{W^{n,r}(\Omega)}\leq C_{d,n}|\Omega|^{-d/p}\|\chi_\Omega\|_{\mathscr{F}L^s(\rr{d})}\| f\|_{M^{p,q}_{1\otimes v_n}(\rr{d})},\quad 0< p\leq\infty.
    $$
\end{proposition}
\begin{proof}
	We  use the fact that 
    $$
    \norm{f}{W^{n,r}(\Omega)}= \sum_{|\alpha|\leq n} \|\partial^\alpha f\|_{L^q(\Omega)}=\sum_{|\alpha|\leq n} \|\chi_\Omega\partial^\alpha f\|_{L^q(\rr d)}.
    $$
	Following \cite[Lemma 2.4 and Definition 3.1]{Tartar07IntroductionSobolevSpaces},
	for any $0<\epsilon <1$ , we define the smoothing sequence:
	\begin{align*}
		\rho_\epsilon(x):=
		\frac{1}{\epsilon^d\|\phi\|_{L^1}}\phi
		\left(\frac{x}{\epsilon}\right)\quad
		\text{with}
		\quad\phi(x)=\exp\left(-\frac{1}{1-|x|^2}\right)\chi_{B_1(0)}(x),
	\end{align*}
	where $B_1(0)$ is the closed unit ball.
	Thus, \(\rho_\epsilon\in C_c^\infty(\rr d)\) with  $$\operatorname{supp}\rho_\epsilon =\overline{B_\epsilon(0)},\quad \|\rho_\epsilon\|_{L^1}=1
		\quad\text{and}\quad
		\|\rho_\epsilon\|_{L^2}
		\leq \frac{1}{\epsilon^{d/2}}.$$
The last inequality follows by substitution in multiple variables and the fact that \(\rho_1(x)<1\), for all \(x\in\rr{d}\).
	For a domain $\Omega\subset \rr{d}$ we define the smoothed characteristic function of \(\Omega\) as
	\begin{align*}
		\chi_{\Omega}^\epsilon:=\chi_{\Omega} \conv \rho_{\epsilon}.
	\end{align*}
	Observe that
	$$
        \operatorname{supp}\chi_{\Omega}^\epsilon\subseteq \overline{\operatorname{supp}\chi_\Omega+\operatorname{supp}\rho_\epsilon}\subseteq {\Omega}_\epsilon,
    $$
	where 
    $${\Omega}_\epsilon:=\{x\in \rr d: \exists y\in\Omega\,\,  \mbox{such\,that}\, |x-y|\leq \epsilon\}.$$
	It was shown in 
	\cite[Proposition A.2]{Abdeljawad23SpaceTimeApproximationShallow}   that
	
	\begin{align*}%\label{eq:conv_smoothed_char}
		\lim_{\epsilon\to0}\norm{\chEF{\Omega}{\epsilon}h}{L^{r}}
		=\norm{\chF{\Omega}h}{L^{r}},
	\end{align*}
	for any $h:\rrd\to\rr{}$  locally in $L^r$.
		
	Now, by the Hausdorff-Young inequality,  for every $|\alpha|\leq n$,
	$$\|\chF{\Omega}^\epsilon\partial^\alpha f\|_{L^r}=\|\mathscr{F}^{-1}\mathscr{F}(
	\chF{\Omega}^\epsilon\partial^\alpha f)\|_{L^r}\leq \|\mathscr{F}(
	\chF{\Omega}^\epsilon\partial^\alpha f)\|_{L^{r'}}.$$
	For compactly supported functions, the $M^{p,q}$-norm is equivalent to the $\mathscr{F}L^q$-norm, see, e.g., \cite[Proposition 2.3.26]{Cordero20TimeFrequencyAnalysisOperators}. In detail, let $R>0$ such that $\Omega_\epsilon\subset B_R(0)$,  and consider   a window $g\in\mathcal{C}_c^\infty(\rr d)$ such that $g=1$
	on $B_{2R}(0)$. Then we have
	$$\hat{h}(\xi) \chF{\Omega}^\epsilon (x)=V_g h(x,\xi)\chF{\Omega}^\epsilon(x)$$
	so that 
	$$
        \|\mathscr{F}(\chF{\Omega}^\epsilon\partial^\alpha f)\|_{L^{r'}}\leq |\Omega_\epsilon|^{-1/p}\|\chF{\Omega}^\epsilon\partial^\alpha f\|_{M^{p,r'}}.
    $$	
	Using the multiplication properties for 
			 modulation spaces, cf.  \cite[Proposition 2.4.23]{Cordero20TimeFrequencyAnalysisOperators}, with the index relations
		$$
             \frac{1}{p}=\frac{1}{\infty}+\frac1p,\quad \frac{1}{r'}+1=\frac1s+\frac1q,
        $$
			 (notice that this implies $1\leq s,q\leq 2$) we obtain the bound 
			$$\|\chF{\Omega}^\epsilon\partial^\alpha f\|_{M^{p,r'}}\lesssim\|\chF{\Omega}^\epsilon\|_{M^{\infty,s}}\|\partial^\alpha f\|_{M^{p,q}}.$$
Since $\chF{\Omega}^\epsilon$ is compactly supported, we use the result in \cite[Proposition 2.3.26]{Cordero20TimeFrequencyAnalysisOperators} which gives
			$$\|\chF{\Omega}^\epsilon\|_{M^{\infty,s}}\leq  \|\chF{\Omega}^\epsilon\|_{\mathscr{F} L^s}$$
			and, as already observed in  
			\cite{Abdeljawad23SpaceTimeApproximationShallow}  (see formula (2.10), 
			$$\|\mathscr{F}(\chi_\Omega \ast\rho_\epsilon)\|_{L^s} \leq \|\mathscr{F}(\chi_\Omega )\|_{L^s}\|\mathscr{F}(\rho_\epsilon)\|_{L^1}=\|\mathscr{F}(\chi_\Omega )\|_{L^s}. $$
			Finally,
			$$\|\partial^\alpha f\|_{M^{p,q}}\leq \| f\|_{M^{p,q}_{1\otimes v_{|\alpha|}}},$$
			see, e.g.,  \cite[Theorem 2.3.14]{Cordero03TimeFrequencyanalysis},
			and the inclusion relations for modulation spaces (see Proposition 2.4.18 in \cite{Cordero20TimeFrequencyAnalysisOperators})
			give
			$$ \| f\|_{M^{p,q}_{1\otimes v_{|\alpha|}}}\leq C\| f\|_{M^{p,q}_{1\otimes v_n}},\quad \forall\,\alpha\in \zzp{d} \text{ such that }|\alpha|\leq n.$$
			To sum up,
			\begin{align*}\norm{f}{W^{n,r}(\Omega)}&= \sum_{|\alpha|\leq n} \|\chi_\Omega\partial^\alpha f\|_{L^q(\rr d)}\leq \|\mathscr{F}(\chi_U)\|_{L^s}\sum_{|\alpha|\leq n}\lim_{\epsilon\to 0}|\Omega_\epsilon|^{-d/p}\| f\|_{M^{p,q}_{1\otimes v_n}}\\
				&\leq C_{d,n}|\Omega|^{-d/p}\|\mathscr{F}(\chi_U)\|_{L^s}\| f\|_{M^{p,q}_{1\otimes v_n}}.
				\end{align*}
			This concludes the proof.
\end{proof}

\begin{remark}
		Observe that from \cref{indici} and the fact that $2\leq r<\infty$ we infer $1\leq s,q\leq 2$. Furthermore, we refer to Section 3.1.1 in \cite{Abdeljawad23SpaceTimeApproximationShallow} for the structure of the domain $\Omega$ and the degree $s$ such that $\chi_{\Omega}\in \mathscr{F}L^s(\rr{})$. For example, in dimension $d=1$, let $\Omega=[-1/2,1/2]$, then $$\mathscr{F}\chi_U=\frac1{\sqrt{2\pi}} \mbox{sinc} (\cdot/2)\in\mathscr{F}L^s(\rr{}),$$ for every $s>1$. 
\end{remark}

Note that, by the inclusion relations for modulation spaces with general weight functions (see, e.g., Theorem 2.4.17 in \cite{Cordero20TimeFrequencyAnalysisOperators}), we can generalize the result in Proposition \ref{prop:SobolevModulation_embedding} to any weight $m$ such that
\begin{equation}\label{eq:weight_ellipticity}
 v_n(\xi)\leq C\,m(x,\xi),\quad (x,\xi)\in\rr {2d}.
\end{equation}

\begin{corollary}
Assume the same hypotheses as in \cref{prop:SobolevModulation_embedding}.
If in addition the weight $m$ satisfies \cref{eq:weight_ellipticity}, then 
    $$ \norm{f}{W^{n,r}(\Omega)}
        \leq C_{d, n} 
      |\Omega|^{-d/p} \norm{{\chF{\Omega}}}{\FL{s}(\rr{d})}
        \norm{f}{M_m^{p,q}(\rr{d})}.$$
\end{corollary}
The extension to weights with sub exponential or exponential growth is also possible using  Gelfand-Shilov spaces as window classes and more general modulation spaces contained in their duals, cf. \cite{Teofanov06ModulationSpacesGelfandShilov}.

\begin{proposition}
 Under the assumptions of Proposition \ref{prop:SobolevModulation_embedding}, if, in addition,  $f\in M^{p,q}_n(\rr{d})$ is a bandlimited function with $\operatorname{supp}\hat{f}=K\subset \rr d$, then
	$$\norm{f}{W^{n,r}(\Omega)}\leq C_{K,n}|\Omega|^{-d/p}\|\chi_U\|_{\mathscr{F}L^s(\rr{d})}\| f\|_{M^{p,q}(\rr{d})},$$
 for every $0< p\leq\infty$, and   with $q,r,s$ satisfying \cref{indici}.\end{proposition}
\begin{proof}
The first part goes as the proof of Proposition \ref{prop:SobolevModulation_embedding}. 
We will show that  
$$\|\partial^\alpha f\|_{M^{p,q}(\rr{d})}
\leq
C_n \|f\|_{M^{p,q}(\rr{d})},\quad \alpha\in \zzp{d}\text{ such that }|\alpha|\leq n.
$$
 
For $k\in\zz{d}$, consider  the frequency-uniform decomposition operator by
	\begin{equation*}%\label{C2freq-unifdecop}
	\Box_k:= \mathscr{F}^{-1} \sigma_k \mathscr{F},
	\end{equation*}
    where $\{\sigma_k\}_k$ is a smooth partition of unity.

The previous operator allows to introduce an equivalent norm on the  modulation spaces $M^{p,q}(\rr d)$, as follows, cf. Definition 2.3.24  and Proposition 2.3.25 in \cite{Cordero20TimeFrequencyAnalysisOperators},
	\begin{equation*}%\label{C2tildeMpq}
	\|f\|_{{M}^{p,q}(\rr d)}=\left(\sum_{k\in \zz{d}} \|\Box_k f\|^q_{L^p}\right)^\frac 1q,\quad f\in\mathscr{S}'(\rr d),
	\end{equation*}
	with obvious modification for $q=\infty$.
    
Now, if $\hat{f}$ has compact support $K\subset \rr d$, the sum above is finite. 
Note that, for every  $\alpha\in \zzp{d}$ such that $|\alpha|\leq n$, we have
$$
\operatorname{supp}\mathscr{F}(\partial^\alpha f)\subseteq\operatorname{supp}\xi^\alpha\mathscr{F}f\subseteq \operatorname{supp}\mathscr{F}{f}=K.
$$
We compute
    $$\|\partial^\alpha f\|_{M^{p,q}}=\left(\sum_{finite} \|\Box_k \partial^\alpha f\|^q_{L^p}\right)^\frac 1q.
    $$
Observe that
    \begin{align*}
        \|\Box_k \partial^\alpha f\|_{L^p}=\|\mathscr{F}^{-1}\xi^\alpha \phi_K \mathscr{F}\mathscr{F}^{-1}\sigma_k\mathscr{F}f\|_{L^p}\leq  \|\mathscr{F}^{-1}\sigma_k\mathscr{F}f\|_{L^p}\leq C_{K,n}\|\Box_k  f\|_{L^p},
    \end{align*}
    where $\phi_K\in\mathcal{C}_c^\infty(\rr d)$  such that $\phi_k(\xi)=1,$ for every $\xi\in K.$  The multiplier  $$T_{K,\alpha}h:=\mathscr{F}^{-1}(\xi^\alpha \phi_K) \mathscr{F}h=\Phi_{K,\alpha}\ast h,$$
    where $\Phi_{K,\alpha}:=\mathscr{F}^{-1}(\xi^\alpha \phi_K )\in\mathscr{S}(\rr d)$, 
     is  bounded on every $L^p(\rr d)$,  $1\leq p\leq\infty$,
     with
     $$
     \|T_{K,\alpha}h\|_{L^p}\leq \|\Phi_{K,\alpha}\|_{L^1}\|h\|_{L^p}\leq C_{K,n}\|h\|_{L^p},
     $$
     for every  $\alpha \in \zzp{d}$
     such that  $|\alpha|\leq n.$
    Hence,
    $$
        \|\partial^\alpha f\|_{M^{p,q}}=\left(\sum_{finite} \|\Box_k \partial^\alpha f\|^q_{L^p}\right)^\frac 1q
        \leq C_{K,n} \left(\sum_{finite} \|\Box_k f\|^q_{L^p}\right)^\frac 1q\asymp\| f\|_{M^{p,q}}
    $$
which gives the desired result.\end{proof}

\begin{lemma}
  Let $\Omega\subset \rr{d}$  be a bounded and measurable set with non-empty interior.  Consider $0<p,q\leq \infty$, $s\in\rr{}$. If $\chi_\Omega f\in M^{p,q}_{1\otimes v_s}(\rr{d})$,  then $\chi_\Omega f\in \mathscr{F}L^q_{v_s}(\rr{d})$ with 
    \begin{equation*}
        \|\chi_\Omega f\|_{\mathscr{F}L^q_{v_s}(\rr{d})}
        \leq
        |\Omega|^{-1/p} \|\chi_\Omega f\|_{ M^{p,q}_{1\otimes v_s}(\rr{d})}.
    \end{equation*}    
\end{lemma}
\begin{proof}
    The main intuition comes from the fact that, for compactly supported functions, the $M^{p,q}$-norm is equivalent to the $\mathscr{F}L^q$-norm, see, e.g., \cite[Proposition 2.3.26]{Cordero20TimeFrequencyAnalysisOperators}.
     	
    Consider $f\in M^{p,q}_{1\otimes v_s}(\rd)$,   $R>0$ such that $\Omega\subset B_R(0)$, and $g \in\mathcal{C}_c^\infty(\rd)$ with  $g\equiv 1$ on $B_{2R}(0)$.
    Observe that 
    $$
    g(t-x)=1,\quad \forall t,x\in B_R(0),
    $$
    and, in particular, 
    $$
    g(t-x)=1,\quad \forall t,x\in \Omega.
    $$
    Hence, for every $\xi \in \rd$ ,
    $$\widehat{(\chi_\Omega f)}(\xi)\chi_{ \Omega}(x)=V_g (\chi_\Omega f)(x,\xi)\chi_{ \Omega}(x) $$
    and, taking the $L^p$-norm with respect to the $x$-variable,
    $$
    |\Omega|^{1/p}|\widehat{\chi_\Omega f}(\xi)|=\| V_g (\chi_\Omega )(\cdot,\xi)\chi_{\Omega}(\cdo)\|_{L^p}\leq \| V_g (\chi_\Omega f)(\cdot,\xi)\|_{L^p}.
    $$     
    This yields
    $$
    \|\widehat{\chi_\Omega f}\|_{L^q_{v_s}}\leq|\Omega|^{-1/p} \| \| V_g (\chi_ \Omega f)\|_{L^p}\|_{L^q_{v_s}},
    $$
    that is,
    $$
    \|\chi_\Omega f\|_{\cF L^q_{v_s}}\leq|\Omega|^{-1/p} \| \chi_ \Omega f\|_{M^{p,q}_{1\otimes v_s}},
    $$
    as desired.
\end{proof}
We also have a vice versa of the previous result, under additional assumptions on the set $\Omega$.

\begin{lemma}\label{cor:FourierLebesgue_embedding_in_modulation_space}
  Let $\Omega\subset \rr{d}$  be a bounded and measurable set with non-empty interior.  Consider $0<p,q\leq \infty$, $s\in\rr{}$. If
	$\chi_\Omega f\in \mathscr{F}L^q_{v_s}(\rr{d})$,  then $\chi_\Omega f\in M^{p,q}_{1\otimes v_s}(\rr{d}) $
	with 
     \begin{equation}\label{stima1}\|\chi_\Omega f\|_{ M^{p,q}_{1\otimes v_s}(\rr{d})}\leq    |\Omega+\Omega|^{1/p} \|\chi_\Omega f\|_{\mathscr{F}L^q_{v_s}(\rr{d})}.
     \end{equation}
\end{lemma}

\begin{proof}
    First, we assume  $\chi_\Omega f\in \cF L^q_{v_m}(\rd)$. Note that, since 
    $$
    \|T_x\chi_\Omega f\|_{M^{p,q}_{1\otimes v_s} }\asymp\|V_g(\chi_\Omega f)(\cdot-x,\cdot)\|_{L^{p,q}_{1\otimes v_s}}
    =
    \|V_g(\chi_\Omega f)(\cdot,\cdot)\|_{L^{p,q}_{1\otimes v_s}}\asymp \|\chi_\Omega f\|_{M^{p,q}_{1\otimes v_s}},
    $$
	we can assume that  $\Omega$  contains a ball $B_R(0)$. Consider a window  $g \in\mathcal{C}_c^\infty(\rd)$ with $\operatorname{supp}g\subset B_R(0)\subset\Omega$,  so that $\Omega + \operatorname{supp} g\subset \Omega+\Omega$ , where
    \[
        \Omega + \Omega =\{x+y,\,\,x,y\in\Omega\}.
    \]
    Moreover, we assume $\|\cF g\|_{L^1}=1$. 
     Then, $V_g (\chi_\Omega f)$  is nonzero only when $g(t-x)
     $ overlaps $\Omega$, in other words, for each $\xi\in\rd$, $V_g (\chi_\Omega f)(\cdot, \xi)$ is supported in $\Omega+\Omega$. Thus, using 
    $$
    V_g(\chi_\Omega f)(x,\xi)=e^{-2\pi ix\cdot \xi} \cF(\widehat{(\chi_\Omega f)}\cdot T_\xi \bar{\hat{g}})(-x),
    $$
    we can write
	$$
    |V_g (\chi_\Omega f )(x,\xi)| =  |\cF^{-1}(\widehat{(\chi_\Omega f)}T_\xi\bar{\hat{g}})(x)|,\; \text{such that }x\in\Omega+\Omega
    $$
    and, taking the $L^p$-norm for the $x$-variable,   
\begin{align*}
		 \|V_g (\chi_\Omega f )(\cdot,\xi)\|_{L^p}&\leq \left(\int_{\Omega+\Omega}dx\right)^{1/p}\|V_g (\chi_\Omega f )(\cdot,\xi)\|_{L^\infty}= |\Omega+\Omega|^{1/p}\|\cF^{-1}(\widehat{(\chi_\Omega f )}T_\xi\bar{\hat{g}})\|_{L^\infty}\\
		 &\leq |\Omega+\Omega|^{1/p} \|\widehat{(\chi_\Omega f )}T_\xi\bar{\hat{g}}\|_{L^1}\leq |\Omega+\Omega|^{1/p} |\widehat{(\chi_\Omega f )}|\ast |\check{g}|(\xi).
		 \end{align*}
Finally, taking the $L^q_{v_s}$-norm in the above inequalities,
	$$
    \| \|V_g \chi_\Omega f\|_{L^p}\|_{L^q_{v_s}}\leq |\Omega+\Omega|^{1/p}   \| |\widehat{\chi_\Omega f}|\ast |\check{g}|\|_{L^q}\leq |\Omega+\Omega|^{1/p}  \|\widehat{\chi_\Omega f}\|_{L^q_{v_m}}\|\check{g}\|_{L^1},
    $$
	i.e., 
    $$
    \|\chi_\Omega f\|_{M^{p,q}_{1\otimes v_s}}\leq C_g |\Omega+\Omega|^{1/p}\|\chi_\Omega f\|_{\cF L^q_{v_s}}=|\Omega+\Omega|^{1/p}\|\chi_\Omega f\|_{\cF L^q_{v_s}},
    $$
    where $C_g=\|\check{g}\|_{L^1}=1$.
\end{proof}

\begin{corollary}\label{cor:FourierLebesgue_embedding_in_modulation_space_convex}
Under the assumptions of Lemma~\ref{cor:FourierLebesgue_embedding_in_modulation_space}, assume in addition that \(\Omega \subset \rr d\) is convex. Then the estimate~\eqref{stima1} can be improved by replacing \(|\Omega + \Omega|^{1/p}\) with \(2^{1/p} |\Omega|^{1/p}\), that is,
\[
\|\chi_\Omega f\|_{M^{p,q}_{1 \otimes v_s}(\rr{d})} \leq 2^{1/p} |\Omega|^{1/p} \|\chi_\Omega f\|_{\mathscr{F}L^q_{v_s}(\rr{d})}.
\]
\end{corollary}

\begin{proof}
    The thesis follows by observing that, for every $x,y\in\Omega$, we can write  \[|x+y|=2|(x+y)/2|=2|z|\]
    where $z=(x+y)/2\in\Omega$.
\end{proof}

%%%%%%%%%%%%%%%%%%%%%%%%%%%%%%%%%%%%%%%%%%%%%%%%%%%%%%%%%%%%%%%%%%
\section{Convergence Rates for Approximation of Modulation Space}
\label{sec:M_space_approximation}
%%%%%%%%%%%%%%%%%%%%%%%%%%%%%%%%%%%%%%%%%%%%%%%%%%%%%%%%%%%%%%%%%%

In this section, we establish several results concerning the approximation capabilities of shallow neural networks for functions in certain weighted modulation spaces $M^{p,q}_{m}(\rr{d})$, evaluated under various norm errors.

To this end, we employ the phase representation of $ e^{2\pi i\eta\cdot x}$.
Recall that $\sigma\in W^{k, \infty}(\rr{})\subset M^\infty(\rr{})$  by \cite[Proposition 2.9]{Toft04ContinuityPropertiesModulation}. Furthermore, for any window $\varphi\in M^1(\rr{})$,  the STFT $V_\varphi \sigma$ belongs to the Wiener amalgam space 
$$
W(\mathscr{F}L^1,L^\infty)(\rr{2})\subset \mathcal{C}(\rr{2})\cap L^\infty(\rr{2}),
$$
see, e.g., \cite[Lemma 2.4.15]{Cordero20TimeFrequencyAnalysisOperators}, where $\mathcal{C}(\rr{2})$ is the space of continuous functions on $\rr{2}$.  Consider a real non-zero window function $\varphi\in M^1(\rr{})$. Since $M^1(\rr{})\hookrightarrow L^1(\rr{})$, the integral 
 \begin{equation*}
        (V_\varphi \sigma)(t,\tau ) = \int _{\rr{} }\sigma(s) \overline {\varphi (s-t)}e^{-2\pi is\tau}\, ds= \int _{\rr{} }\sigma(s)  {\varphi (s-t)}e^{-2\pi i s \tau}\, ds
        \end{equation*}
       is absolutely convergent:
         \begin{equation*}
        |(V_\varphi \sigma)(t,\tau ) |\leq  \int _{\rr{} }|\sigma(s)| \,|{\varphi (s-t)}|\, ds\leq \|\sigma\|_{L^\infty}\|\varphi\|_{L^1}\lesssim \|\sigma\|_{L^\infty}\|\varphi\|_{M^1}.
        \end{equation*}
     Using the linear change of variables    
     $$
        s= \eta\cdot x+b, \,\text{ for some fixed }\, \eta,x\in \rrd,
     $$ 
    the STFT $(V_\varphi \sigma)(t,\tau )$ can be written as 
    \begin{align*}
        (V_\varphi \sigma)(t,\tau ) &= \int _{\rr{} }\sigma(s) \overline {\varphi (s-t)}e^{-2\pi i s \tau}\, ds
        \\
         &=\int _{\rr{} }\sigma(\eta\cdot x+b) {\varphi (\eta\cdot x+b-t)}e^{-2\pi i ({\eta\cdot x+b})\tau}\, db.
    \end{align*}
    Since $\sigma$ and $\varphi$  are  non-zero,  the STFT is a non-zero continuous function on $\rr{2}$, hence the following condition holds:\par
    \textnormal{\textbf{Condition (A)}}: it exists a $(t,\tau)\in \rr{2}$, $\tau\not=0$, such that
    $(V_\varphi \sigma)(t,\tau)\neq 0$.\par
    \vspace{0.1truecm}
Under the above condition we can write
    \begin{align*}
        e^{2\pi i({\eta\cdot x})\tau}
        =\left((V_\varphi \sigma)(t,\tau ) \right)^{-1}
        \int_{\rr{}}\sigma(\eta\cdot x+b) {\varphi (\eta\cdot x+b-t)}e^{-2\pi i {b} \tau}\, db.
    \end{align*}
    This implies that
    \begin{equation}\label{eq:phase_identity}
        e^{2\pi i\eta\cdot x}
        =\left((V_\varphi \sigma)(t,\tau ) \right)^{-1}
        \int_{\rr{}}\sigma\left(\frac{\eta\cdot x}{\tau}+b\right) {\varphi \left(\frac{\eta\cdot x}{\tau}+b-t\right)}e^{-2\pi i {b} \tau}\, db,
    \end{equation}
where the integral on the right-hand side is absolutely convergent.
The above computations will play a role in the proofs of the following theorems.
\begin{theorem}[Local Approximation]
\label{thm:approximation_sobolev_space_local}   
    Let $n\in \zzp{}$, $0< q\le 2 \le r$, and $0< p<\infty$. Consider a bounded domain $\Omega\subset \rrd$, and an activation function
    $$\sigma\in W^{k, \infty}(\rr{})\setminus\{0\},\quad  \mbox{with} \,\,k\geq n.$$  Let $\varphi\in \mathscr{S}(\rr{})\setminus \{0\}$, $\phi\in \mathscr{S}(\rr{d})\setminus \{0\}$,  and define the dictionary $\bb{D}$ by
    \begin{equation}\label{eq:D}
         \bb{D} = \{ x \mapsto \sigma\left(\tfrac{\eta\cdot x}{\tau}+b\right)
        \varphi(\tfrac{\eta\cdot x}{\tau}+b-t)\phi(x-y) \text{ such that }(y,\eta, b) \in \rr{d}\times\rr{d}\times \rr{}\},
    \end{equation}
    with $t,\tau$ satisfying \textnormal{\textbf{Condition (A)}}.
    Let $m =(v_{s_1}\otimes v_{s_2})$ with
    \begin{equation}\label{eq:indices}
    \left\{
        \begin{array}{l}
                s_1=0\quad\mbox{if}\quad 0<p\leq 1,\quad\,s_1>\frac{d}{p'},\quad\mbox{if}\quad p>1\\
                s_2=n+1\quad\mbox{if}\quad 0<q\leq 1,\quad s_2>n + 1 +\frac {d}{q'},\,\,\mbox{if}\quad\,q>1.
                 \rule{0mm}{0.55cm}
        \end{array}
        \right. 
    \end{equation}
    Then, for every $f\in M^{p,q}_m(\rd)$,
    there exists a constant $C>0$ such that
    \begin{align}\label{eq:approximation_bound}
        \inf_{f_{N}\in\Sigma_{N}(\bb{D})}\norm{f-f_N}{W^{n,r}(\Omega)}
        &\leq C N^{-\frac{1}{2}} \abs{\Omega}^{1/r} 
        \norm{f}{M^{p,q}_m(\rd)},
    \end{align}
    for all $N\in\nn{}$.
\end{theorem}

To aid in visualizing the parameter constraints required for the approximation rates, we illustrate the admissible regions for the indices $s_1$ and $s_2$ in \cref{fig:admissible_regions}.
\begin{figure}[htbp]
    \centering
    \includegraphics[width=\textwidth]{figs/admissble_values_thm.pdf} 
    \caption{Visual representation of the admissible regions for the weight indices $s_1$ (left) and $s_2$ (right) as defined in \cref{eq:indices}. The solid blue lines indicate the constant values for $p, q \le 1$, while the red shaded areas represent the necessary growth conditions for $p, q > 1$, which depend on the dimension $d$ and derivative order $n$. In this illustration, we set $d=2$ and $n=1$.}
    \label{fig:admissible_regions}
\end{figure}
\begin{proof}
    First, we consider $f\in\mathscr{S}(\rrd)$, so that at a later stage we can invoke the density of $\mathscr{S}$ in $M^{p,q}_m$, when $0< p,q<\infty$ (see \cref{eq:inclmod} and \cref{eq:density_mod}).
    Note that, since the $M^{p,q}_m$-norm is independent of the window function, we assume that $\phi\in \mathscr{S}(\rrd)$ is positive and that $\|\phi\|_{L^2}=1$.  Applying the inversion formula for the STFT in \eqref{invfmpq} (with $g=\gamma=\phi$),  we obtain 
    \begin{equation}\label{eq:STFT_signal_representation}
    f(x) = 
         \int_{\rr{2d}}
         V_\phi f(y,\eta)\,\phi(x-y)\,e^{2\pi i x\cdot \eta}\,dy\,d\eta,
    \end{equation}
    with converge in $M^{p,q}_m(\rrd)$ (see \cref{sec:modulation_spaces}).
    Observe that for every $h\in M^{p',q'}_{1/m}(\rrd)$,
    $$(f,h)=(V_\phi f,V_\phi h),$$ where on the left-hand side we have the duality between $M^{p,q}_m(\rrd)$  and $M^{p',q'}_{1/m} (\rrd)$, and on the right-hand side $L^{p,q}_m(\rr{2d})$ and $L^{p',q'}_{1/m}(\rr{2d})$. 
    For $\varphi\in\mathscr{S}(\rr{})\hookrightarrow M^1(\rr{})$ we choose $(t,\tau)$ satisfying \textnormal{\textbf{Condition (A)}} and  insert the identity \eqref{eq:phase_identity} for $e^{2\pi i{\eta\cdot x}}$ in the representation \eqref{eq:STFT_signal_representation}  of the signal $f$, obtaining:
    \begin{align*}
        f(x)
        &=
         \int_{\rr{2d}}
         V_\phi f(y,\eta)\,\phi(x-y)\,\left((V_\varphi \sigma)(t,\tau ) \right)^{-1}\\
         &\qquad\quad
        \quad \times \quad\int_{\rr{}}\sigma\left(\frac{{\eta\cdot x}}{\tau}+b\right) {\varphi \left(\frac{{\eta\cdot x}}{\tau}+b-t\right)}e^{-2\pi i {b}\cdot \tau}\, db\,dy\,d\eta\\
        &=\left((V_\varphi \sigma)(t,\tau ) \right)^{-1} \!
         \int_{\rr{2d}}\int_{\rr{}} V_\phi f(y,\eta)\,\phi(x-y)\\
        &\qquad\quad\times\quad\sigma\left(\frac{{\eta\cdot x}}{\tau}+b\right) {\varphi \left(\frac{{\eta\cdot x}}{\tau}+b-t\right)}e^{-2\pi i {b}\cdot \tau}\, db\,dy\,d\eta.\\
    \end{align*}
    In order to simplify the previous identity of the signal $f$, we define
    \begin{align*}
        C_{\sigma ,\varphi}
        &=   
        (V_\varphi \sigma)(t,\tau ) ^{-1}
    \end{align*}
    along with the parametrized function 
    \[
    \aff_{\eta, b}(x) \equiv \aff_{\tau, \eta, b}(x)= \frac{\eta\cdot x}{\tau}+b.
    \]
    Note that, we omit the dependence on $t, \tau$ in the definition of  $C_{\sigma,\varphi}$ as well as in the affine function $\aff$, since $t$ and $\tau$ are fixed constants.
    As a consequence, we get the following integral representation
    \begin{align*}
        f(x)
        &=
        C_{\sigma ,\varphi}\int_{\rr{2d}}\int_{\rr{}}
         \rho(x, y, \eta, b)
        e^{-2\pi i {b}\cdot \tau}
        V_\phi{f}(y, \eta)
        \,db\,dy\,d\eta,
    \end{align*}
    where
    $$
    \rho(x, y, \eta, b) =  \sigma\left(\aff_{\eta, b}\left(x\right)\right)
       \varphi(\aff_{\eta, b}\left(x\right)-t)\phi(x-y)
    $$
    (again we omit the dependence on $t$ and $\tau$ in the atom $\rho$, as they are fixed constants).
    Now we split the previous identity in two parts: first the element of the dictionary then the measure:
    \begin{align}
        \rho(x, y, \eta, b) &=  \sigma\left(\aff_{\eta, b}\left(x\right)\right)
        \varphi(\aff_{\eta, b}\left(x\right)-t)\phi(x-y),\label{eq:rho}
        \\
        d\mu_f(y, \eta, b) &= C_{\sigma ,\varphi}\, e^{-2\pi i {b}\cdot \tau}\,
        V_\phi{f}(y, \eta)
        \, d(b,y, \eta).\label{eq:mu_f}
    \end{align}
    As a result, we can write $f$ as
    \begin{align*}
        f = \int _{\bb{D}} i_{\bb{D} \to \mathcal{B}} \,d\mu_f,
    \end{align*}
    where $ \bb{D}$ is the dictionary in \eqref{eq:D}, namely,
    \begin{align*}
        \bb{D} = \{ \rho(\cdo, y, \eta, b) \text{ such that }(y,\eta, b) \in \rr{d}\times\rr{d}\times \rr{}\}
        \quad\text{and }
        \mathcal{B} = W^{n,r}(\Omega).
    \end{align*}
    Consequently, the variation norm of $f$ can be bounded in terms of the $L^1$ norm as follows:
    \begin{align*}
        \norm{f}{\mathcal{K}(\bb{D})}
        \le
        \int_{\rr{2d}}\int_{\rr{}} \,d\abs{\mu_f}(y, \eta, b)
        =
        \norm{\mu_f}{L^1}.
    \end{align*}
    Although the previous quantity provides a bound on the variation norm of $f$, this does not place $f$ within the variation space of the dictionary $\bb{D}$, since the bound does not converge over $b$.
    For this reason, we adjust the dictionary by introducing weights, as described in the following. Let $\vartheta$ be a weight defined as
    \begin{align*}
        {\vartheta}(\eta, b)
        :=v_n(\eta)
        v_s\left((\abs{b}-R_\Omega\abs{\frac{\eta}{\tau}})_+\right), \text{ such that }s<-1,% \text{ and } \gamma<0.
    \end{align*}
    and that $R_\Omega = \sup_{x\in \Omega} \abs{x}$. Note that, since $R_\Omega, \tau$ and $s$ are fixed constants, we do not include them among the variables that define the weight $\vartheta$.
    Then,  we define the dictionary $\tilde{\bb{D}}$ associated with the weight function $\vartheta$ and derived from the atoms in $\bb{D}$, as follows:
    \begin{align*}%\label{eq:scaled_activation}
        \tilde{\rho}(x, y, \eta, b)
        := \frac{\rho(x, y, \eta, b)}
        {\vartheta(\eta, b)},
    \end{align*}
    thus $\tilde{\bb{D}}$ is defined as
    \[
        \tilde{\bb{D}} = \{ x\mapsto \tilde{\rho}(x, y, \eta, b) \text{ such that }(y,\eta, b) \in \rr{d}\times\rr{d}\times \rr{}\}.
    \]
    Consequently, the representation of $f$ can be expressed for all $x\in \rrd$ as:
    \begin{align*}
        f(x)
        &=
        C_{\sigma ,\varphi}\int_{\rr{2d}}\int_{\rr{}}
         \tilde{\rho}(x, y, \eta, b)
        \vartheta(\eta, b)e^{-2\pi i {b}\cdot \tau}
        V_\phi{f}(y, \eta)
         \,db\,dy\,d\eta.
    \end{align*}
    In the construction of the weight, we mainly focus on the convergence with respect to $b$.
    To proceed with our analysis, we derive an upper bound for the variation norm of $f$ (with respect to the dictionary $\tilde{\bb{D}}$) in terms of an appropriate modulation norm.
    
    Observe that from \cref{prop:variation_norm} and \cref{eq:mu_f}, it is straightforward that 
    \[
        \norm{f}{\mathcal{K}(\tilde{\bb{D}})}
        =
        \inf\left\{\norm{\mu}{}:
        f=\int_{\tilde{\bb{D}}}i_{\tilde{\bb{D}}\to W^{n,r}(\Omega)}\,d \mu\right\}
        \leq 
        \norm{\tilde{\mu}}{L^1},
    \]
    with the density  $\tilde{\mu}(y, \eta, b) = C_{\sigma ,\varphi}\, e^{-i {b}\cdot \tau}\,
         \vartheta(\eta, b)\,V_\phi{f}(y, \eta).$
        % \, d(b,y, \eta)$.
    Then, the variation norm of $f$ is now bounded as follows
    \begin{align*}
        \norm{f}{\mathcal{K}(\tilde{\bb{D}})}
        \leq C_{\sigma ,\varphi}
        \int_{\rr{2d}}\int_{\rr{}}
         \vartheta(\eta, b)
        \abs{V_\phi{f}(y, \eta)}
         \,db\,dy\,d\eta.
    \end{align*}    
    The integration over $b$ involves only the weight $\vartheta$, and can therefore be characterized as a function of $\eta$
    \begin{align*}
        I(\eta)&:=\int_{\rr{}} \vartheta \left( \eta,b \right) \,d b
        =
        \int_{\rr{}} v_n(\eta)
        v_s\left((\abs{b}-R_\Omega\abs{\frac{\eta}{\tau}})_+\right) \,d b
        \\
        &=  2v_n(\eta)  \left( \int_0^{R_\Omega\abs{\frac{\eta}{\tau}}}  \,d b
        + \int _{R_\Omega\abs{\frac{\eta}{\tau}}}^\infty v_s(b-R_\Omega\abs{\frac{\eta}{\tau}}) \,d b\right)      
        \\
        &=2v_n(\eta)  \left( {R_\Omega\abs{\frac{\eta}{\tau}}}
        + \frac{1}{2}\operatorname{B}(\frac{1}{2}, \frac{-s-1}{2})\right)      
        \leq  C_{\Omega , s}v_{n + 1}(\eta),
    \end{align*}
    where $\operatorname{B}(\frac{1}{2}, \frac{-s-1}{2})$ denotes the Beta function and
    $$
        C_{\Omega, s}  = 2R_\Omega + \operatorname{B}(\frac{1}{2}, \frac{-s-1}{2}) 
    $$
    is a finite positive constant depending on $\Omega$ and $s$.
    Note that if $\Omega$ is the unit ball and $s=-2$ then $C_{\Omega, s}= 2 + \pi$.
    Consequently, the variation norm of $f$ is controlled by
    \begin{equation*}%\label{eq:first_bound_on_variation_norm}
    \begin{aligned}
        \norm{f}{\mathcal{K}(\tilde{\bb{D}})}
            &\leq C_{\sigma ,\varphi}
        \int_{\rr{2d}}\int_{\rr{}}
         \vartheta(\eta, b)
        \abs{V_\phi{f}(y, \eta)}
         \,db\,dy\,d\eta
        =C_{\sigma ,\varphi} \int_{\rr{2d}}\int_{\rr{}}
         \vartheta(\eta, b)\,db\,
        \abs{V_\phi{f}(y, \eta)}
         \,dy\,d\eta
        \\
        &=C_{\Omega , s} C_{\sigma ,\varphi}  
           \int_{\rr{2d}}
           v_{n + 1}(\eta)
           \abs{V_\phi{f}(y, \eta)}
         \,dy\,d\eta
        \leq C_{\Omega , s} C_{\sigma ,\varphi}
        \int_{\rr{2d}}
        v_{n + 1}(\eta)
        \abs{V_\phi{f}(y, \eta)}
         \,dy\,d\eta.
    \end{aligned}
    \end{equation*}
    Using the inclusion relations of Theorem \ref{thm:embeddings} for $p>1$ or $q>1$,     we infer that
    \begin{equation*}%\label{eq:estM1}
    \int_{\rr{2d}}
        v_{n + 1}(\eta)
        \abs{V_\phi{f}(y, \eta)}
        \,dy\,d\eta=
        \|f\|_{M^1_{1\otimes v_{n+1}}}
        \le C_{p,q}
        \,
        \|f\|_{M^{p,q}_{v_{s_1}\otimes v_{s_2}}}
    \end{equation*}
    with  the index relation
    \begin{equation*}%\label{eq:eindex}
        \frac1{p}>1-\frac{s_1}{d},\quad \frac1{q}>1+\frac{n+1-s_2}{d},
    \end{equation*}
    for a suitable constant $C_{p,q}>0$. Observe that, for $p\leq1$ or $q\leq1$ we have the weight $m=(1\otimes v_{n+1})$ by Theorem \ref{thm:embeddings}, as well. This yields the index relations in \eqref{eq:indices}.
    
    Finally, we get an upper bound to the variation norm of $f$ that involves weighted modulation norm of $f$ where the weight performs at most polynomially.
    Hence, %using \cref{eq:FL_q_tilde_inclusion_I} we get
    \begin{align}\label{eq:FL_q_tilde_inclusion_II}
        \norm{f}{\mathcal{K}(\tilde{\bb{D}})}
        &\le
         C_{p,q} C_{\Omega , s} C_{\sigma ,\varphi}
            \left(\int_{\rrd}
                \left(\int_{\rrd}
                        (v_{s_1}\otimes v_{s_2})^p(y,\eta)
                          \abs{V_\phi f(y,\eta)}
                    ^{p}
                dy\right)^{\frac{q}{p}}
            d\eta\right)^{\frac{1}{q}}
        \\
        &=  C\,
        \norm{f}{M^{p,q}_m}\nonumber
    \end{align} 
    where
    $
         C = C_{p,q} C_{\Omega , s} C_{\sigma ,d}
        \quad\text{and}\quad m =v_{s_1}\otimes v_{s_2}.
    $
    
    % \textbf{Bound for the Dictionary:} 
    In order to verify that the constructed dictionary lies within the underlying Banach space $W^{n,r}(\Omega)$, we establish a uniform bound, that is,
    $$
    \sup_{h\in \tilde{\bb{D}}}\norm{h}{W^{n, r}(\Omega)} < \infty.
    $$
    This also plays a key role in the application of the Maurey result.
    To this end, we check whether each function $ \tilde{\rho}(x, y, \eta, b)$ belongs to $W^{n, r}(\Omega)$ for any $y, \eta$ and $b$.
    Recall that the activation functions used to construct our dictionary take the form
    \[
        \tilde{\rho}(x, y, \eta, b)
        := \frac{\rho(x, y, \eta, b)}
        {\vartheta(\eta, b)} = \frac{\sigma\left(\aff_{\eta, b}\left(x\right)\right)
        \varphi(\aff_{\eta, b}\left(x\right)-t)\phi(x-y)}  {\vartheta(\eta, b)}.
    \]
    Since the weight $\vartheta$ is independent of the variable $x$, and since $\sigma\in W^{k,\infty}(\rr{})$ whereas the windows $\varphi$ and $\phi$ are smooth, the function $\tilde{\rho}$ admits weak derivatives with respect to $x$ up to order $n$, provided that $n\le k$. If, in addition, $k>n$, then $\tilde{\rho}$ is in fact classically differentiable with respect to $x$ up to order $n$.
    Hence, for any $\alpha\in \zzp{d}$ such that $\abs{\alpha}\le n$, and $C_{\alpha, \beta, \gamma}=\frac{\alpha!}{\beta!\gamma!(\alpha - \beta - \gamma)!}$, we have
    \begin{align*}
        \norm
        {\partial^\alpha&\tilde{\rho}\left(\cdot, y, \eta, b\right)}{L^{r}(\Omega)}
        % \kern-10em&\kern10em=
        =
        \frac{1}{\vartheta(\eta, b)}
        \norm{\partial^\alpha\left(\sigma\left(\aff_{\eta,b}(\cdo)\right)
        \varphi(\aff_{\eta,b}\left(\cdo\right)-t)\phi(\cdo-y)\right)}
        {L^{r}(\Omega)}
        \\
        &\le
        \frac{1}{\vartheta(\eta, b)}
        \sum _{\beta +\gamma\le \alpha}
        C_{\alpha, \beta, \gamma}
        \frac{\abs{\eta}^{\abs{\alpha - \beta - \gamma}+|\beta|}}{|\tau|^{\abs{^{\alpha - \beta - \gamma}}+|\beta|}}
        \norm{\sigma^{(|\alpha-\beta-\gamma|)}\left(\aff_{\eta, b}\left(\cdo\right)\right)
       \varphi ^{(|\beta|)}\left(\aff_{\eta, b}\left(\cdo\right)-t\right)\partial^\gamma\phi(\cdo -y)}
        {L^{r}(\Omega)}
        \\
        &\le
        \frac{1}{\vartheta(\eta, b)}
        \sum _{\beta +\gamma\le \alpha}C_{\alpha, \beta, \gamma}c_\gamma
        \frac{\abs{\eta}^{\abs{\alpha}-\abs{\gamma}}}{|\tau|^{\abs{\alpha} - \abs{\gamma}}}
        \norm{\sigma^{(|\alpha-\beta-\gamma|)}\left(\aff_{\eta, b}\left(\cdo\right)\right)
       \varphi ^{(|\beta|)}\left(\aff_{\eta, b}\left(\cdo\right)-t\right)}
        {L^{r}(\Omega)},
    \end{align*}
    the previous holds true as $\phi\in \mathscr{S}(\rrd)$, and thus for any $\gamma \in \zzp{d}$, it follows that
    \[
        \norm{\partial^\gamma\phi}{L^\infty(\rrd)}\le c_\gamma.
    \]
    Since $s<-1$, and  the estimate 
    \[
        \abs{\aff_{\eta, b}(x)}\geq \left(\abs{b}-R_\Omega\abs{\frac{\eta}{\tau}}\right)_+,
    \]
    holds, we have
    \[
       v_{-s}(\aff_{\eta, b}(\cdo) )
       \geq
       v_{-s}\left((\abs{b}-R_\Omega\abs{\frac{\eta}{\tau}})_+\right).
    \]
    Moreover, giving that $\gamma \le \alpha$, using the following elementary bounds
    \begin{align*}
        \frac{\abs{\eta}^{\abs{\alpha}-\abs{\gamma}}}{v_n(\eta)}
        \le 1
        \quad\text{ and }
        \abs{\tau}^{-(\abs{\alpha} - \abs{\gamma})}
        \le \left(\frac{1 + \abs{\tau}}{\abs{\tau}}\right)^{\abs{\alpha}},
    \end{align*}
    we conclude the upper bound 
    \begin{multline*}
       \norm
        {\partial^\alpha\tilde{\rho}\left(\cdot, y, \eta, b\right)}{L^{r}(\Omega)}
        \\
        \le
          \left(\frac{1 + \abs{\tau}}{\abs{\tau}}\right)^{\abs{\alpha}}
          \sum _{\beta +\gamma\le \alpha}
          c_\gamma
          \norm{\sigma^{(|\alpha-\beta - \gamma|)}}{L^{\infty}(\rr{})}
          \norm{
        \frac{\varphi^{(|\beta|)}\left(\aff_{\eta, b}\left(\cdo\right)-t)\right)}{v_s(\aff_{\eta, b}(\cdo) )}}
        {L^{r}(\Omega)}.
    \end{multline*}
    Given that $\varphi\in \mathscr{S}(\rr{})$, it is straightforward to verify that 
    \[
        \norm{
        \frac{\varphi^{(|\beta|)}\left(\aff_{\eta, b}\left(\cdo\right)-t)\right)}{v_s(\aff_{\eta, b}(\cdo) )}}
        {L^{r}(\Omega)}
        \le
        \abs{\Omega}^{\frac{1}{r}}C_{s, \beta},
    \]
    holds for any $\eta\in \rrd, b\in \rr{}$ and $\abs{\beta}\le n$, where $C_{s,\beta}>0$ depends on $s$ and $\beta$. Putting everything together, we get
    \begin{align*}
          \norm
        {\partial^\alpha\tilde{\rho}\left(\cdot, y,  \eta, b\right)}{L^{r}(\Omega)}
        &\le
        \abs{\Omega}^{\frac{1}{r}}
         \left(\frac{1 + \abs{\tau}}{\abs{\tau}}\right)^{\abs{\alpha}}
        \sum _{\beta + \gamma \le \alpha}
        c_\gamma
        C_{s, \beta}
        \norm{\sigma^{(|\alpha- \beta - \gamma|)}}{L^\infty(\rr{})}
        \\
        &\le
        \abs{\Omega}^{\frac{1}{r}}
        \norm{\sigma}{W^{k, \infty}(\rr{})}
        \left(\frac{1 + \abs{\tau}}{\abs{\tau}}\right)^{n}
        \sum _{\beta + \gamma \le \alpha}
        C_{\alpha, \beta, \gamma}c_\gamma
        C_{s, \beta}.         
    \end{align*}
    Thus, we conclude that
    \begin{align*}
        \norm{\tilde{\rho}(\cdo, y, \eta, b)}{W^{n,r}(\Omega)}
        % \kern-5em\kern5em
        &=
        \Bigg(\sum_{\abs{\alpha}\leq n}
            \norm{\partial^\alpha
                \tilde{\rho}(\cdo, y, \eta, b)
                }{L^{r}(\Omega)}^r
        \Bigg)^{\frac{1}{r}}\\
        &\le
        \abs{\Omega}^{\frac{1}{r}}
        \norm{\sigma}{W^{k, \infty}(\rr{})}
        \left(\frac{1 + \abs{\tau}}{\abs{\tau}}\right)^{n}
        \left(\sum_{\abs{\alpha}\leq n}
        \left( \sum _{\beta + \gamma \le \alpha}
        c_\gamma
        C_{s, \beta}
        \right)^{r}
        \right)^{\frac{1}{r}}.
    \end{align*}
    The previous quantity is finite for any fixed $t$ and $\tau\neq 0$, and uniformly bounded for any $\eta \in\rrd$ and $b\in\rr{}$.
    Hence, we conclude that the weighted dictionary $\tilde{\bb{D}}$ is uniformly bounded in $W^{n, r}(\Omega)$.
    
    Finally, by selecting $r\geq 2$ and  $n\in \zzp{}$, it follows that $W^{n,r}(\Omega)$ is a type-2 Banach space, see \cite[Corollary A.6]{Brzezniak95StochasticPartialDifferential}.
    Furthermore, the previous step clearly shows that $\tilde{\bb{D}} \subset W^{n,r}(\Omega)$ and that the dictionary $ \tilde{\bb{D}}$ is uniformly bounded in $W^{n,r}(\Omega)$, that is
    $$
    K_{ \tilde{\bb{D}}} := \sup_{h\in \tilde{\bb{D}}}\norm{h}{W^{n, r}(\Omega)}
    \equiv 
    \sup_{y, \eta, b}\norm{\tilde{\rho}(\cdo, y, \eta, b)}{W^{n,r}(\Omega)}
    < \infty.
    $$
    Since $\mathscr{S}$ is dense in $M^{p,q}_m$, $p,q<\infty$,
    the estimate in \cref{eq:FL_q_tilde_inclusion_II} places $f$ in the variation space $ K_{ \tilde{\bb{D}}}$ with a finite variation norm $\norm{f}{ K_{ \tilde{\bb{D}}}}$.
    Applying Maurey's approximation bound (see \cref{prop:approximation_type2}), with $M_f = \norm{f}{ K_{ \tilde{\bb{D}}}}$, we obtain the following estimate:
    \begin{align*}
        \inf_{f_N\in\Sigma_{N, M_f}(\tilde{\bb{D}})}
        \norm{f-f_N}{W^{n,r}(\Omega)}
        &\leq
        4C_{2, W^{n,r}(\Omega)}  K_{ \tilde{\bb{D}}}N^{-\frac{1}{2}}\norm{f}{ K_{ \tilde{\bb{D}}}},
        \\
        &\leq 
        4C_{2, W^{n,r}(\Omega)}  K_{ \tilde{\bb{D}}}N^{-\frac{1}{2}}C\norm{f}{ M^{p,q}_{m}}.
    \end{align*}
    To complete the proof, we observe the inclusion
    \[
    \Sigma_{N,M_f}(\tilde{\bb{D}})\subseteq\Sigma_{N}(\bb{D}),
    \]
    holds by construction. Consequently, the approximation error over $\Sigma_{N}(\bb{D})$ admits the upper bound
    \begin{align*}
        \inf_{f_N\in\Sigma_{N}(\bb{D})}
        \norm{f-f_N}{W^{n,r}(\Omega)}
        &\leq
        \inf_{f_N\in\Sigma_{N,M_f}(\tilde{\bb{D}})}
        \norm{f-f_N}{W^{n,r}(\Omega)}
        \\
        &\leq
        4C_{2, W^{n,r}(\Omega)}  K_{ \tilde{\bb{D}}}N^{-\frac{1}{2}}C\norm{f}{ M^{p,q}_{m}}.
    \end{align*}
    This establishes the claimed approximation bound.
\end{proof}

\begin{remark}
    Note that \cref{thm:approximation_sobolev_space_local} holds in particular  when 
    $$
        s_1=\frac{d+1}{p'},\, s_2=n + 1 +\frac {d+1}{q'}.
    $$
    Furthermore, in \cref{eq:approximation_bound} we have full control over the constant, including its exact dependence on the relevant parameters as shown in the proof of \cref{thm:approximation_sobolev_space_local}.
\end{remark}

We highlight that  \cref{thm:approximation_sobolev_space_local} for
$p=q=1$ gives the approximation result for the weighted  Feichtinger algebra $M^1_m(\rd)$ as follows: 
\begin{corollary}[Local Approximation for Feichtinger's Algebra]
\label{cor:Feichtinger_approximation_sobolev_space_local} 
  Under the assumptions of  \cref{thm:approximation_sobolev_space_local}, with $m(y,\eta) =(1\otimes v_{n+1})(y,\eta)=v_{n+1}(\eta)$ and
for every $f\in M^{1}_m(\rr{d})$,
    there exists a constant $C>0$ such that
\begin{align*}%\label{eq:Feichtinger_approximation_bound_L}
        \inf_{f_{N}\in\Sigma_{N}(\bb{D})}\norm{f-f_N}{W^{n,r}(\Omega)}
        &\leq C N^{-\frac{1}{2}}
        \norm{f}{M^{1}_m(\rr{d})},
\end{align*}
    for all $N\in\nn{}$.
\end{corollary}
A special example of weighted modulation space is the Shubin-Sobolev space ${Q}^{s}$. 

\begin{corollary}[Local Approximation for Sobolev and Shubin-Sobolev Spaces]
\label{cor:Hilbert_approximation_sobolev_space_local} 
     Consider $n\in \zzp{}$, $ r\geq 2$, and a bounded domain $\Omega\subset\rr d$. Under the dictionary  assumptions of \cref{thm:approximation_sobolev_space_local},  for any $$s_1>\frac d2,\qquad s_2>n+1+\frac d2,$$ 
     we have
     \begin{align}\label{eq:Hilbert_approximation_bound_L}
        \inf_{f_{N}\in\Sigma_{N}(\bb{D})}\norm{f-f_N}{W^{n, r}(\Omega)}
        &\leq  C  N^{-\frac{1}{2}} \begin{cases}\norm{f}{{Q}^{s_2}}\\
\norm{f}{L^{2}_{s_1}}+\norm{f}{\FL{2}_{s_2}}.
\end{cases}
    \end{align}
\end{corollary}
\begin{proof}
    The proof is a consequence of \cref{thm:approximation_sobolev_space_local}, the embedding relations in Theorem \ref{thm:embeddings}, and 
the characterization in Lemma \ref{lem:shubin}. In detail,
$$M^2_{v_{s_1}\otimes v_{s_2}}(\rr{d})\hookrightarrow M^1_{1\otimes v_{n+1}}(\rr{d})$$
if and only if $s_1>d/2$ and $s_2>n+1+d/2$.
\end{proof}
The inequality in \eqref{eq:Hilbert_approximation_bound_L} can be understood as an alternative formulation of the uncertainty principle, where the decay of $f$ and $\widehat{f}$ quantifies the time-frequency concentration.\par
Locally, modulation spaces coincide with Fourier-Lebesgue spaces, so another consequence of Theorem \ref{thm:approximation_sobolev_space_local} is the following. 

%\begin{proposition}[Local Approximation in Weighted $\cF L^q$ Spaces]
\label{prop:supported_approximation_sobolev_space_local} 
  % Consider $n\in \zzp{}$, and a %bounded domain $\Omega\subset\rr %d$. Under the dictionary  assumptions %of %\cref{thm:approximation_sobolev_spac%e_local},  for any $f\in M^{p,q}_{1\otimes v_{s_2}}(\rr d)$, with $0<p<\infty$, $0<q\leq 2\leq r$, 
%   and the index $s_2$ satisfying the %condition in \eqref{eq:indices},
%    there exists a constant $C>0$ such that 
%\begin{align*}
      %  \inf_{f_{N}\in\Sigma_{N}(\bb{D})}\norm{f-f_N}{W^{n, r}(\Omega)}
      %  &\leq C N^{-\frac{1}{2}}\abs{\Omega + \Omega}
   %     \norm{f}{\mathscr{F}L^{q}_{v_{s_2}}},
%    \end{align*}
  %  for all $N\in\nn{}$.
  %  If  $\Omega$ is convex, then there %exists a constant $C>0$ such that
   % \begin{align*}
%        \inf_{f_{N}\in\Sigma_{N}%%(\bb{D})}\norm{f-f_N}{W^{n, r}(\Omega)}
   %     &\leq 2C N^{-\frac{1}{2}}\abs{\Omega}
   %     \norm{f}{\mathscr{F}L^{q}_{v_{s_2}}},
%    \end{align*}
%    for all $N\in\nn{}$.
%\end{proposition}
\begin{proof}%
    %The proof is a combination of \cref{thm:approximation_sobolev_space_local} and \cref{cor:FourierLebesgue_embedding_in_modulation_space} as well as \cref{cor:FourierLebesgue_embedding_in_modulation_space_convex}.  
    
\end{proof}

A particular instance of Fourier-Lebesgue space for $p=1$ is the Barron space, cf. equality \eqref{eq:Barron-spaces} above.
One can then restate Proposition \ref{prop:supported_approximation_sobolev_space_local} for this case: 
\begin{corollary}[Local Approximation in Barron Spaces]\label{cor:Barron}
    Consider $n\in \zzp{}$, and a bounded domain $\Omega\in\rr d$. Under the dictionary  assumptions of \cref{thm:approximation_sobolev_space_local},  for any $r\geq 2$, $f\in B_{v_{n+1}}(\rr d)$,
    there exists a constant $C>0$ such that 
\begin{align*}
        \inf_{f_{N}\in\Sigma_{N}(\bb{D})}\norm{f-f_N}{W^{n, r}(\Omega)}
        &\leq C N^{-\frac{1}{2}}\abs{\Omega + \Omega}
        \norm{f}{B_{v_{n+1}}},
    \end{align*}
    for all $N\in\nn{}$.
    If  $\Omega$ is convex, then there exists a constant $C>0$ such that
    \begin{align*}
        \inf_{f_{N}\in\Sigma_{N}(\bb{D})}\norm{f-f_N}{W^{n, r}(\Omega)}
        &\leq 2C N^{-\frac{1}{2}}\abs{\Omega}
        \norm{f}{B_{v_{n+1}}}
    \end{align*}
    for all $N\in\nn{}$. 
\end{corollary}
\begin{proof}
    It follows from Proposition \ref{prop:supported_approximation_sobolev_space_local} and the equality \eqref{eq:Barron-spaces}.   
\end{proof}
\begin{remark}
   (1) Corollary \ref{cor:Barron} generalizes the result by Siegel and Xu in \cite{Siegel20ApproximationRatesNeural} in two directions: by extending the approximation to any dimension $d\geq1$, and by considering the more general class of Sobolev spaces $W^{n,r}(\Omega)$ instead of   $W^{n,2}(\Omega)=H^n(\Omega)$, cf. Corollary 1 in the aforementioned paper.\\
  % (2) We highlight that the result in \cref{prop:supported_approximation_sobolev_space_local} is closely related to \cite[Theorem 1.4]{Abdeljawad23SpaceTimeApproximationShallow}, but differs in two aspects: first, it does not involve two separate blocks of variables, and second, the right-hand side here is independent of the integrability exponent in the error norm ${W^{n, r}(\Omega)}$, unlike in \cite[Theorem 1.4]{Abdeljawad23SpaceTimeApproximationS%hallow}.
\end{remark}

After establishing approximation results for functions in weighted modulation spaces $ M^{p,q}_m$ by means of shallow neural networks $f_N\in \Sigma_N(\bb{D})$ with error norm $W^{n,r}(\Omega)$ measured on a bounded domain$\Omega$, we now turn to the unbounded domain case. Unlike the previous case, where boundedness of the domain simplifies the control of the approximation errors, working on the whole space $\rd$ requires additional care.

\begin{theorem}[Global Approximation]\label{thm:approximation_sobolev_space_G} 
     Consider $n\in \zzp{}$, $0<p,q<\infty$, $r\geq 2$, and an activation function
    $\sigma\in W^{k, \infty}(\rr{})\setminus\{0\}$ (with $k\geq n$). Fix  a bounded domain $\Omega\subset \rrd$ and define the dictionary 
    $\bb{D}_\Omega $ 
    as follows:
    \begin{equation}\label{eq:D-Omega}
         \bb{D}_\Omega = \{ x \mapsto \sigma\left(\tfrac{\eta\cdot x}{\tau}+b\right)
        \varphi(\tfrac{\eta\cdot x}{\tau}+b-t)\phi(x-y) \text{ such that }(y,\eta, b) \in\Omega\times\rr{d}\times \rr{}\},
    \end{equation}
    with $t,\tau$ satisfying \textnormal{\textbf{Condition (A)}}.
    Consider  the weight $m =(v_{s_1}\otimes v_{s_2})$ with $s_1,s_2$ satisfying \eqref{eq:indices}. 
    Then, for every $f\in M^{p,q}_m(\rd)$,
    there exists a constant $C>0$ such that
    \begin{align*}%\label{eq:approximation_bound_G}
        \inf_{f_{N}\in\Sigma_{N}(\bb{D}_\Omega)}\norm{f-f_N}{W^{n,r}(\rd)}
        &\leq C N^{-\frac{1}{2}}
        \norm{f}{M^{p,q}_m(\rd)},
    \end{align*}
    for all $N\in\nn{}$.
\end{theorem}

\begin{proof}
    Analogously to the proof of \cref{thm:approximation_sobolev_space_local}, we first apply \cref{eq:phase_identity} for nontrivial window function $\varphi\in \mathscr{S}(\rr{})$ and subsequently express $f$ in the following integral form:
    \begin{equation}\label{eq:f_representation}
                f(x)
                =
                C_{\sigma ,\varphi}\int_{\rr{2d}}\int_{\rr{}}
                % \sigma\left(\aff_{\eta, b}\left(x\right)\right)
                % \varphi(\aff_{\eta, b}\left(x\right)-t)
                 \rho(x, y, \eta, b)
                \,
                e^{-2\pi ib\cdot \tau}
                \,
                V_\phi{f}(y, \eta)
                 \,db\,dy\,d\eta,
    \end{equation}
    where $f$ and $\phi$ belong to $\mathscr{S}(\rd)\setminus\{0\}$ such that $\phi$ is a positive function with $\norm{\phi}{L^2}=1$. Furthermore, using \textnormal{\textbf{Condition (A)}},
    \begin{align*}
                C_{\sigma ,\varphi} 
                &=               \left|(V_\varphi \sigma)(t,\tau ) \right|^{-1}
                 \\[.5em]
                \rho(x, y, \eta, b) &=  \sigma\left(\aff_{\eta, b}\left(x\right)\right)
                \varphi(\aff_{\eta, b}\left(x\right)-t)\phi(x-y)
                \\[.5em]
                 \aff_{\eta, b}(x) &= \aff_{\tau, \eta, b}(x)= \frac{\scal x\eta}{\tau}+b.
    \end{align*}   
    Since the parameters $t$ and $\tau\neq0$ are fixed constants in $\rr{}$, we suppress them in  our notation.
    Based on the representation of the signal $f$ in \cref{eq:f_representation}, we introduce the dictionary $\bb{D}_\Omega$ as in \eqref{eq:D-Omega}.

    To ensure that the dictionary remains uniformly bounded in $W^{n, r}(\rrd)$ and that the target function $f$ lies in the associated variation spaces, we introduce a suitable weight function in order to control the behavior at infinity and guarantee convergence. Since the domain in the $x$-variable is unbounded, the weight used in the proof of  \cref{thm:approximation_sobolev_space_local} is no longer applicable. Instead, we define the weight $\vartheta$ as follows:
    \[
           \vartheta(\eta, b) = \frac{v_{n + s}(\eta)}{v_s(b)}, \text{ such that } s>1.
    \]
    Accordingly, the modified dictionary $ \tilde{\bb{D}}_\Omega$ takes the form:
    \begin{align*}
                \tilde{\bb{D}}_\Omega&=\left\{
                   \frac{v_s(b)}{v_{n +s}(\eta)} \rho(\cdo, y, \eta, b)
                 : \rrd \to \rr{} \big|\; (y, \eta, b)\in \Omega\times\rrd \times \rr{}
                 \right\}
    \intertext{and the associated measure is given by} 
                d \mu_f(y, \eta,b)
                &=
                C_{\sigma ,\varphi}\frac{v_{n + s}(\eta)}{v_s(b)} 
                e^{-i {b}\cdot \tau}
                    V_\phi{f}(y, \eta)
                     \,db\,dy\,d\eta,
    \end{align*}
    which in turn allows us to represent $f$ in the integral form containing all the required components:
    \[
        f=\int_{\tilde{\bb{D}}_\Omega} i_{\tilde{\bb{D} }_\Omega\rightarrow W^{n,r}(\rrd)} d \mu_f.
    \]
    In order to derive an upper bound on the  variation norm of  $f$, we recall that
    \[
        \norm{f}{\mathcal{K}(\tilde{\bb{D}}_\Omega)}
        =
        \inf\left\{\norm{\mu}{}:
        f=\int_{\tilde{\bb{D}}_\Omega}i_{\tilde{\bb{D}}_\Omega \to W^{n,r}(\rrd)}\,d \mu\right\},
    \]
    where the infimum is taken over all Borel measures $\mu$ on $\tilde{\bb{D}}_\Omega$. In particular,
    \[
         \norm{f}{\mathcal{K}(\tilde{\bb{D}}_\Omega)}
         \leq 
         \norm{\mu_f}{L^1}.
    \]
    From the previous inequality, we obtain
    \begin{align}\notag
                \norm{f}{\mathcal{K}(\tilde{\bb{D}}_\Omega)}
                &\le
                \int_{\rr{2d}}\int_{\rr{}}
                C_{\sigma ,\varphi}\frac{v_{n + s}(\eta)}{v_{s}(b)} 
                    \abs{V_\phi{f}(y, \eta)}
                     \,db\,dy\,d\eta
                     \\[.5em]\notag
                     &\le 
                     C_{\sigma ,\varphi}
                     \int_{\rr{2d}}\int_{\rr{}}
                    v_{-s}(b)\,db\,
                    v_{n + s}(\eta)
                    \abs{V_\phi{f}(y, \eta)}
                     \,dy\,d\eta\notag
                     \\[.5em]
                     &=
                     C_{\sigma ,\varphi}\,
                     \sqrt{\pi}\frac{\Gamma(\frac{s-1}{2})}{\Gamma(\frac{s}{2})}
                     \int_{\rr{2d}}
                     v_{n + s}(\eta)
                    \abs{V_\phi{f}(y, \eta)}
                     \,dy\,d\eta.\label{eq:constant}
    \end{align}
   Using the inclusion relations of Theorem \ref{thm:embeddings} we majorize \eqref{eq:constant} as follows:
    \begin{equation}\label{eq:cpq}
    \int_{\rr{2d}}
        v_{n + s}(\eta)
        \abs{V_\phi{f}(y, \eta)}
        \,dy\,d\eta=
        \|f\|_{M^1_{1\otimes v_{n+s}}}
        \le C_{p,q}
        \,
        \|f\|_{M^{p,q}_{v_{s_1}\otimes v_{s_2}}}
    \end{equation}
    where the index $s_1$ satisfies \eqref{eq:indices}, and
$$s_2>n+s+\frac{d}{q'},$$    
    for a suitable constant $C_{p,q}>0$. The arguments above work for any index $s>1$, this allows to extend the range of $s_2$ as in \eqref{eq:indices}, providing to choose $s>1$ accordingly.

For $m=v_{s_1}\otimes v_{s_2}$, we conclude that
        \begin{align}\label{eq:modulation_space_embedding_in_variation_space}
             \norm{f}{\mathcal{K}(\tilde{\bb{D}}_\Omega)}
                &\le
               C_{p,q}\,
                C_{\sigma ,\varphi}\,
                 \sqrt{\pi}\frac{\Gamma(\frac{s-1}{2})}{\Gamma(\frac{s}{2})}\, 
                \norm{f}{M^{p,q}_m}
                =
                C\norm{f}{M^{p,q}_m}.
        \end{align}
        A central task at this stage is to verify that the chosen dictionary is uniformly bounded in the Sobolev space $W^{n,r}(\rd)$. This follows from the properties of the window functions, the activation function, and the weights. In fact, all together ensure that the modified atom 
        \[
             \tilde{\rho}(x, y, \eta, b):= \frac{v_s(b)}{v_{n + s}(\eta)} \rho(x, y, \eta, b)
        \]
        is differentiable with respect to the $x$-variable up to the order $n$. Consequently, for every multi-index $\alpha$ with $\abs{\alpha}\le n$, and $C_{\alpha, \beta, \gamma}=\frac{\alpha!}{\beta!\gamma!(\alpha - \beta - \gamma)!}$, we obtain
         \begin{multline*}
            \norm{\partial^\alpha\tilde{\rho}\left(\cdo, y, \eta, b\right)}{L^{r}(\rrd)}
                =
            \frac{v_s(b)}{v_{n + s}(\eta)} 
            \norm{\partial^\alpha\left(\sigma\left(\aff_{\eta, b}\left(\cdo\right)\right)
            \varphi(\aff_{\eta, b}\left(\cdo\right)-t)\phi(\cdo-y)\right)}{L^{r}(\rrd)}
            \\
            \le 
            \frac{v_s(b)}{v_{n + s}(\eta)} 
            \sum _{\beta +\gamma\le \alpha}
            C_{\alpha, \beta, \gamma}
            \frac{\abs{\eta}^{|\alpha -\gamma|}}{|\tau|^{\abs{\alpha - \gamma}}}
            \norm{\sigma^{(|\alpha - \beta - \gamma|)}\left(\aff_{\eta, b}\left(\cdo\right)\right)
            \varphi ^{(|\beta|)}\left(\aff_{\eta, b}\left(\cdo\right)-t\right)\partial^\gamma \phi(\cdo - y)}{L^{r}(\rrd)}
            \\
            \le
            \frac{v_s(b)}{v_s(\eta)}
            \left(\frac{1+\abs{\tau}}{\abs{\tau}}\right)^n
            \sum _{\beta + \gamma \le \alpha}
            C_{\alpha, \beta, \gamma}
            \frac{\abs{\eta}^{|\alpha -\gamma|}}{v_n(\eta)}
            \norm{\sigma^{(|\alpha-\beta - \gamma|)}}{L^\infty(\rr{})}
            \norm{\varphi ^{(|\beta|)}\left(\aff_{\eta, b}\left(\cdo\right)-t\right)\partial^\gamma \phi(\cdo -y)}{L^r(\rrd)}.
        \end{multline*}
        Furthermore, we have
        \begin{multline*}
            \norm{\varphi ^{(|\beta|)}\left(\aff_{\eta, b}\left(\cdo\right)-t\right)\partial^\gamma \phi(\cdo -y)}{L^r(\rrd)}
            \\
            \leq
            \norm{v_{-u}(\cdo)
               \varphi ^{(|\beta|)}\left(\aff_{\eta, b}\left(\cdo\right)-t\right)}
                {L^{r}(\rrd)}
                \norm{v_u(\cdo)\partial^\gamma \phi(\cdo -y)}{L^\infty(\rrd)}.
        \end{multline*}
        Since $\varphi\in \mathscr{S}(\rr{})$, we have $\varphi v_k\in W^{\ell, p}(\rr{})$, for every $k,\ell,p\in \nn{}$. Applying   \cite[Lemma 32]{Abdeljawad24WeightedApproximationBarron} with parameters $\ell=0$, $p=r$, $k\geq s$, and $u>s$,  we obtain
        \begin{align*}
            \norm{v_{-u}
               \varphi ^{(|\beta|)}\left(\aff_{\eta, b}\left(\cdo\right)-t\right)}
                {L^{\infty}}
                &=
                 \norm{v_{-u}
               \varphi^{(|\beta|)} \left(\aff_{\eta, b}\left(\cdo\right)-t\right)}
                {W^{0, \infty}}
                \\
                &\le
                C_{\beta, \tau, d}
                v_{-s}(\min\{1, \abs{\tau}/\eta \}\abs{b}).
        \end{align*}
        This implies that
        \begin{equation}\label{eq:window_upper_bound}
                \frac{v_s(b)}{v_{s}(\eta)}
               \norm{v_{-u}
               \varphi ^{(|\beta|)}\left(\aff_{\eta, b}\left(\cdo\right)-t\right)}
                {L^{\infty}}
                \leq
                C_{\beta, \tau, d}
                 \frac{v_s(b)}{v_s(\eta)}
                 v_{-s}(\min\{1, \abs{\tau}/\eta \}\abs{b}).
        \end{equation}
        In order to establish a uniform upper bound for \cref{eq:window_upper_bound}, we distinguish two cases according to the relation between $\eta$ and $\tau$.
        \begin{itemize}
            \item Case 1: If $\abs{\eta}<\abs{\tau}$, then
                \begin{align*}
                    v_{-s}(\eta)
                    v_s(b)
                    v_{-s}(\min\{1,\abs{\tau}/\abs{\eta}\}\abs{b})
                    =
                    v_{-s}(\eta),
                \end{align*}
                which uniformly bounded.
            \item Case 2: If $\abs{\eta}\geq\abs{\tau}$, then
                \begin{align*}
                    v_{-s}(\eta)
                    v_s(b)
                    v_{-s}(\min\{1,\abs{\tau}/\abs{\eta}\}\abs{b})
                    \le
                    \frac
                        {v_s(b)}
                        {(\abs{\eta}+\abs{\tau}\abs{b})^s},
                \end{align*}
                which is uniformly bounded by $\abs{\tau}^{-s}$.
        \end{itemize}
        Combining both cases,  we conclude that
        \[
            \frac{v_s(b)}{v_s(\eta)}
                \norm{ v_{-u}
               \varphi ^{(|\beta|)}\left(\aff_{\eta, b}\left(\cdo\right)-t\right)}
                {L^{\infty}}
                \le
                C_{\beta, \tau, d}
                \min\{1, \abs{\tau}^{-s} \}.
        \]
        At this step, in order to obtain a uniform upper bound, it is necessary to assume that the set $\Omega\subset \rrd$ is bounded.
        Consequently, we get 
        \[
            \norm{v_{u}(\cdo)\partial^\gamma \phi(\cdo -y)}{L^\infty(\rrd)}
            \le
            C_{\gamma, u, \Omega}.
        \]
        As a consequence, we obtain
        \[
            \norm{\partial^\alpha\tilde{\rho}\left(\cdot, y, \eta, b\right)}{L^{r}(\rrd)}
            \le
            \norm{\sigma}{W^{k, \infty}(\rr{})}
                \left(\frac{1+\abs{\tau}}{\abs{\tau}}\right)^n
                \min\{1, \abs{\tau}^{-s} \}
                \sum _{\beta +\gamma \le \alpha}
                C_{\alpha, \beta, \gamma}
                C_{\gamma, u, \Omega}
                C_{\beta, \tau, d},
        \]
        where we used the fact that $\abs{\eta}^{\abs{\alpha - \gamma}}v_{-n}(\eta) \le 1$ and that $C_{\gamma, u, \Omega}, C_{\alpha, \beta, \gamma}$ and $C_{\beta, \tau, d}$ are positive constants.
        Similar to the proof of \cref{thm:approximation_sobolev_space_local}, a simple count of partial derivatives up to order $n$ yields the boundedness of the atoms $\tilde{\rho}$ in the Sobolev norm $W^{n, r}(\rrd)$. This, in turn, implies the uniform boundedness of the weighted dictionary $\tilde{\bb{D}}$ since the right-hand side of the preceding estimate is independent of $\eta$, $y$ and $b$.

        With the uniform boundedness of the dictionary in $W^{n,r}(\rd)$ established, we are now prepared to present the final bound.
        Specifically, for \( r \geq 2 \) and \( n \in \zzp{} \), the Sobolev space \( W^{n,r}(\rrd) \) is a type-2 Banach space; see \cite[Corollary A.6]{Brzezniak95StochasticPartialDifferential}. As shown earlier, \( \tilde{\bb{D}} _\Omega\subset W^{n,r}(\rrd) \) and is uniformly bounded, namely,
        \[
            K_{\tilde{\bb{D}}_\Omega} := \sup_{h \in \tilde{\bb{D}}_\Omega} \|h\|_{W^{n,r}(\rrd)} 
                = \sup \|\tilde{\rho}(\cdo, y, \eta, b)\|_{W^{n,r}(\rrd)} < \infty,
        \]
        where the supremum is taken over $y\in \Omega\subset\rrd, \eta\in \rrd, b\in\rr{}$. By \eqref{eq:cpq} we obtain in particular that \( f \in M^{1}_m(\rrd) \) with $m=1\otimes v_{n+s}$. Then, $f$ belongs to $W^{n,r}(\rrd)$ (see \cref{prop:embedding_in_fractional_Sobolev_spaces}). Furthermore, the embedding in \cref{eq:modulation_space_embedding_in_variation_space} implies that \( f \in K_{\tilde{\bb{D}}_\Omega} \) with \( M_f := \|f\|_{K_{\tilde{\bb{D}}_\Omega}} \). Applying Maurey's bound (see \cref{prop:approximation_type2}), we obtain
        \begin{align*}
            \inf_{f_N \in \Sigma_{N, M_f}(\tilde{\bb{D}}_\Omega)} \|f - f_N\|_{W^{n,r}(\rrd)} 
            &\leq 4C_{2, W^{n,r}(\rrd)} K_{\tilde{\bb{D}}_\Omega} N^{-1/2} \|f\|_{K_{\tilde{\bb{D}}_\Omega}} 
            \\
            &\leq
            4C_{2, W^{n,r}(\rrd)} K_{\tilde{\bb{D}}_\Omega} N^{-1/2} C \|f\|_{M^{p,q}_m(\rrd)}.
        \end{align*}
        Since \( \Sigma_{N, M_f}(\tilde{\bb{D}}_\Omega) \subseteq \Sigma_{N}(\bb{D}_\Omega) \), the same estimate carries over:
        \[
            \inf_{f_N \in \Sigma_{N}(\bb{D}_\Omega)} \|f - f_N\|_{W^{n,r}(\rrd)} 
            \leq 4C\,C_{2, W^{n,r}(\rrd)}\,
            K_{\tilde{\bb{D}}_\Omega}\, N^{-1/2} \,
            \|f\|_{M^{p,q}_m(\rrd)}.
        \]
        This concludes the proof.
\end{proof}
\begin{remark}
    The uniform constant $C_{p,q}$ in \eqref{eq:cpq} follows from the inclusion relations for modulation spaces in Theorem \ref{thm:embeddings}. Note that this allows to have  indices $0<p,q<\infty$. Of course small indices $p,q$ come at the expenses of bigger weights $v_{s_1}$ and $v_{s_2}$.  
    To obtain an explicit expression of $C_{p,q}$  one can employ Jensen's inequality for a smaller range of indices $p,q\geq 1$. We leave the details to the interested reader.
\end{remark}
Theorem \ref{thm:approximation_sobolev_space_G} for $p=q=1$ gives the global approximation for the weighted Feichtinger algebra:
\begin{corollary}[Global Approximation for Feichtinger's Algebra]
 Consider $ n\in \zzp{}$,  $r\geq 2$,  the dictionary and the activation function as in \cref{thm:approximation_sobolev_space_G}.
    If $m =(1\otimes v_{s_2})$ with
    $$
       s_2>n +1 ,
    $$
    then, for every $f\in M^{1}_m(\rr{d})$,
    there exists a constant $C>0$ such that
\begin{align*}%\label{eq:Feichtinger_approximation_bound_G}
        \inf_{f_{N}\in\Sigma_{N}(\bb{D}_\Omega)}\norm{f-f_N}{W^{n,r}(\rr{d})}
        &\leq C N^{-\frac{1}{2}}
        \norm{f}{M^{1}_m(\rr{d})},
\end{align*}
    for all $N\in\nn{}$.
\end{corollary}

What has been done so far can be applied to Potential Sobolev spaces  $W^{s,r}$ defined in Subsection \ref{subsec:sobolev}. 

\begin{corollary}[Global Approximation for Potential Sobolev spaces]
Assume the hyphotheses of Theorem \ref{thm:approximation_sobolev_space_G} with the integer $n\in \zzp{}$ replaced by $s\in\rr{}_+$,  and the weight index $s_2$ in \eqref{eq:indices} satisfying $$s_2=\lfloor s \rfloor+2\quad\mbox{if}\quad 0<q\leq1,\quad s_2>\lfloor s \rfloor
+2+\frac{d}{q'} \quad\mbox{if}\quad q\geq1.$$ 
Then, for every $f\in M^{p,q}_m(\rr{d})$, there exists a constant $C>0$ such that
\begin{align*}%\label{eq:Feichtinger_approximation_bound_G}
        \inf_{f_{N}\in\Sigma_{N}(\bb{D}_\Omega)}\norm{f-f_N}{W^{s,r}(\rr{d})}
        &\leq C N^{-\frac{1}{2}}
        \norm{f}{M^{p,q}_m(\rr{d})},
\end{align*}
    for all $N\in\nn{}$.
\end{corollary}
\begin{proof}
    Using the inclusion relations:  $$W^{n,r}(\rr{d})\hookrightarrow W^{s,r}(\rr{d}),$$ for $n\geq s,$ 
the claim follows.
\end{proof}

%%%%%%%%%%%%%%%%%%%%%%%%%%%%%%%%%%%%%%%%%%%%%%%%%%%%%%%%%%%%%%%%%%
\section{Experiments: Function Approximation with Modulation Dictionary}
\label{sec:experiments}
%%%%%%%%%%%%%%%%%%%%%%%%%%%%%%%%%%%%%%%%%%%%%%%%%%%%%%%%%%%%%%%%%%

%%%%%%%%%%%%%%%%%%%%%%%%%%%%%%%%%%%%%%%%%%%%%%%%%%%%%%%%%%%%%%%%%%
\paragraph{Activation Function Strategy.}
While our theoretical framework assumes $\sigma \in W^{k, \infty}(\rr{})$. We employ the standard ReLU activation in our numerical experiments.% for optimization efficiency (to avoid vanishing gradients).
This is justified by the fact that  a ``ramp'' or a ``tooth'' profile belong to $W^{1, \infty}(\rr{})$ (see \cref{lem:rampSobolev}) and can be exactly represented by linear combinations of $2$ or $3$ ReLU units, respectively. For instance, a bounded ramp can be decomposed as:
\[
    \sigma_{\text{ramp}}(x) = (x-b_1)_+ - (x-b_2)_+,\quad b_1, b_2 \in \rr{} \text{ such that }b_1< b_2.
\]
While the symmetric tooth function can be characterized as follows:
\[
    \sigma_{\text{tooth}}(x) = (x-b_1)_+ - 2(x-b_2)_+ +(x-b_3)_+,\quad b_1, b_2, b_3 \in \rr{},
\]
such that $b_1<b_2<b_3$ and $b_2-b_1 = b_3-b_2$.
To guarantee sufficient representational capacity, we increase the number of neurons. This ensures that, in the worst-case scenario, the network can recover the bounded activation profiles required by the theory.

We formally establish the Sobolev regularity of the ramp and the tooth profiles in the following lemma.

\begin{lemma}\label{lem:rampSobolev}
Let $x, b_1,b_2,b_3\in\rr{}$ with $b_1<b_2<b_3$, and define the ramp function
\[
\sigma_{\text{ramp}}(x) := (x-b_1)_+ - (x-b_2)_+ .
\]
as well as the symmetric tooth function
\[
 \sigma_{\text{tooth}}(x) = (x-b_1)_+ - 2(x-b_2)_+ +(x-b_3)_+, \text{ such that } b_2-b_1 = b_3-b_2.
\]
Then $\sigma_{\text{ramp}}, \sigma_{\text{tooth}}\in W^{1,\infty}(\rr{})$.
\end{lemma}

\begin{proof}
The ReLU function $(x-b)_+$ is locally absolutely continuous and satisfies
\[
\frac{d}{dx}(x-b)_+ = \mathbf 1_{(b,\infty)}(x)
\quad \text{a.e.}
\]
By linearity of weak derivatives, we obtain
\[
\sigma_{\text{ramp}}'(x)
= \mathbf 1_{(b_1,\infty)}(x)-\mathbf 1_{(b_2,\infty)}(x)
= \mathbf 1_{(b_1,b_2)}(x)
\quad \text{a.e.}
\]
Hence $\sigma_{\text{ramp}}'\in L^\infty(\rr{})$, which implies $\sigma_{\text{ramp}}\in W^{1,\infty}(\rr{})$.  
Furthermore, $\sigma_{\text{ramp}}$ is explicitly given by
\[
\sigma_{\text{ramp}}(x)=
\begin{cases}
0, & x\le b_1,\\
x-b_1, & b_1<x<b_2,\\
b_2-b_1, & x\ge b_2,
\end{cases}
\]
and is therefore bounded, i.e. $\sigma_{\text{ramp}}\in L^\infty(\rr{})$. With a similar technique one can show that $\sigma_{\text{tooth}}\in W^{1,\infty}(\rr{})$.
\end{proof}

We introduce a novel architecture which we term the shallow modulation neural network whose units are taken from the modulation dictionary (see \cref{thm:approximation_sobolev_space_local})  
\[
\bb{D}
  = \Bigl\{
      x \mapsto
      \sigma\!\left(\tfrac{\eta\cdot x}{\tau}+b\right)\,
      \varphi\!\left(\tfrac{\eta \cdot  x}{\tau}+b-t\right)\,
      \phi(x-y)
      \;\Big|\;
      (y,\eta,b)\in\rr{d}\times\rr{d}\times\rr{}
    \Bigr\},
\]
where the constants $\tau, t\neq 0$, $\sigma$  is the $\operatorname{ReLU}$ activation function, so that \textnormal{\textbf{Condition (A)}} is satisfied, cf. Lemma \ref{lem:31} in the Appendix below.  Furthermore, $\varphi,\phi$ are Gaussian windows that provide localization both along the one-dimensional response $\tfrac{\eta\!\cdot\! x}{\tau}+b$ and in the input domain. 
Let $\tau, t \neq 0$.  
For each $x \in \rr{d}$, we define the \emph{modulation atom}
\[
\phi_k(x)
  = \operatorname{ReLU}\!\left(\tfrac{\eta_k \!\cdot\! x}{\tau} + b_k\right)
    \exp\!\Bigl[-\tfrac{1}{2}\Bigl(\tfrac{\eta_k \!\cdot\! x}{\tau} + b_k - t\Bigr)^{2}\Bigr]
    \exp\!\Bigl[-\tfrac{1}{2}\|x - y_k\|_{2}^{2}\Bigr],
  \qquad k \in \nn{}.
\]
The associated network output with $N$ hidden units is then given by
\[
f_{N}(x)
  = \sum_{k=1}^{N} a_k\, \phi_k(x) + c,
  \quad \text{ such that }a_k, c \in \rr{} \text{ where }k\in \{1, \dots, N\}.
\]
This architecture can be interpreted as a shallow neural network whose activation functions
are atoms drawn from the modulation dictionary~$\bb{D}$.

To provide a benchmark, we also consider a plain (vanilla) shallow ReLU network of comparable complexity,
\[
p_{M}(x)
  = \sum_{k=1}^{M} \zeta_k\,
      \operatorname{ReLU}\!\bigl(\omega_k \!\cdot\! x + m_k\bigr)
    + z,
  \quad \text{ such that }\zeta_k, z\in \rr{} \text{ where }k\in \{1, \dots, M\}.
\]

We approximate the target function $f(x)=e^{-x^{2}}\sin(3x)$ in one dimension
as well as a similar two–dimensional extensions $F(x,y) = e^{-(x^{2} + y^{2})}\sin(x + y)$. Each simulation is trained for $100k$ epochs
using both the Adam optimizer (without learning–rate scheduling)
and AdamW equipped with a ReduceLROnPlateau scheduler,
with the following parameters in the one-dimensional and  the two-dimensional cases
\begin{align*}
&\text{factor}=0.9,\qquad
&\text{patience}=100,\qquad
&\text{cooldown}=200,\qquad
&\text{min\_lr}=10^{-8},
\\
&\text{factor}=0.9,\qquad
&\text{patience}=50,\qquad
&\text{cooldown}=100,\qquad
&\text{min\_lr}=10^{-8},
\end{align*}
respectively.

To ensure robustness and reproducibility, each experiment is repeated using ten different random seeds.  
These seeds influence both the data generation process of $10k$ samples and the initialization of the network parameters in the one-dimensional experiments, whereas in the two-dimensional setting only the weight initialization is randomized.  

In the one-dimensional case, the modulation network is implemented with $300$ hidden neurons, while the plain ReLU network employs $400$ neurons so that both architectures contain the same total number of trainable parameters ($1201$).  
For the two-dimensional experiments, the number of hidden units in the plain network is increased to $450$, ensuring again an equal total parameter count ($1801$) between the two architectures.

We compare the two networks in terms of their $H^{1}$–approximation accuracy.
Across the considered benchmarks, Across all benchmarks, the modulation network consistently outperforms the plain network, demonstrating superior convergence in Sobolev norms during training (see \cref{fig:1d_loos_vs_epochs}) and enhanced generalization on unseen data compared to the plain network (see \cref{fig:1d_prediction,fig:2d_prediction,fig:2d_prediction_dx,fig:2d_prediction_dy}), at the cost of a moderately increased runtime.

\begin{figure}[H]
  \centering
  \begin{subcaptiongroup}
    \begin{subfigure}{0.4\textwidth}
      \centering
      \includegraphics[width=\linewidth]{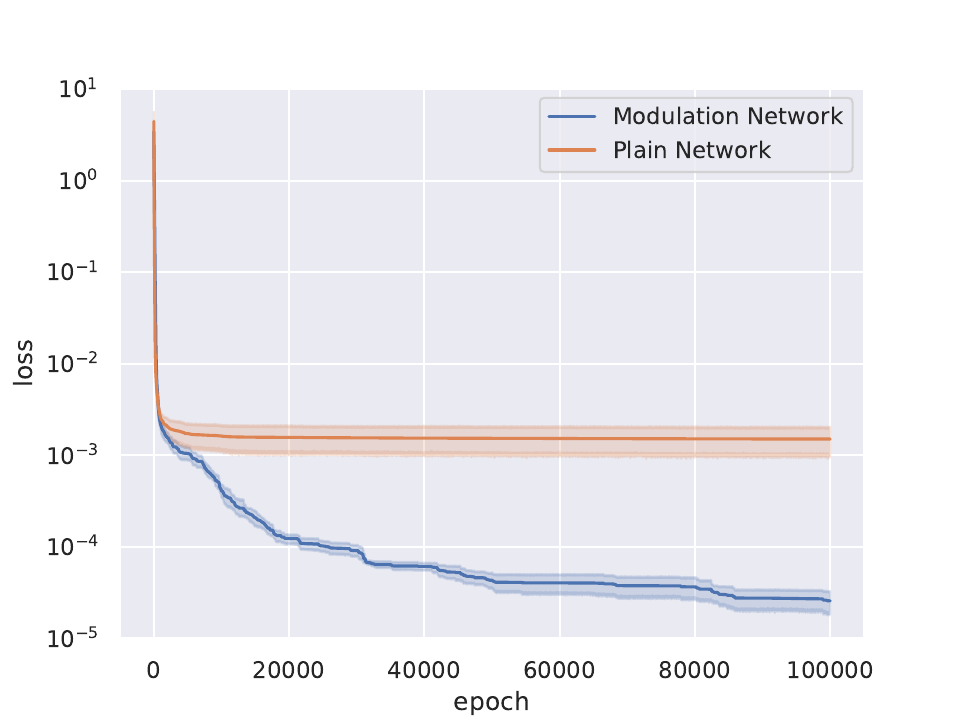}
      \caption{Adam (no scheduler).}
    \end{subfigure}\hfill
    \begin{subfigure}{0.4\textwidth}
      \centering
      \includegraphics[width=\linewidth]{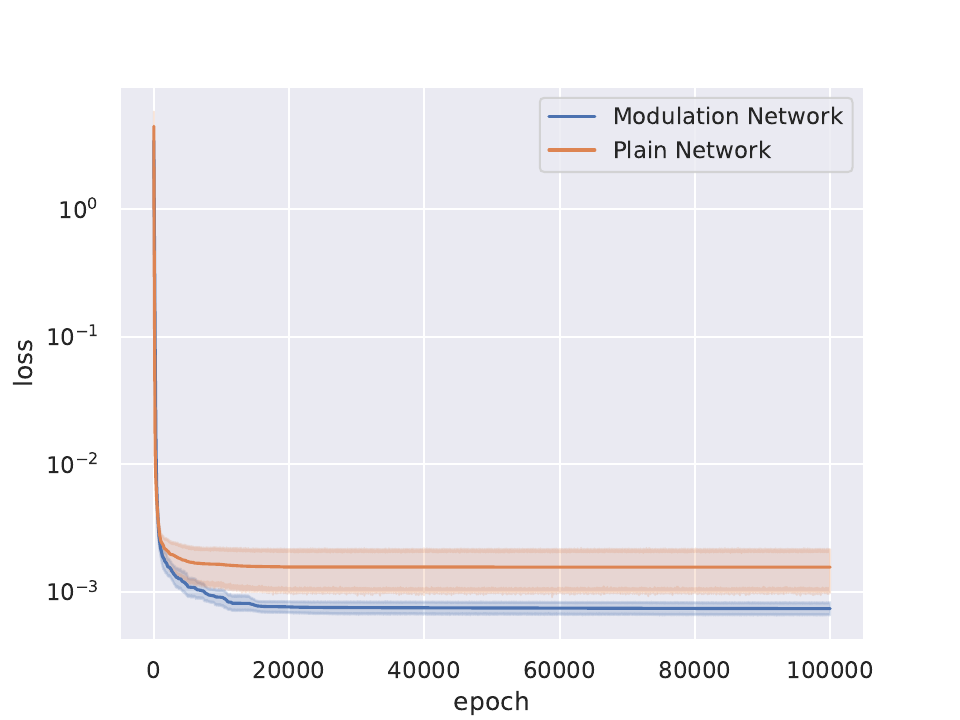}
      \caption{AdamW with ReduceLROnPlateau.}
    \end{subfigure}
  \end{subcaptiongroup}
  \caption{Training loss over epochs for the modulation and plain ReLU networks (1201 parameters each). Curves show the median over 10 seeds with variability bands.}
  \label{fig:1d_loos_vs_epochs}
\end{figure}

\begin{figure}[H]
  \centering
  % --- First row ---
  \begin{subfigure}{0.24\textwidth}
    \centering
    \includegraphics[width=\linewidth]{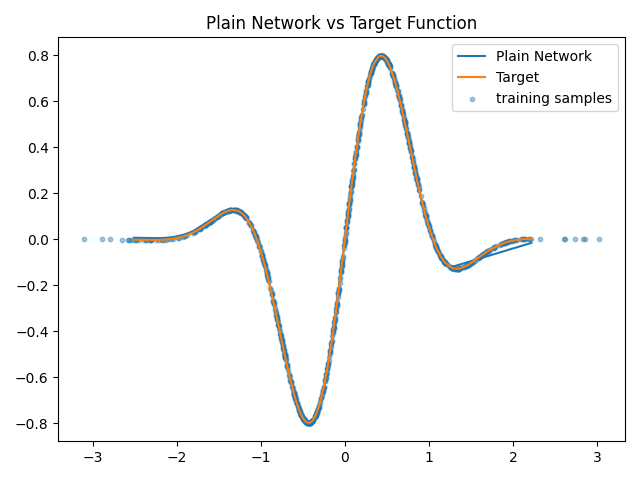}
  \end{subfigure}\hfill
  \begin{subfigure}{0.24\textwidth}
    \centering
    \includegraphics[width=\linewidth]{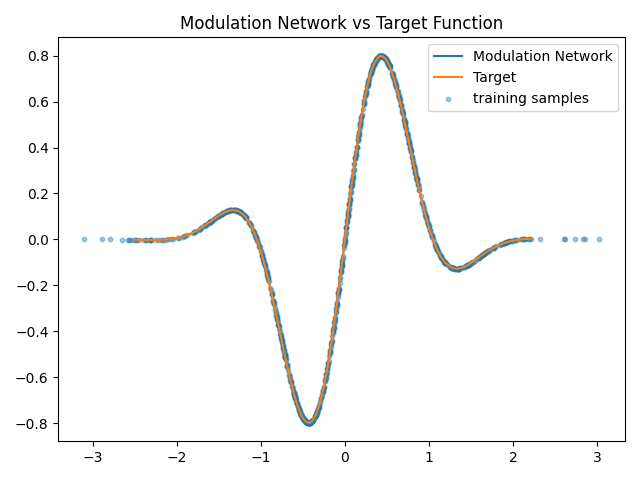}
  \end{subfigure}
  % --- Second row ---
  \begin{subfigure}{0.24\textwidth}
    \centering
    \includegraphics[width=\linewidth]{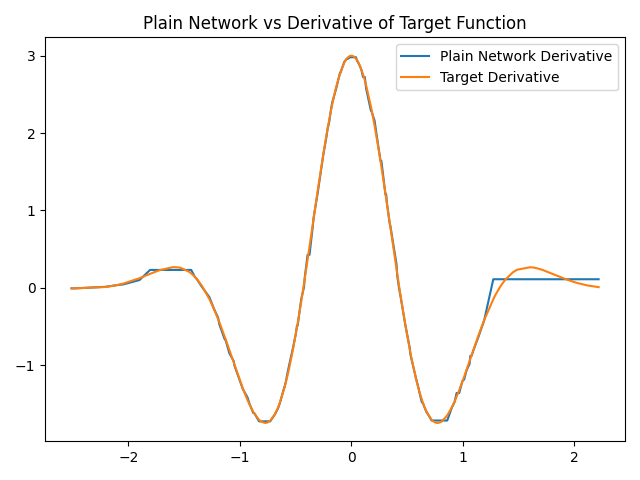}
  \end{subfigure}\hfill
  \begin{subfigure}{0.24\textwidth}
    \centering
    \includegraphics[width=\linewidth]{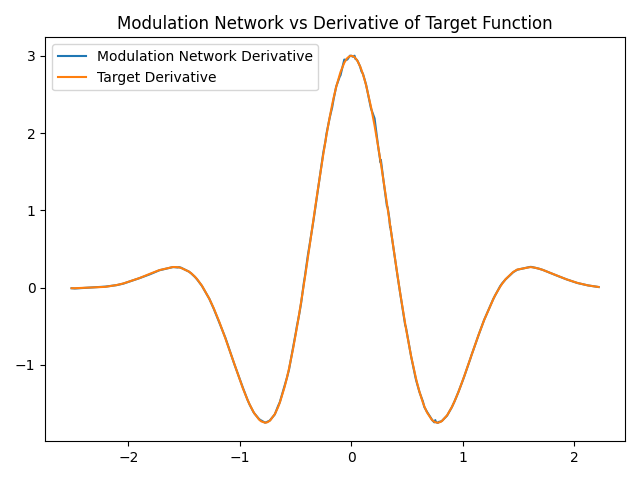}
  \end{subfigure}
  \caption{Comparison of plain and modulation model predictions on unseen one-dimensional data using Adam optimizer. The top row displays the predicted values of the target function $e^{-x^{2}}\sin(3x)$, whereas the bottom row displays the predicted values of its derivative.}
  \label{fig:1d_prediction}
\end{figure}

\begin{figure}[H]
  \centering
  \begin{subcaptiongroup}
    \begin{subfigure}{0.4\textwidth}
      \centering
      \includegraphics[width=\linewidth]{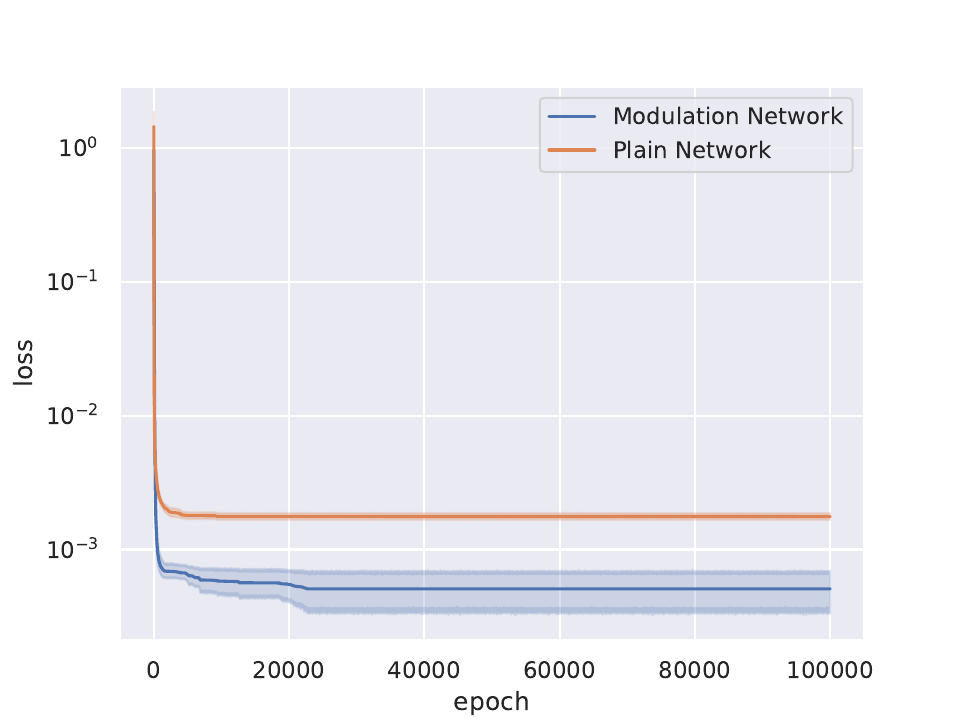}
      \caption{Adam (no scheduler).}
    \end{subfigure}\hfill
    \begin{subfigure}{0.4\textwidth}
      \centering
      \includegraphics[width=\linewidth]{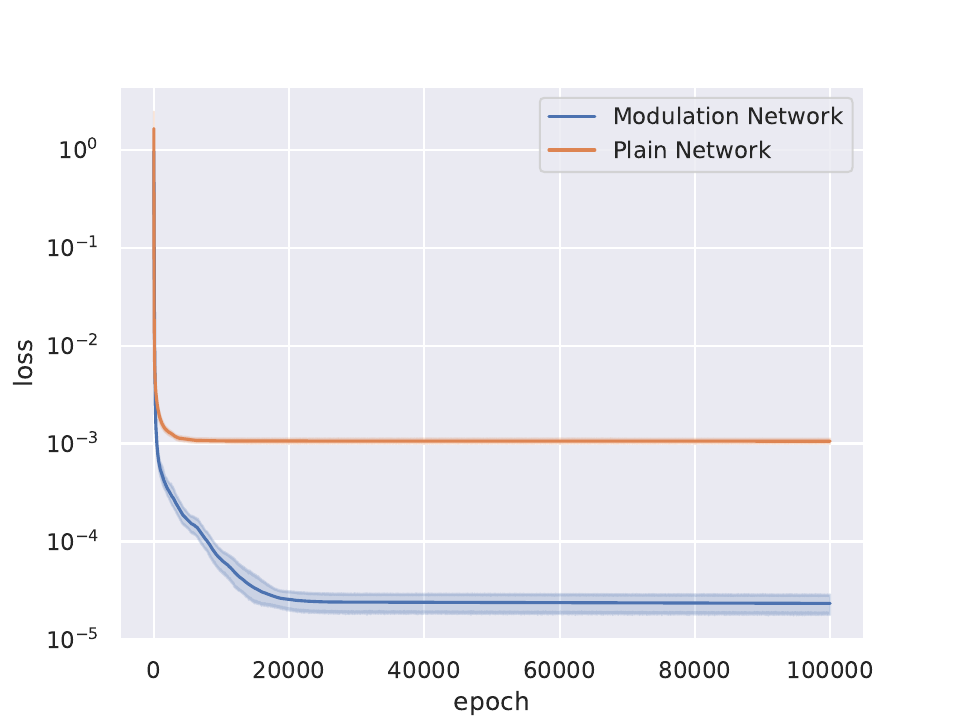}
      \caption{AdamW with ReduceLROnPlateau.}
    \end{subfigure}
  \end{subcaptiongroup}
  \caption{Training loss over epochs for the modulation and plain ReLU networks (1801 parameters each). Curves show the median over 10 seeds with variability bands.}
  \label{fig:2d_loos_vs_epochs}
\end{figure}

\begin{figure}[H]
  \centering
  \begin{subfigure}{0.4\textwidth}
      \centering
      \includegraphics[width=\linewidth]{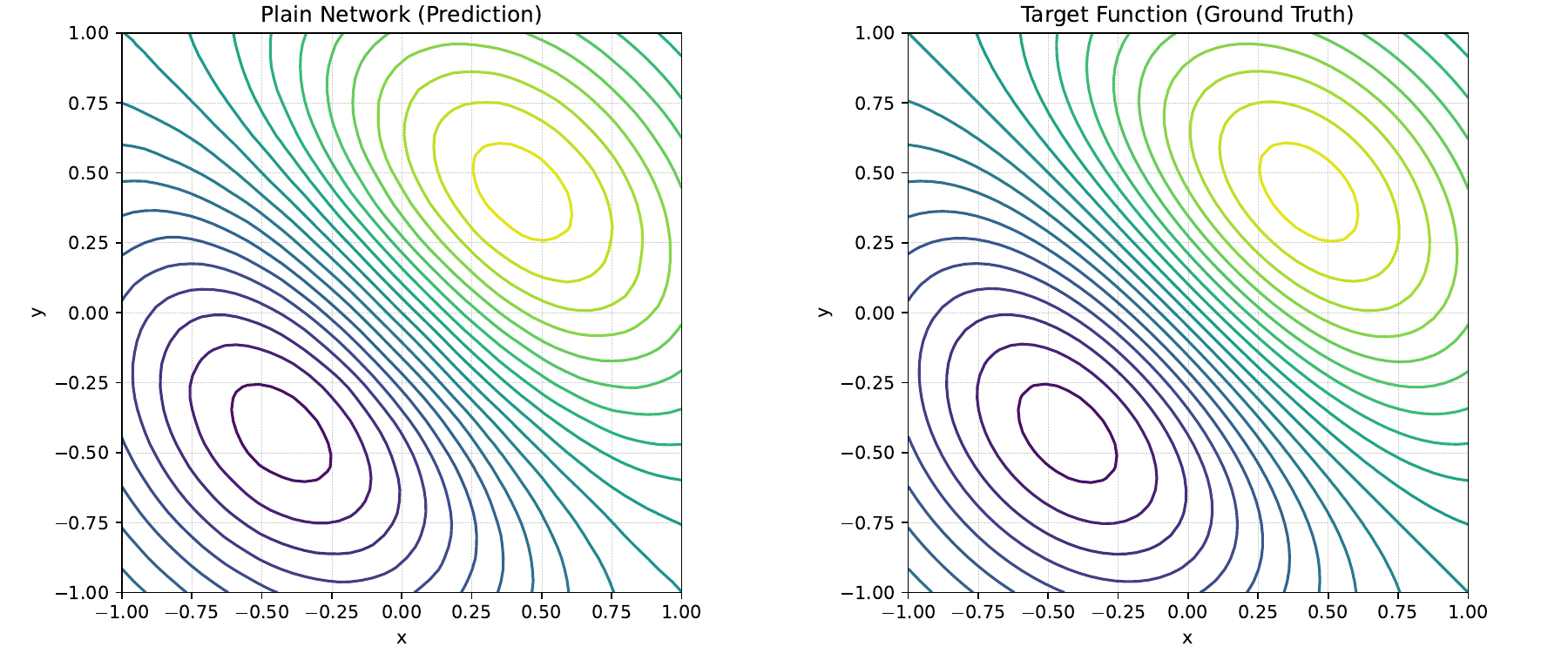}
  \end{subfigure}
  \begin{subfigure}{0.4\textwidth}
      \centering
      \includegraphics[width=\linewidth]{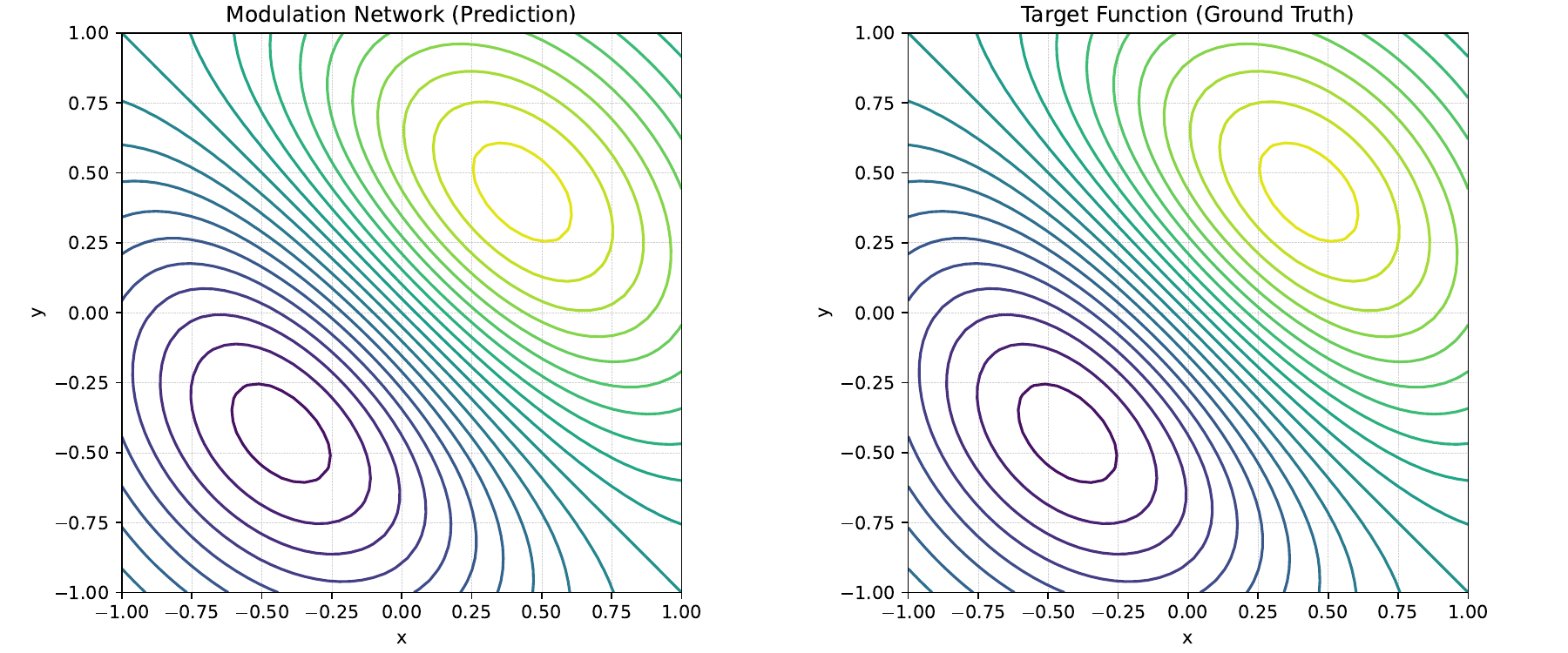}
  \end{subfigure}
    \caption{Comparison of plain and modulation model predictions on unseen two-dimensional data using AdamW optimizer with scheduler, when predicting $F(x, y)$.}
  \label{fig:2d_prediction}
\end{figure}

\begin{figure}[H]
  \centering
    \begin{subfigure}{0.4\textwidth}
      \centering
      \includegraphics[width=\linewidth]{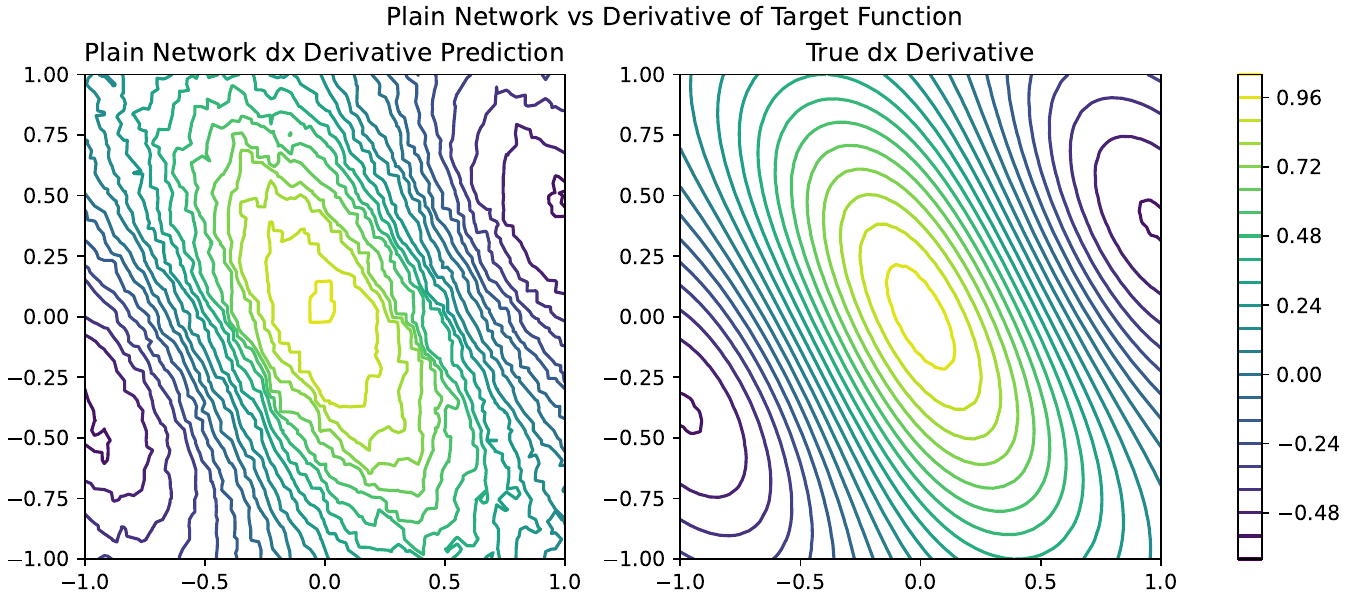}
  \end{subfigure}
    \begin{subfigure}{0.4\textwidth}
      \centering
      \includegraphics[width=\linewidth]{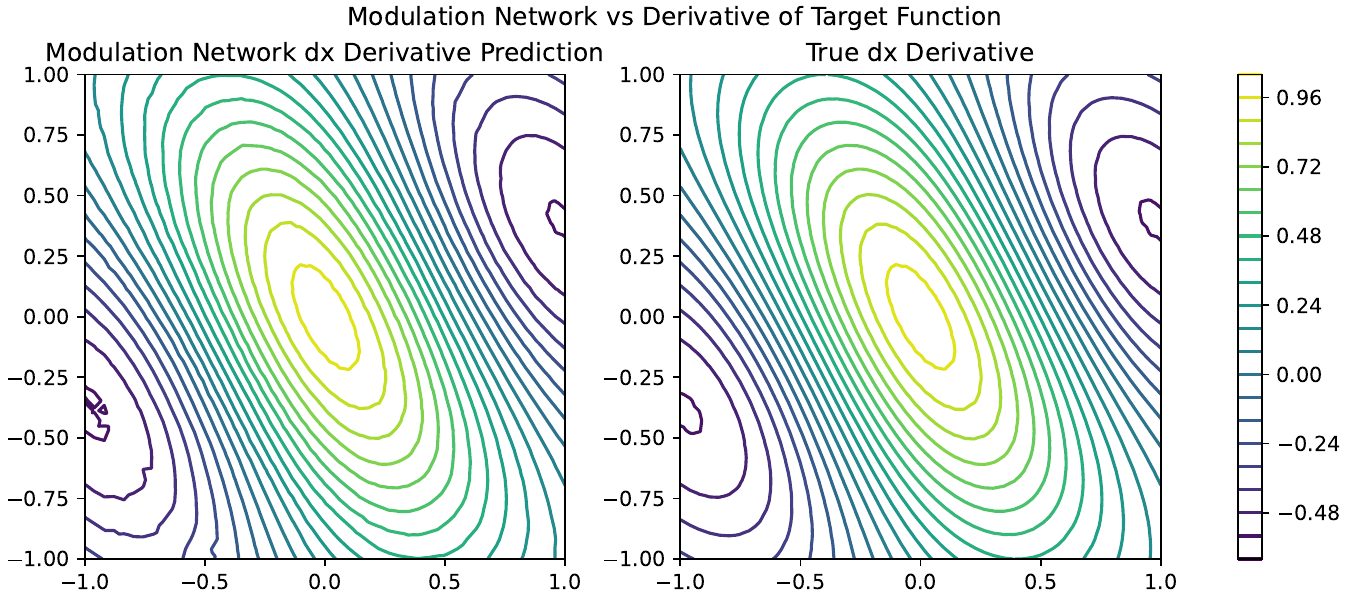}
  \end{subfigure}
    \caption{Comparison of plain and modulation model predictions on unseen two-dimensional data using AdamW optimizer with scheduler when predicting $\partial_xF(x,y)$.}
  \label{fig:2d_prediction_dx}
\end{figure}

  \begin{figure}[H]
  \centering
  \begin{subfigure}{0.4\textwidth}
      \centering
      \includegraphics[width=\linewidth]{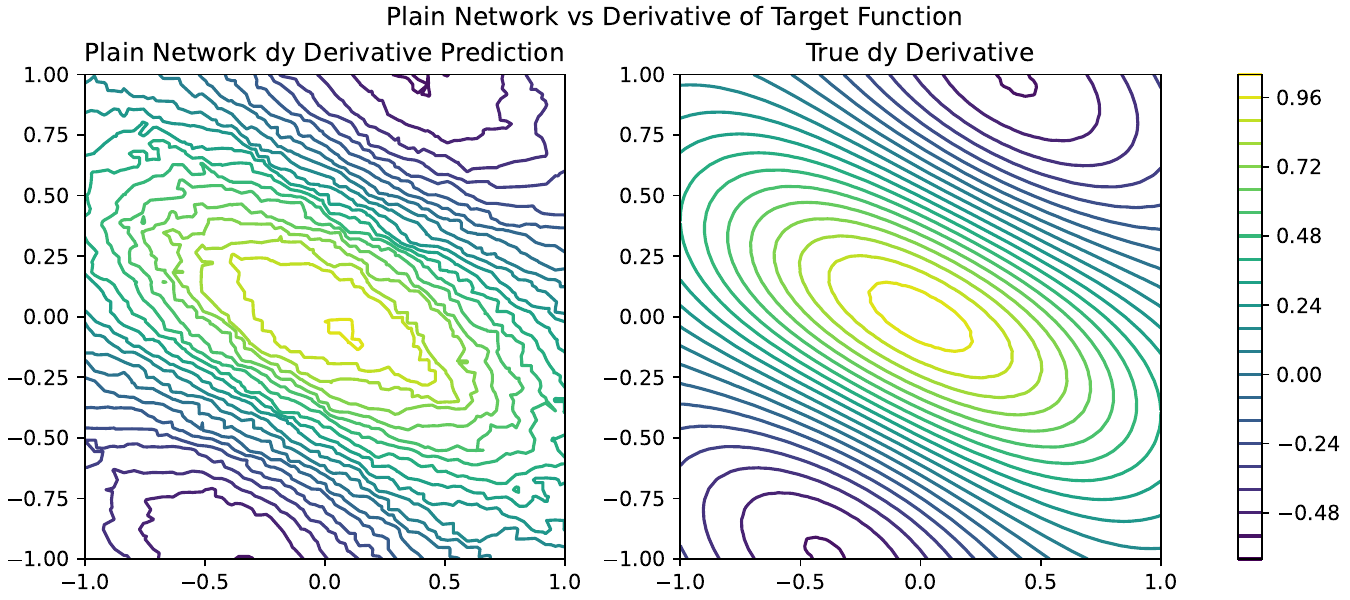}
  \end{subfigure}
  \begin{subfigure}{0.4\textwidth}
      \centering
      \includegraphics[width=\linewidth]{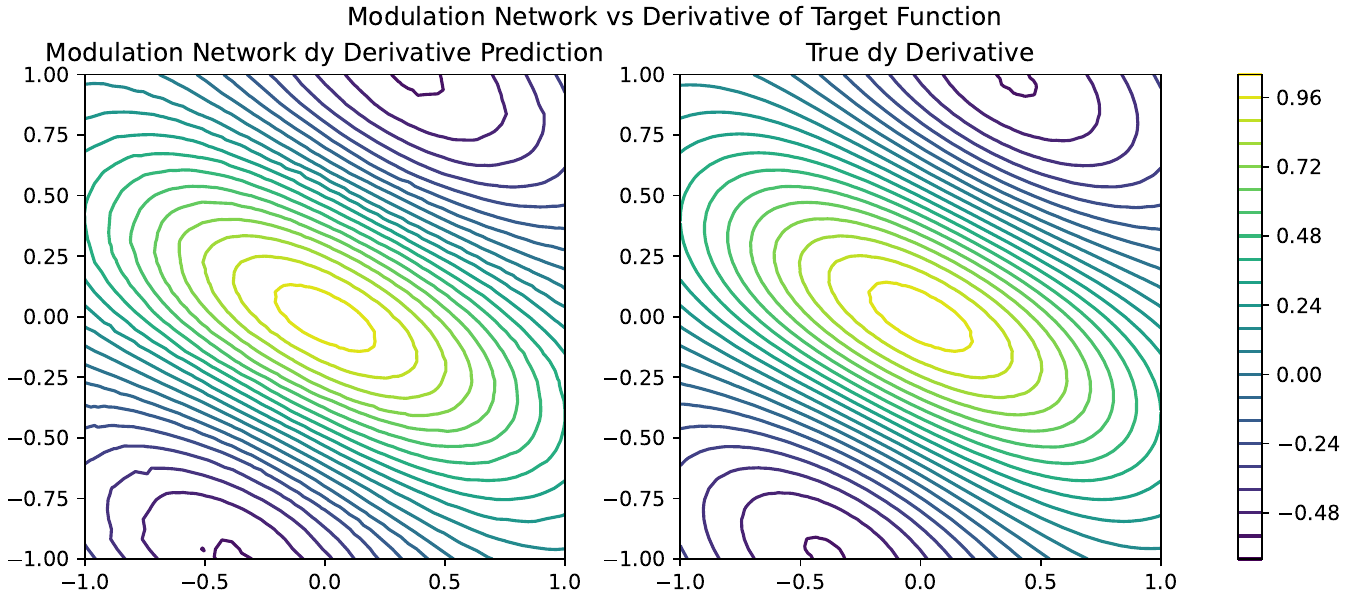}
  \end{subfigure}
  \caption{Comparison of plain and modulation model predictions on unseen one-dimensional data using AdamW optimizer with scheduler when predicting $\partial_yF(x,y)$.}
  \label{fig:2d_prediction_dy}
\end{figure}

\begin{figure}[H]
  \centering
    \includegraphics[width=.7\linewidth]{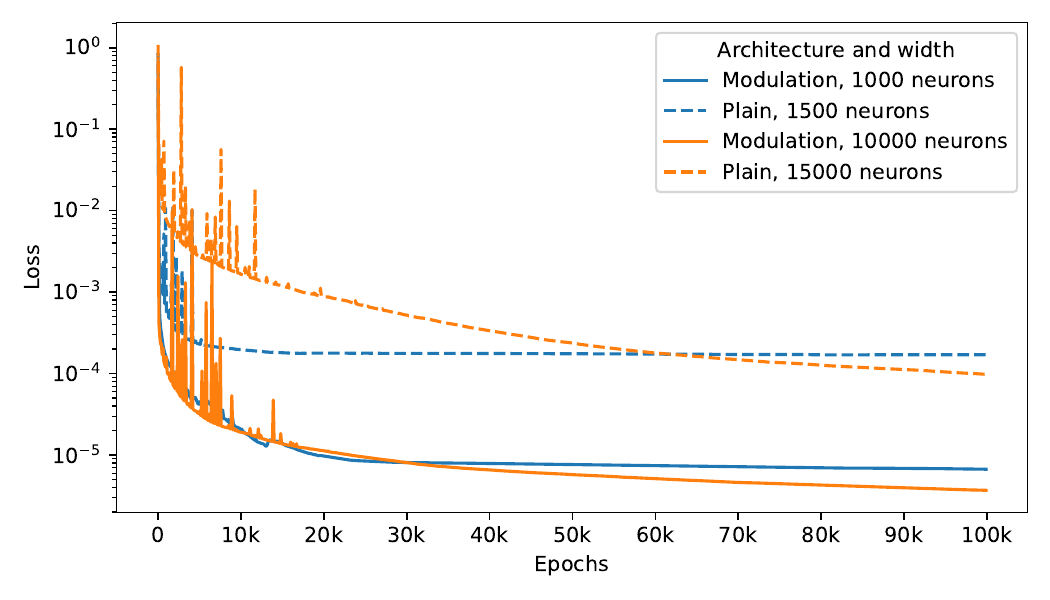}
  \caption{Loss versus epochs for the approximation of $F(x,y)$ using modulation and plain networks with different hidden neurons.}
  \label{fig:2d_prediction_expressivity}
\end{figure}
\cref{fig:2d_prediction_expressivity} shows that despite the $1.5\times$ larger width, the plain architecture consistently converges more slowly and attains a higher loss than the modulation network, demonstrating the superior approximation efficiency of the modulation architecture.

\appendix
\renewcommand{\theequation}{\thesection.\arabic{equation}}
\refstepcounter{section}
\setcounter{equation}{0} 

%%%%%%%%%%%%%%%%%%%%%%%%%%%%%%%%%%%%%%%%%%%%%%%%%%%%%%%%%%%%%%%%%%
\section*{Appendix A. Short-Time Fourier Transform of the ReLU Activation}
%%%%%%%%%%%%%%%%%%%%%%%%%%%%%%%%%%%%%%%%%%%%%%%%%%%%%%%%%%%%%%%%%%

In this appendix we compute explicitly the STFT
of the rectified linear unit
\[
\sigma(t)=t_+ = \max\{0,t\},
\]
with respect to the Gaussian window
\[
\varphi(t)=e^{-\pi t^2}.
\]
Throughout we use the convention
\[
V_\varphi f(x,\omega)
= \int_{\rr{}} f(t)\,\varphi(t-x)\,e^{-2\pi i \omega t}\,dt.
\]

%--------------------------------------------------------------------------
\begin{theorem}[Explicit STFT of the ReLU]\label{thm:relu-stft}
Let $\sigma(t)=t_+=\max\{0,t\}$ and $\varphi(t)=e^{-\pi t^2}$.  
Then for all $(x,\omega)\in\rr{2}$,
\begin{equation}\label{eq:relu-stftE}
V_\varphi \sigma(x,\omega)
= 
\frac12\, e^{-\pi\omega^2}(x - i\omega)\,
e^{-2\pi i\omega x}\,
\operatorname{erfc}\!\bigl(\sqrt{\pi}(-x+i\omega)\bigr)
\;+\;
\frac{1}{2\pi}\, e^{-\pi x^2}.
\end{equation}
Here $\operatorname{erfc}(z)$ denotes the complementary error function
extended to $z\in\cc{}$.
\end{theorem}

%--------------------------------------------------------------------------
\begin{proof}
Since $\sigma(t)=0$ for $t<0$,
\[
V_\varphi \sigma(x,\omega)
= \int_{0}^{\infty} t\,e^{-\pi (t-x)^2} e^{-2\pi i\omega t}\,dt.
\]
Set $s=t-x$, i.e.\ $t=s+x$ and $dt=ds$, so that the lower limit becomes $s=-x$:
\[
V_\varphi \sigma(x,\omega)
= e^{-2\pi i\omega x}
   \int_{-x}^{\infty} (s+x)\,e^{-\pi s^2} e^{-2\pi i\omega s}\,ds.
\]
Define
\begin{equation}
I_0(x,\omega)=\int_{-x}^{\infty} e^{-\pi s^2}e^{-2\pi i\omega s}\,ds,
\qquad
I_1(x,\omega)=\int_{-x}^{\infty} s\,e^{-\pi s^2}e^{-2\pi i\omega s}\,ds.
\end{equation}\label{eq:relu-stft}
Then
\begin{equation}\label{eq:relu-stft-decomp}
V_\varphi \sigma(x,\omega)
= e^{-2\pi i\omega x}\bigl(xI_0(x,\omega) + I_1(x,\omega)\bigr).
\end{equation}

\medskip
\noindent\textit{Step 1: Computation of $I_0$.}
Complete the square:
\[
-\pi s^2 - 2\pi i\omega s
= -\pi(s+i\omega)^2 - \pi\omega^2.
\]
Hence
\[
I_0(x,\omega)
= e^{-\pi\omega^2}
  \int_{-x}^{\infty} e^{-\pi(s+i\omega)^2}\,ds.
\]
Let $u=\sqrt{\pi}(s+i\omega)$; then $ds=du/\sqrt{\pi}$.
Using
\[
\int_{z}^{\infty} e^{-u^2}\,du
= \frac{\sqrt{\pi}}{2}\operatorname{erfc}(z),
\]
we obtain
\begin{equation}\label{eq:I0}
I_0(x,\omega)
= \frac12\,e^{-\pi\omega^2}\,
   \operatorname{erfc}\!\bigl(\sqrt{\pi}(-x+i\omega)\bigr).
\end{equation}
Write $z:=\sqrt{\pi}(-x+i\omega)$ for brevity.

\medskip
\noindent\textit{Step 2: Computation of $I_1$.}
Differentiate the integrand:
\[
\frac{\partial}{\partial\omega}\!
\left(e^{-\pi s^2} e^{-2\pi i\omega s}\right)
= -2\pi i s\,e^{-\pi s^2}e^{-2\pi i\omega s}.
\]
Thus
\[
\frac{\partial I_0}{\partial\omega}(x,\omega)
= -2\pi i\,I_1(x,\omega),
\qquad
I_1(x,\omega)
= -\frac{1}{2\pi i}\,\frac{\partial I_0}{\partial\omega}.
\]
Differentiating \eqref{eq:I0} and using 
$\operatorname{erfc}'(z)=-\frac{2}{\sqrt{\pi}}e^{-z^2}$ and 
$z'(\omega)=\sqrt{\pi}i$, we find
\begin{equation}\label{eq:dI0}
\frac{\partial I_0}{\partial\omega}
= e^{-\pi\omega^2}\bigl[-\pi\omega\,\operatorname{erfc}(z)
                         - i\,e^{-z^2}\bigr].
\end{equation}
Substituting \eqref{eq:dI0} into the relation for $I_1$ yields
\[
I_1(x,\omega)
= -\frac{i\omega}{2}e^{-\pi\omega^2}\operatorname{erfc}(z)
  + \frac{1}{2\pi}e^{-\pi x^2}e^{2\pi i x\omega}.
\]

\medskip
\noindent\textit{Step 3: Reconstruction of the STFT.}
From \eqref{eq:I0} and the expression for $I_1$,
\[
xI_0+I_1
= \frac12 e^{-\pi\omega^2}(x-i\omega)\operatorname{erfc}(z)
  + \frac{1}{2\pi}e^{-\pi x^2}e^{2\pi i x\omega}.
\]
Multiplying by $e^{-2\pi i\omega x}$ as in 
\eqref{eq:relu-stft-decomp} proves formula \eqref{eq:relu-stftE}.
\end{proof}
\begin{lemma}[Value at the Origin]\label{lem:31}
For $\sigma(t)=t_+$ and $\varphi(t)=e^{-\pi t^2}$,
\[
V_\varphi \sigma(0,0)
= \int_{0}^{\infty} t\,e^{-\pi t^2}\,dt
= \frac{1}{2\pi}.
\]
\end{lemma}

\begin{proof}
Since $\sigma(t)=0$ for $t<0$,
\[
V_\varphi \sigma(0,0)
= \int_{0}^{\infty} t\,e^{-\pi t^2}\,dt.
\]
With $u=\pi t^2$ (so $t\,dt = du/(2\pi)$) we obtain
\[
\int_0^\infty t\,e^{-\pi t^2}\,dt
= \frac{1}{2\pi}\int_0^\infty e^{-u}\,du
= \frac{1}{2\pi}.
\]
\end{proof}
%------------------------------------------------------------

%--------------------------------------------------------------------------
\begin{corollary}[Decay Estimates]\label{cor:relu-decay}
Let $\sigma(t)=t_+$ and $\varphi(t)=e^{-\pi t^2}$.  
Then for all $(x,\omega)\in\rr{2}$,
\begin{equation*}%\label{eq:relu-decay}
|V_\varphi \sigma(x,\omega)|
\,\le\,
C\,(1+|x|+|\omega|)\,e^{-\pi(x^2+\omega^2)}
\;+\;
\frac{1}{2\pi} e^{-\pi x^2},
\end{equation*}
for some constant $C>0$.
\end{corollary}

\begin{proof}
We call the  first term in \eqref{eq:relu-stftE} 
\[
T_1(x,\omega)
:=
\frac12 e^{-\pi\omega^2}(x-i\omega)
e^{-2\pi i\omega x}
\operatorname{erfc}(z),
\qquad z=\sqrt{\pi}(-x+i\omega).
\]
Using standard bounds on \( |\operatorname{erfc}(z)| \) in the complex plane together with the Gaussian decay, it follows that
\[
\bigl|T_1(x,\omega)\bigr|
\le
C\,(1+|x|+|\omega|)\,e^{-\pi(x^2+\omega^2)}.
\]
The second term in \eqref{eq:relu-stftE} is
$\frac{1}{2\pi}e^{-\pi x^2}$, completing the proof.
\end{proof}

\begin{remark}\label{rem:33}
The estimates above imply that the STFT is in $L^{1,\infty}_{v_s\otimes 1}(\rr{2})$ which means $\sigma\in M^{1,\infty}_{v_s\otimes 1}(\rr{})$, for every $s\in\rr{}$.
These properties justify the use of $\sigma$ 
within the analytic framework of
time-frequency localization.
\end{remark}

%%%%%%%%%%%%%%%%%%%%%%%%%%%%%%%%%%%%%%%%%%%%%%%%%%%%%%%%%%%%%%%%%%
\section*{Acknowledgments}
%%%%%%%%%%%%%%%%%%%%%%%%%%%%%%%%%%%%%%%%%%%%%%%%%%%%%%%%%%%%%%%%%%

The authors thank the \textit{Erwin Schr\"odinger International Institute for Mathematics and Physics (ESI)},  University of Vienna. This work began during their ESI stay from May 5 to 9, 2025.\par
\noindent The authors would like to thank Dr. Thomas Dittrich for helpful discussions regarding the numerical implementation.
E.~Cordero has been partially supported by the Italian Ministry of the University and Research - MUR, 
within the framework of the Call relating to the scrolling of the final rankings of the PRIN 2022 - Project Code
2022HCLAZ8, CUP D53C24003370006 (PI A. Palmieri, Local unit Sc. Resp. S. Coriasco).
A.~Abdeljawad acknowledges the support of the Austrian Science Fund (FWF) through project PAT4788625 (Grant-DOI: 10.55776/PAT4788625).

\printbibliography

@article{Guo18SharpWeightedConvolution,
  title = {Sharp Weighted Convolution Inequalities and Some Applications},
  author = {Guo, Weichao and Fan, Dashan and Wu, Huoxiong and Zhao, Guoping},
  date = {2018},
  journaltitle = {Studia Mathematica},
  shortjournal = {Studia Math.},
  volume = {241},
  number = {3},
  pages = {201--239},
  publisher = {Institute of Mathematics, Polish Academy of Sciences},
  issn = {0039-3223, 1730-6337},
  doi = {10.4064/sm8583-5-2017},
  url = {http://www.impan.pl/get/doi/10.4064/sm8583-5-2017},
  urldate = {2025-07-23},
  langid = {english}
}

@article{Abdeljawad22ApproximationsDeepNeural,
	title = {Approximations with deep neural networks in {Sobolev} time-space},
	volume = {20},
	issn = {0219-5305, 1793-6861},
	url = {https://www.worldscientific.com/doi/10.1142/S0219530522500014},
	doi = {10.1142/S0219530522500014},
	language = {en},
	number = {03},
	urldate = {2024-12-26},
	journal = {Analysis and Applications},
	author = {Abdeljawad, Ahmed and Grohs, Philipp},
	month = may,
	year = {2022},
	pages = {499--541},
}

@book{Benyi20ModulationSpacesApplications,
	address = {New York, NY},
	series = {Applied and {Numerical} {Harmonic} {Analysis}},
	title = {Modulation {Spaces}: {With} {Applications} to {Pseudodifferential} {Operators} and {Nonlinear} {Schrödinger} {Equations}},
	copyright = {http://www.springer.com/tdm},
	isbn = {978-1-07-160330-7},
	shorttitle = {Modulation {Spaces}},
	url = {http://link.springer.com/10.1007/978-1-0716-0332-1},
	language = {en},
	urldate = {2025-06-18},
	publisher = {Springer New York},
	author = {Bényi, Árpád and Okoudjou, Kasso A.},
	year = {2020},
	doi = {10.1007/978-1-0716-0332-1},
}

@article{Sjostrand94AlgebraPseudodifferentialOperators,
	title = {An algebra of pseudodifferential operators},
	volume = {1},
	issn = {1073-2780, 1945-001X},
	url = {https://link.intlpress.com/JDetail/1806607787310346241},
	doi = {10.4310/mrl.1994.v1.n2.a6},
	language = {en},
	number = {2},
	urldate = {2025-07-08},
	journal = {Mathematical Research Letters},
	author = {Sjöstrand, Johannes},
	year = {1994},
	note = {Publisher: International Press of Boston},
	pages = {185--192},
}

@article{Galperin04TimefrequencyAnalysisModulation,
	title = {Time-frequency analysis on modulation spaces  $M_m^{p,q}$,  \$0{\textless}p,q{\textbackslash}leq {\textbackslash}infty\$},
	volume = {16},
	copyright = {https://www.elsevier.com/tdm/userlicense/1.0/},
	issn = {1063-5203},
	url = {https://linkinghub.elsevier.com/retrieve/pii/S1063520303000770},
	doi = {10.1016/j.acha.2003.09.001},
	language = {en},
	number = {1},
	urldate = {2025-07-08},
	journal = {Applied and Computational Harmonic Analysis},
	author = {Galperin, Yevgeniy V. and Samarah, Salti},
	month = jan,
	year = {2004},
	note = {Publisher: Elsevier BV},
	pages = {1--18},
}

@incollection{Feichtinger03ModulationSpacesLocally,
	address = {New Delhi},
	title = {Modulation spaces on locally compact abelian groups},
	booktitle = {Technical report, {University} of {Vienna}, 1983; also in {Proceedings} of the international conference on wavelets and applications},
	publisher = {Allied Publishers},
	author = {Feichtinger, Hans G.},
	editor = {Radha, R. and Krishna, M. and Thangavelu, S.},
	year = {2003},
	pages = {1--56},
}

@article{Teofanov06ModulationSpacesGelfandShilov,
	title = {Modulation {Spaces}, {Gelfand}-{Shilov} {Spaces} and {Pseudodifferential} {Operators}},
	volume = {5},
	issn = {1530-6429},
	url = {https://link.springer.com/10.1007/BF03549452},
	doi = {10.1007/BF03549452},
	language = {en},
	number = {2},
	urldate = {2025-07-02},
	journal = {Sampling Theory in Signal and Image Processing},
	author = {Teofanov, Nenad},
	month = may,
	year = {2006},
	pages = {225--242},
}

@article{Barron93UniversalApproximationBounds,
	title = {Universal approximation bounds for superpositions of a sigmoidal function},
	volume = {39},
	copyright = {https://ieeexplore.ieee.org/Xplorehelp/downloads/license-information/IEEE.html},
	issn = {0018-9448, 1557-9654},
	url = {https://ieeexplore.ieee.org/document/256500/},
	doi = {10.1109/18.256500},
	number = {3},
	urldate = {2025-06-25},
	journal = {IEEE Transactions on Information Theory},
	author = {Barron, A.R.},
	month = may,
	year = {1993},
	pages = {930--945},
}

@article{Cohen22OptimalStableNonlinear,
	title = {Optimal {Stable} {Nonlinear} {Approximation}},
	volume = {22},
	issn = {1615-3375, 1615-3383},
	url = {https://link.springer.com/10.1007/s10208-021-09494-z},
	doi = {10.1007/s10208-021-09494-z},
	language = {en},
	number = {3},
	urldate = {2025-06-25},
	journal = {Foundations of Computational Mathematics},
	author = {Cohen, Albert and DeVore, Ronald and Petrova, Guergana and Wojtaszczyk, Przemyslaw},
	month = jun,
	year = {2022},
	pages = {607--648},
}

@article{Feichtinger92WilsonBasesModulation,
	title = {Wilson {Bases} and {Modulation} {Spaces}},
	volume = {155},
	copyright = {http://onlinelibrary.wiley.com/termsAndConditions\#vor},
	issn = {0025-584X, 1522-2616},
	url = {https://onlinelibrary.wiley.com/doi/10.1002/mana.19921550102},
	doi = {10.1002/mana.19921550102},
	language = {en},
	number = {1},
	urldate = {2025-06-25},
	journal = {Mathematische Nachrichten},
	author = {Feichtinger, H. G. and Gröchenig, K. and Walnut, D.},
	month = jan,
	year = {1992},
	pages = {7--17},
}

@article{Brzezniak95StochasticPartialDifferential,
	title = {Stochastic partial differential equations in {M}-type 2 {Banach} spaces},
	volume = {4},
	copyright = {http://www.springer.com/tdm},
	issn = {0926-2601, 1572-929X},
	url = {http://link.springer.com/10.1007/BF01048965},
	doi = {10.1007/BF01048965},
	language = {en},
	number = {1},
	urldate = {2025-06-23},
	journal = {Potential Analysis},
	author = {Brzeźniak, Zdzisław},
	month = feb,
	year = {1995},
	pages = {1--45},
}

@inproceedings{Parhi2023ModulationSpacesCurse,
  title = {Modulation {{Spaces}} and the {{Curse}} of {{Dimensionality}}},
  booktitle = {2023 {{International Conference}} on {{Sampling Theory}} and {{Applications}} ({{SampTA}})},
  author = {Parhi, Rahul and Unser, Michael},
  date = {2023-07-10},
  pages = {1--5},
  publisher = {IEEE},
  location = {New Haven, CT, USA},
  doi = {10.1109/SampTA59647.2023.10301395},
  url = {https://ieeexplore.ieee.org/document/10301395/},
  urldate = {2025-12-12},
  eventtitle = {2023 {{International Conference}} on {{Sampling Theory}} and {{Applications}} ({{SampTA}})},
  isbn = {979-8-3503-2885-1}
}

@article{Abdeljawad2020LiftingsUltramodulationSpaces,
  title = {Liftings for Ultra-Modulation Spaces, and One-Parameter Groups of {{Gevrey-type}} Pseudo-Differential Operators},
  author = {Abdeljawad, Ahmed and Coriasco, Sandro and Toft, Joachim},
  date = {2020-07},
  journaltitle = {Analysis and Applications},
  shortjournal = {Anal. Appl.},
  volume = {18},
  number = {04},
  pages = {523--583},
  issn = {0219-5305, 1793-6861},
  doi = {10.1142/S0219530519500143},
  url = {https://www.worldscientific.com/doi/abs/10.1142/S0219530519500143},
  urldate = {2025-12-12},
  langid = {english}
}

@article{Guo17InclusionRelationsModulation,
	title = {Inclusion relations between modulation and {Triebel}-{Lizorkin} spaces},
	volume = {145},
	copyright = {https://www.ams.org/publications/copyright-and-permissions},
	issn = {0002-9939, 1088-6826},
	url = {https://www.ams.org/proc/2017-145-11/S0002-9939-2017-13614-4/},
	doi = {10.1090/proc/13614},
	language = {en},
	number = {11},
	urldate = {2025-06-18},
	journal = {Proceedings of the American Mathematical Society},
	author = {Guo, Weichao and Wu, Huoxiong and Zhao, Guoping},
	month = may,
	year = {2017},
	pages = {4807--4820},
}

@article{Grochenig00NonlinearApproximationLocal,
	title = {Nonlinear {Approximation} with {Local} {Fourier} {Bases}},
	volume = {16},
	copyright = {http://www.springer.com/tdm},
	issn = {0176-4276, 1432-0940},
	url = {http://link.springer.com/10.1007/s003659910014},
	doi = {10.1007/s003659910014},
	language = {en},
	number = {3},
	urldate = {2025-06-18},
	journal = {Constructive Approximation},
	author = {Gröchenig, K. and Samarah, S.},
	month = jul,
	year = {2000},
	pages = {317--331},
}

@article{Kobayashi11InclusionRelationSobolev,
	title = {The inclusion relation between {Sobolev} and modulation spaces},
	volume = {260},
	copyright = {https://www.elsevier.com/tdm/userlicense/1.0/},
	issn = {00221236},
	url = {https://linkinghub.elsevier.com/retrieve/pii/S0022123611000723},
	doi = {10.1016/j.jfa.2011.02.015},
	language = {en},
	number = {11},
	urldate = {2025-06-18},
	journal = {Journal of Functional Analysis},
	author = {Kobayashi, Masaharu and Sugimoto, Mitsuru},
	month = jun,
	year = {2011},
	pages = {3189--3208},
}

@book{Cordero20TimeFrequencyAnalysisOperators,
	title = {Time-{Frequency} {Analysis} of {Operators}},
	isbn = {978-3-11-053245-6},
	url = {https://www.degruyterbrill.com/document/doi/10.1515/9783110532456/html},
	urldate = {2025-06-18},
	publisher = {De Gruyter},
	author = {Cordero, Elena and Rodino, Luigi},
	month = sep,
	year = {2020},
	doi = {10.1515/9783110532456},
}

@article{Siegel22SharpBoundsApproximation,
	title = {Sharp {Bounds} on the {Approximation} {Rates}, {Metric} {Entropy}, and n-{Widths} of {Shallow} {Neural} {Networks}},
	issn = {1615-3375, 1615-3383},
	shorttitle = {Sharp {Bounds}},
	url = {https://link.springer.com/10.1007/s10208-022-09595-3},
	doi = {10.1007/s10208-022-09595-3},
	language = {en},
	urldate = {2023-06-14},
	journal = {Foundations of Computational Mathematics},
	author = {Siegel, Jonathan W. and Xu, Jinchao},
	month = nov,
	year = {2022},
}

@article{DeVore98NonlinearApproximation,
	title = {Nonlinear {Approximation}},
	volume = {7},
	doi = {10.1017/S0962492900002816},
	language = {en},
	journal = {Acta Numerica},
	author = {DeVore, Ronald Alvin},
	month = jan,
	year = {1998},
	pages = {51--150},
}

@article{Siegel23CharacterizationVariationSpaces,
	title = {Characterization of the {Variation} {Spaces} {Corresponding} to {Shallow} {Neural} {Networks}},
	issn = {0176-4276, 1432-0940},
	shorttitle = {Variation {Spaces}},
	url = {https://link.springer.com/10.1007/s00365-023-09626-4},
	doi = {10.1007/s00365-023-09626-4},
	language = {en},
	urldate = {2023-06-07},
	journal = {Constructive Approximation},
	author = {Siegel, Jonathan W. and Xu, Jinchao},
	month = feb,
	year = {2023},
}

@book{Grochenig01FoundationsTimeFrequencyAnalysis,
	address = {Boston, MA},
	series = {Applied and {Numerical} {Harmonic} {Analysis}},
	title = {Foundations of {Time}-{Frequency} {Analysis}},
	copyright = {http://www.springer.com/tdm},
	isbn = {978-1-4612-6568-9},
	url = {http://link.springer.com/10.1007/978-1-4612-0003-1},
	urldate = {2025-06-11},
	publisher = {Birkhäuser Boston},
	author = {Gröchenig, Karlheinz},
	editor = {Benedetto, John J.},
	year = {2001},
	doi = {10.1007/978-1-4612-0003-1},
}

@article{Siegel20ApproximationRatesNeural,
	title = {Approximation {Rates} for {Neural} {Networks} {With} {General} {Activation} {Functions}},
	volume = {128},
	issn = {08936080},
	url = {https://linkinghub.elsevier.com/retrieve/pii/S0893608020301891},
	doi = {10.1016/j.neunet.2020.05.019},
	language = {en},
	urldate = {2023-08-31},
	journal = {Neural Networks},
	author = {Siegel, Jonathan W. and Xu, Jinchao},
	month = aug,
	year = {2020},
	pages = {313--321},
}

@article{Siegel24SharpBoundsApproximation,
	title = {Sharp {Bounds} on the {Approximation} {Rates}, {Metric} {Entropy}, and n-{Widths} of {Shallow} {Neural} {Networks}},
	volume = {24},
	issn = {1615-3375, 1615-3383},
	url = {https://link.springer.com/10.1007/s10208-022-09595-3},
	doi = {10.1007/s10208-022-09595-3},
	language = {en},
	number = {2},
	urldate = {2025-06-12},
	journal = {Foundations of Computational Mathematics},
	author = {Siegel, Jonathan W. and Xu, Jinchao},
	month = apr,
	year = {2024},
	pages = {481--537},
}

@inproceedings{Marwah23Neuralnetworkapproximations,
	series = {Proceedings of machine learning research},
	title = {Neural network approximations of {PDEs} beyond linearity: a representational perspective},
	volume = {202},
	url = {https://proceedings.mlr.press/v202/marwah23a.html},
	booktitle = {Proceedings of the 40th international conference on machine learning},
	publisher = {PMLR},
	author = {Marwah, Tanya and Lipton, Zachary Chase and Lu, Jianfeng and Risteski, Andrej},
	editor = {Krause, Andreas and Brunskill, Emma and Cho, Kyunghyun and Engelhardt, Barbara and Sabato, Sivan and Scarlett, Jonathan},
	month = jul,
	year = {2023},
	pages = {24139--24172},
}

@article{Chen23RegularityTheoryStatic,
	title = {A {Regularity} {Theory} for {Static} {Schrödinger} {Equations} on {\textbackslash}({\textbackslash}boldsymbol\{{\textbackslash}mathbb\{{R}\}\}{\textbackslash})$^{\textrm{ \textit{d} }}$ in {Spectral} {Barron} {Spaces}},
	volume = {55},
	issn = {0036-1410, 1095-7154},
	url = {https://epubs.siam.org/doi/10.1137/22M1478719},
	doi = {10.1137/22M1478719},
	language = {en},
	number = {1},
	urldate = {2025-06-11},
	journal = {SIAM Journal on Mathematical Analysis},
	author = {Chen, Ziang and Lu, Jianfeng and Lu, Yulong and Zhou, Shengxuan},
	month = feb,
	year = {2023},
	pages = {557--570},
}

@article{Cordero03TimeFrequencyanalysis,
	title = {Time–{Frequency} analysis of localization operators},
	volume = {205},
	copyright = {https://www.elsevier.com/tdm/userlicense/1.0/},
	issn = {00221236},
	url = {https://linkinghub.elsevier.com/retrieve/pii/S0022123603001666},
	doi = {10.1016/S0022-1236(03)00166-6},
	language = {en},
	number = {1},
	urldate = {2025-06-11},
	journal = {Journal of Functional Analysis},
	author = {Cordero, Elena and Gröchenig, Karlheinz},
	month = dec,
	year = {2003},
	pages = {107--131},
}

@misc{Abdeljawad24WeightedApproximationBarron,
	title = {Weighted {Sobolev} {Approximation} {Rates} for {Neural} {Networks} on {Unbounded} {Domains}},
	copyright = {arXiv.org perpetual, non-exclusive license},
	url = {https://arxiv.org/abs/2411.04108},
	doi = {10.48550/ARXIV.2411.04108},
	urldate = {2025-01-14},
	publisher = {arXiv},
	author = {Abdeljawad, Ahmed and Dittrich, Thomas},
	year = {2024},
	note = {Version Number: 1},
}

@article{Toft04ContinuityPropertiesModulation,
	title = {Continuity {Properties} for {Modulation} {Spaces}, with {Applications} to {Pseudo}-{Differential} {Calculus}, {II}},
	volume = {26},
	copyright = {https://www.springernature.com/gp/researchers/text-and-data-mining},
	issn = {0232-704X, 1572-9060},
	url = {https://link.springer.com/10.1023/B:AGAG.0000023261.94488.f4},
	doi = {10.1023/B:AGAG.0000023261.94488.f4},
	language = {en},
	number = {1},
	urldate = {2025-06-09},
	journal = {Annals of Global Analysis and Geometry},
	author = {Toft, Joachim},
	month = aug,
	year = {2004},
	pages = {73--106},
}

@book{Tartar07IntroductionSobolevSpaces,
	address = {Berlin, Heidelberg},
	series = {Lecture {Notes} of the {Unione} {Matematica} {Italiana}},
	title = {An {Introduction} to {Sobolev} {Spaces} and {Interpolation} {Spaces}},
	volume = {3},
	isbn = {978-3-540-71482-8},
	shorttitle = {Sobolev {Space}},
	url = {http://link.springer.com/10.1007/978-3-540-71483-5},
	language = {en},
	urldate = {2023-06-05},
	publisher = {Springer Berlin Heidelberg},
	author = {Tartar, Luc},
	year = {2007},
	doi = {10.1007/978-3-540-71483-5},
}

@misc{Abdeljawad23SpaceTimeApproximationShallow,
	title = {Space-{Time} {Approximation} with {Shallow} {Neural} {Networks} in {Fourier} {Lebesgue} {Spaces}},
	url = {http://arxiv.org/abs/2312.08461},
	urldate = {2024-11-06},
	publisher = {arXiv},
	author = {Abdeljawad, Ahmed and Dittrich, Thomas},
	month = dec,
	year = {2023},
	note = {arXiv:2312.08461 [cs]},
}

@book{Shubin01PseudodifferentialOperatorsSpectral,
    address = {Berlin, Heidelberg},
    title = {Pseudodifferential {Operators} and {Spectral} {Theory}},
    copyright = {http://www.springer.com/tdm},
    isbn = {978-3-540-41195-6},
    url = {http://link.springer.com/10.1007/978-3-642-56579-3},
    language = {en},
    urldate = {2025-07-25},
    publisher = {Springer Berlin Heidelberg},
    author = {Shubin, Mikhail A.},
    year = {2001},
    doi = {10.1007/978-3-642-56579-3},
}

@article{Boggiatto04GeneralizedAntiWickOperators,
    title = {Generalized {Anti}-{Wick} {Operators} with {Symbols} in {Distributional} {Sobolev} spaces},
    volume = {48},
    copyright = {http://www.springer.com/tdm},
    issn = {0378-620X},
    url = {http://link.springer.com/10.1007/s00020-003-1244-x},
    doi = {10.1007/s00020-003-1244-x},
    number = {4},
    urldate = {2025-07-25},
    journal = {Integral Equations and Operator Theory},
    author = {Boggiatto, Paolo and Cordero, Elena and Gr\"ochenig, Karlheinz},
    month = apr,
    year = {2004},
    note = {Publisher: Springer Science and Business Media LLC},
    pages = {427--442},
}

@article{Bastianoni21SubexponentialDecayRegularity,
    title = {Subexponential decay and regularity estimates for eigenfunctions of localization operators},
    volume = {12},
    issn = {1662-9981, 1662-999X},
    url = {http://link.springer.com/10.1007/s11868-021-00383-1},
    doi = {10.1007/s11868-021-00383-1},
    language = {en},
    number = {1},
    urldate = {2025-09-04},
    journal = {Journal of Pseudo-Differential Operators and Applications},
    author = {Bastianoni, Federico and Teofanov, Nenad},
    month = mar,
    year = {2021},
    pages = {19},
}

@article{Pilipovic10MicrolocalanalysisFourier,
    title = {Micro-{Local} {Analysis} in {Fourier} {Lebesgue} and {Modulation} {Spaces}: {Part} {II}},
    volume = {1},
    issn = {1662-9981, 1662-999X},
    url = {http://link.springer.com/10.1007/s11868-010-0013-2},
    doi = {10.1007/s11868-010-0013-2},
    language = {english},
    number = {3},
    urldate = {2023-06-06},
    journal = {Journal of Pseudo-Differential Operators and Applications},
    author = {Pilipović, Stevan and Teofanov, Nenad and Toft, Joachim},
    month = sep,
    year = {2010},
    pages = {341--376},
}

@book{Katznelson2004IntroductionHarmonicAnalysis,
    edition = {3},
    title = {An {Introduction} to {Harmonic} {Analysis}},
    copyright = {https://www.cambridge.org/core/terms},
    isbn = {978-0-521-83829-0},
    url = {https://www.cambridge.org/core/product/identifier/9781139165372/type/book},
    urldate = {2025-11-19},
    publisher = {Cambridge University Press},
    author = {Katznelson, Yitzhak},
    month = jan,
    year = {2004},
    doi = {10.1017/CBO9781139165372},
}

@article{E2022BarronSpaceFlowInduced,
    title = {The {Barron} {Space} and the {Flow}-{Induced} {Function} {Spaces} for {Neural} {Network} {Models}},
    volume = {55},
    issn = {0176-4276, 1432-0940},
    url = {https://link.springer.com/10.1007/s00365-021-09549-y},
    doi = {10.1007/s00365-021-09549-y},
    language = {en},
    number = {1},
    urldate = {2025-11-19},
    journal = {Constructive Approximation},
    author = {E, Weinan and Ma, Chao and Wu, Lei},
    month = feb,
    year = {2022},
    pages = {369--406},
}

@article{Voigtlaender22SamplingNumbersFourierAnalytic,
    title = {\${L}{\textasciicircum}p\$ {Sampling} {Numbers} for the {Fourier}-{Analytic} {Barron} {Space}},
    shorttitle = {Barron {Space} {Sampling}},
    url = {https://arxiv.org/abs/2208.07605},
    urldate = {2023-06-13},
    journal = {arXiv preprint arXiv:2208.07605},
    author = {Voigtlaender, Felix},
    year = {2022},
    note = {Publisher: arXiv
tex.version: 1},
}

@article{Cybenko89ApproximationSuperpositionsSigmoidal,
    title = {Approximation by {Superpositions} of a {Sigmoidal} {Function}},
    volume = {2},
    issn = {0932-4194, 1435-568X},
    url = {http://link.springer.com/10.1007/BF02551274},
    doi = {10.1007/BF02551274},
    language = {english},
    number = {4},
    urldate = {2023-06-15},
    journal = {Mathematics of Control, Signals, and Systems},
    author = {Cybenko, G.},
    month = dec,
    year = {1989},
    pages = {303--314},
}

@inproceedings{Chen2018NeuralOrdinaryDifferential,
  title = {Neural Ordinary Differential Equations},
  booktitle = {Advances in Neural Information Processing Systems},
  author = {Chen, Ricky T. Q. and Rubanova, Yulia and Bettencourt, Jesse and Duvenaud, David K},
  editor = {Bengio, S. and Wallach, H. and Larochelle, H. and Grauman, K. and Cesa-Bianchi, N. and Garnett, R.},
  date = {2018},
  volume = {31},
  publisher = {Curran Associates, Inc.},
  url = {https://proceedings.neurips.cc/paper_files/paper/2018/file/69386f6bb1dfed68692a24c8686939b9-Paper.pdf}
}

@inproceedings{Suzuki2019AdaptivityDeepReLU,
  title = {Adaptivity of Deep {{ReLU}} Network for Learning in {{Besov}} and Mixed Smooth {{Besov}} Spaces: Optimal Rate and Curse of Dimensionality},
  booktitle = {International Conference on Learning Representations},
  author = {Suzuki, Taiji},
  date = {2019},
  url = {https://openreview.net/forum?id=H1ebTsActm}
}

@article{Elbrachter2021DeepNeuralNetwork,
  title = {Deep {{Neural Network Approximation Theory}}},
  author = {Elbrachter, Dennis and Perekrestenko, Dmytro and Grohs, Philipp and Bolcskei, Helmut},
  date = {2021-05},
  journaltitle = {IEEE Transactions on Information Theory},
  shortjournal = {IEEE Trans. Inform. Theory},
  volume = {67},
  number = {5},
  pages = {2581--2623},
  issn = {0018-9448, 1557-9654},
  doi = {10.1109/TIT.2021.3062161},
  url = {https://ieeexplore.ieee.org/document/9363169/},
  urldate = {2025-11-25}
}

@article{Petersen2018OptimalApproximationPiecewise,
  title = {Optimal Approximation of Piecewise Smooth Functions Using Deep {{ReLU}} Neural Networks},
  author = {Petersen, Philipp and Voigtlaender, Felix},
  date = {2018},
  journaltitle = {Neural Networks},
  volume = {108},
  pages = {296--330},
  issn = {08936080},
  doi = {10.1016/j.neunet.2018.08.019}
}

@article{Abdeljawad2025UniformApproximationQuadratic,
  title = {Uniform Approximation with Quadratic Neural Networks},
  author = {Abdeljawad, Ahmed},
  date = {2025-12},
  journaltitle = {Neural Networks},
  shortjournal = {Neural Networks},
  volume = {192},
  pages = {107742},
  issn = {08936080},
  doi = {10.1016/j.neunet.2025.107742},
  url = {https://linkinghub.elsevier.com/retrieve/pii/S0893608025006227},
  urldate = {2025-11-25},
  langid = {english}
}

@article{Klusowski18ApproximationCombinationsReLU,
  title = {Approximation by {{Combinations}} of {{ReLU}} and {{Squared ReLU Ridge Functions}} with $\ell^1$ and $\ell^0$ {{Controls}}},
  author = {Klusowski, Jason M. and Barron, Andrew R.},
  date = {2018-12},
  journaltitle = {IEEE Transactions on Information Theory},
  shortjournal = {IEEE Trans. Inform. Theory},
  volume = {64},
  number = {12},
  pages = {7649--7656},
  issn = {0018-9448, 1557-9654},
  doi = {10.1109/TIT.2018.2874447},
  url = {https://ieeexplore.ieee.org/document/8485650/},
  urldate = {2023-06-14},
  langid = {english}
}

@article{Yang2025OptimalRatesApproximation,
  title = {Optimal {{Rates}} of {{Approximation}} by {{Shallow ReLU}}$^k$ {{Neural Networks}} and {{Applications}} to {{Nonparametric Regression}}},
  author = {Yang, Yunfei and Zhou, Ding-Xuan},
  date = {2025-10},
  journaltitle = {Constructive Approximation},
  shortjournal = {Constr Approx},
  volume = {62},
  number = {2},
  pages = {329--360},
  issn = {0176-4276, 1432-0940},
  doi = {10.1007/s00365-024-09679-z},
  url = {https://link.springer.com/10.1007/s00365-024-09679-z},
  urldate = {2025-11-30},
  langid = {english}
}

@article{DeVore2025WeightedVariationSpaces,
  title = {Weighted Variation Spaces and Approximation by Shallow {{ReLU}} Networks},
  author = {DeVore, Ronald and Nowak, Robert D. and Parhi, Rahul and Siegel, Jonathan W.},
  date = {2025-01},
  journaltitle = {Applied and Computational Harmonic Analysis},
  shortjournal = {Applied and Computational Harmonic Analysis},
  volume = {74},
  pages = {101713},
  issn = {10635203},
  doi = {10.1016/j.acha.2024.101713},
  url = {https://linkinghub.elsevier.com/retrieve/pii/S1063520324000903},
  urldate = {2025-11-30},
  langid = {english}
}

@article{Chaudhari2018DeepRelaxationPartial,
    title = {Deep relaxation: partial differential equations for optimizing deep neural networks},
    volume = {5},
    issn = {2522-0144, 2197-9847},
    shorttitle = {Deep relaxation},
    url = {http://link.springer.com/10.1007/s40687-018-0148-y},
    doi = {10.1007/s40687-018-0148-y},
    language = {en},
    number = {3},
    urldate = {2025-11-25},
    journal = {Research in the Mathematical Sciences},
    author = {Chaudhari, Pratik and Oberman, Adam and Osher, Stanley and Soatto, Stefano and Carlier, Guillaume},
    month = sep,
    year = {2018},
    pages = {30},
}

@article{Grohs23ProofTheorytoPracticeGap,
    title = {Proof of the {Theory}-to-{Practice} {Gap} in {Deep} {Learning} {Via} {Sampling} {Complexity} {Bounds} for {Neural} {Network} {Approximation} {Spaces}},
    issn = {1615-3375, 1615-3383},
    url = {https://link.springer.com/10.1007/s10208-023-09607-w},
    doi = {10.1007/s10208-023-09607-w},
    language = {english},
    urldate = {2023-11-17},
    journal = {Foundations of Computational Mathematics},
    author = {Grohs, Philipp and Voigtlaender, Felix},
    month = jul,
    year = {2023},
}

@article{Kutyniok2022TheoreticalAnalysisDeep,
    title = {A {Theoretical} {Analysis} of {Deep} {Neural} {Networks} and {Parametric} {PDEs}},
    volume = {55},
    issn = {0176-4276, 1432-0940},
    url = {https://link.springer.com/10.1007/s00365-021-09551-4},
    doi = {10.1007/s00365-021-09551-4},
    language = {en},
    number = {1},
    urldate = {2025-11-25},
    journal = {Constructive Approximation},
    author = {Kutyniok, Gitta and Petersen, Philipp and Raslan, Mones and Schneider, Reinhold},
    month = feb,
    year = {2022},
    pages = {73--125},
}

@article{Lagaris1998ArtificialNeuralNetworks,
    title = {Artificial neural networks for solving ordinary and partial differential equations},
    volume = {9},
    copyright = {https://ieeexplore.ieee.org/Xplorehelp/downloads/license-information/IEEE.html},
    issn = {10459227},
    url = {http://ieeexplore.ieee.org/document/712178/},
    doi = {10.1109/72.712178},
    number = {5},
    urldate = {2025-11-25},
    journal = {IEEE Transactions on Neural Networks},
    author = {Lagaris, I.E. and Likas, A. and Fotiadis, D.I.},
    month = sep,
    year = {1998},
    pages = {987--1000},
}

@article{Raissi2019PhysicsinformedNeuralNetworks,
    title = {Physics-informed neural networks: {A} deep learning framework for solving forward and inverse problems involving nonlinear partial differential equations},
    volume = {378},
    issn = {00219991},
    shorttitle = {Physics-informed neural networks},
    url = {https://linkinghub.elsevier.com/retrieve/pii/S0021999118307125},
    doi = {10.1016/j.jcp.2018.10.045},
    language = {en},
    urldate = {2025-11-25},
    journal = {Journal of Computational Physics},
    author = {Raissi, M. and Perdikaris, P. and Karniadakis, G.E.},
    month = feb,
    year = {2019},
    pages = {686--707},
}

@article{E2018DeepRitzMethod,
    title = {The {Deep} {Ritz} {Method}: {A} {Deep} {Learning}-{Based} {Numerical} {Algorithm} for {Solving} {Variational} {Problems}},
    volume = {6},
    issn = {2194-6701, 2194-671X},
    shorttitle = {The {Deep} {Ritz} {Method}},
    url = {http://link.springer.com/10.1007/s40304-018-0127-z},
    doi = {10.1007/s40304-018-0127-z},
    language = {en},
    number = {1},
    urldate = {2025-11-25},
    journal = {Communications in Mathematics and Statistics},
    author = {E, Weinan and Yu, Bing},
    month = mar,
    year = {2018},
    pages = {1--12},
}

@article{Caragea23NeuralNetworkApproximation,
    title = {Neural {Network} {Approximation} and {Estimation} of {Classifiers} with {Classification} {Boundary} in a {Barron} {Class}},
    volume = {33},
    issn = {1050-5164},
    url = {https://projecteuclid.org/journals/annals-of-applied-probability/volume-33/issue-4/Neural-network-approximation-and-estimation-of-classifiers-with-classification-boundary/10.1214/22-AAP1884.full},
    doi = {10.1214/22-AAP1884},
    language = {english},
    number = {4},
    urldate = {2023-12-05},
    journal = {The Annals of Applied Probability},
    author = {Caragea, Andrei and Petersen, Philipp and Voigtlaender, Felix},
    month = aug,
    year = {2023},
}

@article{Yarotsky2017ErrorBoundsApproximations,
    title = {Error {Bounds} for {Approximations} with {Deep} {ReLU} {Networks}},
    volume = {94},
    issn = {08936080},
    url = {https://linkinghub.elsevier.com/retrieve/pii/S0893608017301545},
    doi = {10.1016/j.neunet.2017.07.002},
    language = {english},
    urldate = {2023-05-16},
    journal = {Neural Networks},
    author = {Yarotsky, Dmitry},
    month = oct,
    year = {2017},
    pages = {103--114},
}

@article{FEICHTINGER1989ATOMICCHARACTERIZATIONSMODULATION,
    title = {{Atomic} {characterizations} {of} {modulation} {spaces} {through} {Gabor}-{type} {representations}},
    volume = {19},
    issn = {00357596, 19453795},
    url = {http://www.jstor.org/stable/44237199},
    number = {1},
    urldate = {2025-11-26},
    journal = {The Rocky Mountain Journal of Mathematics},
    author = {Feichtinger, Hans G.},
    year = {1989},
    note = {Publisher: Rocky Mountain Mathematics Consortium},
    pages = {113--125},
}

@article{Liao23SpectralBarronSpace,
    title = {Spectral {Barron} {Space} and {Deep} {Neural} {Network} {Approximation}},
    url = {https://arxiv.org/abs/2309.00788},
    doi = {10.48550/ARXIV.2309.00788},
    language = {english},
    urldate = {2024-02-14},
    author = {Liao, Yulei and Ming, Pingbing},
    year = {2023},
    note = {Publisher: arXiv
tex.version: 1},
}

\end{document}